\documentclass[11pt,letterpaper,onecolumn,oneside]{article}
\usepackage[english]{babel}
\usepackage{amsmath,amssymb,amsthm,amsfonts}
\usepackage[latin1]{inputenc}
\usepackage[T1]{fontenc}
\usepackage{anysize}
\usepackage{textpos}
\usepackage{bbm, dsfont}
\usepackage{enumerate}
\usepackage{mathrsfs}
\usepackage{txfonts}
\usepackage{hyperref}
\usepackage{bbm}

\newcommand{\Ex}[1]{\mathbb{E}_{\,\delta_x}\hspace{-0.05cm}\left( #1\right)}
\newcommand{\Px}[1]{\mathbb{P}_{\delta_x}\hspace{-0.05cm}\left( #1\right)}
\newcommand{\sbullet}{\raisebox{0.4mm}{\scalebox{0.6}{\textbullet}}}
\newcommand{\Pio}{\Pi^{\raisebox{0.4mm}{\scalebox{0.8}{$\perp$}}}}
\newcommand{\Cpio}{C_{\Pio}}
\newcommand{\Cpi}{C_{\Pi}}
\newcommand{\Oun}{\mathcal{O}(1)}
\newcommand{\pcun}{\mathbbm{1}}

\newcommand{\MB}{\mathscr{B}} 
\newcommand{\MBZ}{\mathscr{B}_{0}} 
\newcommand{\CB}{C_{b}(\RR)}
\newcommand{\CBZ}{C_{0}(\RR)}

\newcommand{\mybox}[2]{\makebox[#1cm]{#2}}
\newcommand\stackarrow[3]{%
\mathrel{\stackunder[2pt]{\stackon[-2pt]
{\xrightarrow{\mybox{#3}{}}}
{\hbox{\footnotesize $#1$}}}
{\hbox{\footnotesize $#2$}}}}

\newcommand{\unf}{u^{(f)}_{n}}
\newcommand{\un}{u_{n}}

\newcommand{\vsbullet}{\,\sbullet\,}
 \newcommand{\MBD}{\mathscr{B}_{\tau}}

\newcounter{nquestion}
\addtocounter{nquestion}{1}

\usepackage[normalem]{ulem}

\usepackage{bbm}
\usepackage[mathcal]{euscript}
\usepackage{graphicx}
\usepackage{pstricks}
\usepackage{pstricks-add}
\usepackage{pst-plot}
\usepackage{enumerate}
\usepackage{multicol}
\usepackage{subfigure}
\usepackage{mathrsfs}
\usepackage{amsmath,amssymb}
\usepackage{tikz}
\usepackage{stackengine}
\stackMath

\newcommand{\B}{\mathscr{B}}    

\newcommand{\CC}{\mathbb{C}}    

\newcommand{\e}{\varepsilon}   
\newcommand{\E}{\mathcal{E}}    
\newcommand{\EE}{\mathbb{E}}     
\newcommand{\F}{\mathcal{F}}    
 
\newcommand{\ind}{\mathbbm{1}} 

\newcommand{\NN}{\mathbb{N}}
\newcommand{\PP}{\mathbb{P}}    

\newcommand{\RR}{\mathbb{R}}    

\newcommand{\support}{\mathop{\rm supp}}

\numberwithin{equation}{section}
\theoremstyle{plain}
\newtheorem{theorem}{Theorem}[section]
\newtheorem{lemma}[theorem]{Lemma}
\newtheorem{proposition}[theorem]{Proposition}
\newtheorem{corollary}[theorem]{Corollary}
\newtheorem*{remark}{Remark}

\theoremstyle{definition}

\title{Long time behavior and Yaglom limit for real trait-structured Birth and Death Processes}

\author{Pierre Collet\hspace{-0.1cm}
\thanks{CNRS, Ecole polytechnique, Institut polytechnique de Paris, route de
Saclay, 91128 Palaiseau Cedex-France; E-mail:
pierre.collet@cpht.polytechnique.fr} ,
Sylvie M\'el\'eard\hspace{-0.1cm}
\thanks{CMAP, Ecole polytechnique, CNRS, Institut polytechnique de Paris, Inria, route de
Saclay, 91128 Palaiseau Cedex-France; E-mail: 
sylvie.meleard@polytechnique.edu} ,
Jaime San Mart\'in\hspace{-0.1cm}
\thanks{CMM-DIM;  Universidad de Chile; UMI-CNRS 2807; BASAL FB210005; Santiago; Chile. E-mail: jsanmart@dim.uchile.cl} .}
\begin{document}
\parindent=0pt
\maketitle

\begin{abstract}
In this article we study the long time behaviour of measure-valued birth and death processes in continuous time, where the 
dynamics between jumps are one-dimensional Markov processes including diffusion and jumps. We consider the three regimes, 
critical, subcritical and supercritical. Under suitable hypotheses on the Feynman-Kac semigroup, we prove a new  recurrence 
for the moments and the extinction probability, their time asymptotics and the convergence in law for the measure-valued 
birth and death process  conditioned to non extinction, leading to the existence of  $Q$-process and Yaglom limit (in this infinite 
dimensional setting). We develop three classes of natural examples where our results apply. 
\end{abstract}



\tableofcontents
\section{Introduction}

We consider a population dynamics where individuals are characterized by a trait
$x\in \mathbb{R}$ (or a position) and give birth and die in continuous time. More precisely we consider a
branching Markov process $(Z_{t})$ defined as follows: 
an individual with trait $x$ gives birth to a (unique) new individual with the same trait at rate $b(x)\ge 0$, 
and dies with rate $d(x)\ge 0$. 
During their life the individual
trait (or position) variations are modeled by a 
 Markov process  $(X_{t})$ with infinitesimal generator $\mathscr{G}$ accounting for infinitesimal  or macroscopic changes of 
the phenotypes (or of the positions) due for example to adaptation to a changing environment (see \cite{calvez22}, \cite{henry23}). 
We have in mind three classes of examples detailed in Section \ref{sec:example} : continuous diffusion process, 
continuous diffusion process with jumps and drifted jump process.  

\medskip
Denote by ${\cal V}_{t}$ the set of individuals alive at time $t$. Then, 
the process $Z=(Z_{t}, t\ge 0)$ is the point measure valued process on the trait space $\RR$, defined for any $t$ by
$$
Z_{t} = \sum_{i\in  {\cal V}_{t}} \delta_{X^i_{t}},
$$
where $X^i_{t}$ is the trait of individual $i$ at time $t$. We denote the total mass of the branching diffusion process by 
$$ 
N_{t}=\#{\cal V}_{t}=  \langle Z_{t} ,1\rangle.
$$
Recall that the measure valued process $(Z_t)$ has the following semimartingale decomposition:
for all $f\in Dom(\mathscr{G})$ (see for example \cite{champagnatmeleard2007})
\begin{equation}
\label{eq:Master General}
\begin{array}{ll}
\langle Z_t\, ,f\rangle&\hspace{-0.2cm}=f(x)+\int_0^t \langle Z_s,\mathscr{G} f\rangle ds+ M^{f}(t) \\
&\hspace{0.1cm}+\int_{[0,t]\times \RR_+\times \NN} \ind_{1\le j\le N_{s-}} \,f(X^j_{s-})
\left(\ind_{\theta \le b(X^j_{s-})}-\ind_{b(X^j_{s-})<\theta\le b(X^j_{s-})+d(X^j_{s-})}\right)
Q(ds,d\theta,dj),
\end{array}  
\end{equation}
where $Q$ is a Poisson point measure with intensity measure $ds \otimes d\theta \otimes n(dj)$
on $\RR_+\times \RR_+\times \NN$ ($n$ being the counting measure)  
and $M^{f}$ a local martingale. We also point out that for $f\equiv 1$ we have $\mathscr{G} 1=0$ and $M^1=0$.

The main objective of this article is to study the long time behavior for the distribution of 
$\, Z_t\, $ conditioning to non extinction  starting from a Dirac mass
and   the existence of quasi-stationary distributions on the space of finite measures.
This long time behavior depends on   the principal eigenvalue $\,- \lambda_0\,$ of $\,(P_t)$, the Feymnan-Kac 
semigroup associated to $Z$ (see for example \cite{calvez22}, \cite{henry23} and references therein). This semigroup is given for $x\in \RR$ by
\begin{equation}
\label{FK} 
P_{t}f(x) =  \Ex{\langle Z_{t},f \rangle} = \EE_{x}\left(\exp\left(\int_{0}^t V(X_{s})ds\right)\, f(X_{t})\right),
\end{equation}
 for any $f\in C_{b}(\RR)$, $t\ge 0$, $y\in \RR$,
where $$V(y)=b(y)-d(y)$$ is the growth rate of an individual with trait $y$.

    Our approach is
based on properties of this semigroup and on new recursive relations for
the moments and for the extinction probability (Proposition \ref{lem:for_moments}).
We consider  death rates that are not necessary bounded so we cannot use the standard moment recursion  (see for example \cite{harris2022}),  which involves 
explicitly this death rate. This new recursion is adapted from results by Z. Li (cf \cite{ZL}).
Our main hypotheses $(HP)$ concern  the Feynman Kac semigroup. 
We only require a standard spectral gap property for this semigroup, some regularity properties involving the space $\CBZ$, and that the particles do not escape  to infinity in finite time.  For each class of examples, we state the precise set of hypotheses under which Assumptions (HP) are satisfied.

We prove an equation for the uniform limit of 
the extinction probability when time goes to infinity (Corollary \ref{eqh} and Proposition \ref{prop:hcon}).
Then we study the asymptotic behavior of the process distinguishing three  different regimes: critical corresponding to $\lambda_0= 0$, subcritical 
corresponding to $\lambda_0>0$, supercritical corresponding to $\lambda_0<0$.
In all  these regimes we describe the precise time asymptotic behavior for the moments and for the extinction probability, 
the existence of the $Q$-process (Theorem \ref{the:Q-process}) and the convergence in law for the normalized  measure-valued 
branching processes and some kind of Yaglom limit.
 
In the critical case motivated by examples developed in
\cite{calvez22} and \cite{henry23},  we generalize in this infinite-dimensional setting well-known results for Galton-Watson processes
(\cite{Athreya-Ney} and references therein), that is, the extinction probability is of order $1/t$ and the conditional limiting distribution is 
exponential (Theorem \ref{the:critical convergence}). The techniques we developed are inspired by nonlinear dynamical systems theory. The
corresponding results for the subcritical case (Theorem \ref{the:subcritical}) are: the extinction probability is of order 
$e^{-\lambda_{0}t}$ and the conditional limiting distribution in characterized by its moments and do not depend on the initial 
condition. In the supercritical case (Theorem \ref{the:supercritical}), the extinction probability converges to a non trivial 
function and the  conditional limiting distribution also depends on the initial trait. Under extra 
mild conditions (satisfied by our examples) we prove the almost sure convergence of the rescaled measure-valued branching 
process to a  random measure whose mass is positive on the non-extinction event (Theorem \ref{the:convergence a.s.}).

Moment asymptotics  and limit theorems for time continuous  branching-Markov processes were studied in 
particular in \cite{bansaye2013}, \cite{harris2022}, see also \cite{harris2020}, \cite{gonzalez22}  
and \cite{horton20}. In all these papers, the branching rate is assumed to be bounded and appears 
explicitely in the recursive moment equations. They also assume a ratio-Perron-Frobenius theorem  
(Hypothesis (G2) in \cite{harris2022} where the inverse of the dominant eigenfunction appears) and 
in some sense that the birth rate is bounded below (Hypothesis (G5) in \cite{harris2022}).   
Our results cover examples where these hypotheses are not satisfied, namely a 
dominant eigenfunction vanishing at infinity or a birth rate with compact support.


 \medskip
{\bf Notation} :  $\MB$ will denote the set of bounded measurable
functions on $\RR$ and $\MBZ$ the set of bounded measurable functions, which
converge to $0$ at infinity. $\CB$ denotes the set of bounded continuous functions and $\CBZ$ 
the set of continuous functions tending to $0$ at infinity. These spaces are Banach spaces for the supnorm denoted by $\|\vsbullet\|_{\infty}$.
We also denote by $C_c(\RR)$ the set of continuous functions with compact support and $C^\infty_c(\RR)$,
the subset of $C_c(\RR)$, of functions which are infinitely differentiable. We notice that $C^\infty_c(\RR)$
is dense in $\CBZ$. On the other hand, we denote by
$M_{F}(\RR)$ the set of finite measures on $\RR$ endowed with the topology of weak convergence an its
associated Borel $\sigma$-field.

\subsection{Hypotheses}
\label{sec:hypotheses}
We summarize in this subsection the main hypotheses of this article.

\medskip We make the following hypotheses  for the birth and death rates.

\noindent {\bf \large Assumptions} (HV):
\vspace{ -0.5cm}
\begin{enumerate}[HV1)]
\item $b$ and $d$ are nonnegative nonzero and continuous, hence $V=b-d$ is continuous.
\item $b$ is bounded from above  by a finite number $b^*>0$, which implies $A(V):=\sup\limits_{z\in \RR} V(z)<+\infty$.
\end{enumerate}
\medskip

We also make the following hypotheses on the semigroup $(P_t)$.

\noindent {\bf \large Assumptions} (HP):

\noindent HP1) There exists a  positive bounded continuous function $\Theta_{0}\in \CBZ$ and a
  real number $\lambda_{0}$ such that for any $t\ge0$
$$
P_{t}\Theta_{0}=e^{-\lambda_{0}\,t}\Theta_{0}\;.
$$

\noindent HP2) There exists a positive finite measure $\mu_{0}$  charging
every nonempty open set, and two numbers
  $\lambda_{1}>\lambda_{0}$ and $H>0$ such that   $\int \Theta_{0}\,d\mu_{0} = 1$  and the one
dimensional projection
$\Pi$ on the subspace $\RR \Theta_{0}$ in $\CB$ given by
$$
\Pi(g)=\Theta_{0}\;\int g(x)\;\mu_{0}(dx)
$$
satisfies  for all $g\in C_{b}$ and all
$t\ge 0$
\begin{equation}
\label{eq:bound_uniforme}
\|e^{\lambda_0 t}P_t(g)-\Pi(g)\|_\infty\le H \|g\|_\infty e^{
-(\lambda_1-\lambda_0)t}\;.
\end{equation}
 
\noindent HP3) The semi-group $(P_{t})_{t\ge 0}$  is strongly continuous and irreducible in $\CBZ$ 
(in the sense of Arendt et al \cite{AGGGLMNFS} p.  182) and $\int_{0}^ tP_{s}\,ds$ maps $\MB$ to $\CB$, 
in the sense that for any $f\in \MB$, the function
$$
x \mapsto \int_{0}^tP_{s}f(x)ds = \int_{0}^t  \EE_{x}\left(\exp\left(\int_{0}^s V(X_{r})dr\right)\, f(X_{s})\right) ds
$$
belongs to $\CB$.\\

\noindent HP4) For any $t>0$, $P_{t}$ maps $\CB$ to $\CBZ$.

\noindent HP5) This assumption concerns the extreme values of the individuals traits.  For any $m>0$, we define
 $$T_{m} = \inf\{ s\ge 0, \langle Z_{s},\ind_{[-m,m]^c} \rangle>0\}.$$
 We assume that for any $x\in \RR$, $T_{m}$ tends to $+\infty$ $\mathbb{P}_{\delta_{x}}$-a.s.

\begin{remark}
(i) Note that if the semigroup $P_{t}$ is compact and satisfies HP3), then it also satisfies 
HP1) and HP2) (see \cite{AGGGLMNFS} p.210).

(ii)  In \cite{AGGGLMNFS}, $- \lambda_{0}$ is called the spectral bound which, in our setting,  is equal to the growth bound of the semigroup (i.e.  the growth rate of $~\|P_{t}\|_{\infty}$), see \cite{AGGGLMNFS} Theorem 1.4 p. 207.

(iii) HP1) and HP2) imply that, for any $t>0$, $e^{-\lambda_{0}t}$ is a simple isolated eigenvalue of $P_{t}$ in $\CB$. Then  $e^{-\lambda_{0}t}$ is also  a simple isolated eigenvalue of the adjoint $P'_{t}$ in $M_{F}(\RR)$ associated to the eigenvector $\mu_{0}$.

(iv) HP4) implies that the potential $V$ is not bounded below. 
\end{remark}

In the sequel, we will also use the semigroup $(Q_{t})$ defined as the perturbation of $(P_{t})$ by the operator of pointwise
multiplication by the function $-2b$, namely 
\begin{equation}
\label{def:Qt}
 Q_{t}f(x) = \EE_x\left(e^{- \int_0^t (b+d)(X_r) \ dr}f(X_{t})\right).
\end{equation}
We will need an additional assumption in the supercritical case  ($\lambda_{0}<0$).

\noindent {\bf \large Assumption} (HQ)
\vspace{-0.25cm}

 If $\lambda_{0}<0$, either the function $b+d$ does not vanish or the semigroup $(Q_{t})$ is quasicompact in $\CBZ$ (see for example Arendt et al \cite{AGGGLMNFS} p. 214).

\medskip This article is organized as follows. In Section \ref{sec:example}, we give three classes of examples for which  our general assumptions are satisfied, hence our results apply. The three examples are respectively pure diffusion process, diffusion with jumps process and drifted pure jump process.  In Section \ref{sec:results}, we
state the main results in the three cases (critical, subcritical and supercritical): estimates for the
normalized moments, convergence of theses moments, convergence of the
normalized process conditioned on non-extinction and $Q$-process. In
Section \ref{sec:moments}, we establish the recursive relations for
the time evolution of the moments of $\langle Z_{\sbullet},f\rangle$
and for the survival probability $\Px{N_t >0}$.  Sections
\ref{sec:proof_critical}, \ref{sec:proof_subcritical}  and
\ref{sec:proof_supercritical} are devoted to the proofs  of the main
results in the critical, subcritical, supercritical cases,
respectively.  In Section \ref{sec:examples} we check that the general hypotheses hold in the different classes of examples.

\section{ Different classes of examples for our
  results}\label{sec:example} 

We present three classes of examples  of underlying Markov processes for which (HP) and (HQ)  are satisfied. 
The first example deals with diffusion processes driven by a Brownian motion. The proof is mostly based on 
the $L^2$-spectral results developed in \cite{CMSM} but necessitate to be transferred to the space 
$C_{b}$. A key point is to remark that the semigroup $P_{t}$ maps $C_{b}$ in $C_{0}$. 
The second example concerns diffusions with bounded jumps. The proof relies on the results for the 
previous case and the fact that the resolvent of the semigroup is compact. The third example deals 
with a class of drifted pure jump processes introduced in \cite{cloezgabriel}.

In all the following classes of examples we will assume an additional property on the potential $V$ : 

\noindent {\bf \large Assumption} (HD):

(HD) - $d$ has at least a linear growth at $\pm \infty$, then $\lim\limits_{|x|\to \infty} d(x)=+\infty$.

We point out that  $HV2)$ and $HD)$ imply that  there exist constants $E>0, x_0\ge 0$ such that 
for all $|x|\ge x_0$
\begin{equation}
\label{H4}
V(x)\le -E |x|.  
\end{equation}

\subsection{Continuous diffusion processes}
\label{21}
The underlying diffusion process is  given by 
\begin{equation}
\label{diffusion}
dX_{t}= dB_{t} - a(X_{t})\;dt.
\end{equation}

\noindent In this case, the semigroup $(P_{t})$ defined in \eqref{FK}, with this diffusion process,  is  denoted by 
$(P_{t}^{(1)})$ whose generator is given by 
\begin{equation}\label{gendiff}
\mathscr{L}^{(1)} f=\frac12 f''-af'+Vf.
\end{equation}

\noindent
In this section we derive the hypotheses (HP) from hypotheses on the
diffusion process $(X_{t})$. 

We introduce the following assumptions on the drift term $a$ and on the function.
\begin{equation}
\label{def:fonctionl}\ell(x) = \int_{0}^x a(z)dz.
\end{equation}

\noindent {\bf \large Assumptions} (HA) 
\vspace{-0.5cm}
 \begin{enumerate}[{HA}1)]
\item $a\in C^1(\RR)$ and it has at most a linear growth, that is, there exists a constant $C$ such that, for all 
$x\in \RR$
$$
|a(x)|\le C(|x|+1).
$$
\item $|\ell|$ has at most a linear growth,  namely there exists constants $\beta,\gamma$ so that for all $x$
it holds $|\ell(x)|\le \gamma + \beta|x|$.
\item the function $a'-a^2$ is bounded from above.
\end{enumerate}

Note that $a\in C^1(\RR)$, for which $a,a'$ are bounded, satisfies these assumptions.

\begin{remark}
We require $\ell$ to have linear growth to control the eigenfunctions of $\mathcal{L}^{(1)}$. 
It is possible to relax this assumption, at the price of strengthening \eqref{H4}:
$V(x)\le -E (|\ell(x)|+|x|)$, for all large $|x|$.
$(HA3)$ is used to perform a Girsanov transformation relating the spectral properties of $\mathscr{L}^{(1)}$ to
the well-known spectral properties of the operator $\frac12 u''+\tilde V u$, where $\tilde V=V+\frac12 (a'-a^2)$.
\end{remark}

\begin{remark} Notice that our model can be generalized to an underlying  diffusion process
$$dY_{t}= \sigma(Y_{t})dB_{t} - a_{0}(Y_{t})dt,
$$
 with a $C^1$ diffusion coefficient $ \sigma$ bounded below away from zero, 
  by the image-measure transformation on the trait space
(leading to a change of measure on the measure space), with the (classical)
diffeomorphism  $y\to  G(y)=\int_{0}^y \frac{1}{\sigma(u)}du$ .
In that case, the jump rates $b$ and $d$ should be replaced by $b\circ G$ and $d\circ G$. 
\end{remark}


We will use the spectral results in \cite{CMSM} for the long time behavior of the Feynman-Kac semigroup. They involve the so called speed measure $\rho$  which makes $(P_t)$ a symmetric semigroup in $L^2(d\rho)$ and is given by
\begin{equation}
\label{lerho}
\rho(dy)=e^{-2\ell(y)}\, dy,
\end{equation}
where 
$\ell$ has been defined in \eqref{def:fonctionl}.

The following result summarizes the asymptotic properties shown in \cite{CMSM} that we will 
need in this article.

\begin{theorem}
\label{the:A}
Assume (HV), (HD) and (HA), then hypotheses (HP) and HQ) are
satisfied with 
$\mu_{0}(dy)=\Theta_{0}(y)\rho(dy) $.
\end{theorem}

The proof is split in  Proposition  \ref{HP3-diff}, Proposition \ref{HP1-HP2-diff} and Proposition \ref{HP3-HP4-diff}.



\subsection{Continuous diffusion processes with jumps}
\label{22}

In this case the underlying Markov process is a diffusion process with jumps
whose generator is given by 
$$\frac12 f''-af'+ L_{1}f,
$$ where
$$
L_{1}f(y) = \int_{\RR} \big(f(z)-f(y)\big) \;R(y, dz)\;.
$$

We  assume as above Assumptions (HA)  and introduce  on the jump measure  the following

{\bf \large Assumptions} (HJ)
\vspace{-0.5cm} \begin{enumerate}[HJ1)]
\item The total mass and forth order moment of the measure $R(y,.)$  are uniformly bounded :
$$
\sup_{y} \int (1+ |z-y|^4)\; R(y, dz) < +\infty\;.
$$
\item The function $y\to R(y, .)$ is weakly continuous 

\end{enumerate}
This implies that the function $
y\to \int R(y, dz)$ belongs to $\CB$.

Assumptions HA1) and HJ1) ensure the existence and uniqueness of the underlying Markov process $(X_{t})$. The semigroup $(P_{t})$ defined in \eqref{FK} with this Markov process  is denoted by $(P_{t}^{(2)})$ whose generator is given by 
\begin{equation}\label{gendiffsauts}
\mathscr{L}^{(2)} f= \mathscr{L}^{(1)} f + L_{1}f     =\frac12 f''-af'+Vf + L_{1}f.
\end{equation}

\begin{theorem}
\label{the:B}
Assume (HV), (HD),  (HA) and  (HJ). 
Then hypotheses (HP) and HQ) are satisfied. 
\end{theorem}

The proof is split in  Proposition  \ref{HP3-diff}, Propositions \ref{HP1-HP2-diff-sauts-comp-res}
and \ref{HP3-diff_sauts}.

\subsection{Drifted pure jump process}
\label{23}
This case has been inspired by the spectral gap property proved by  Cloez and Gabriel in \cite{cloezgabriel} in a case of  jump processes
with constant drift. The semigroup $(P_{t})$  denoted by $(P_{t}^{(3)})$, whose generator is given by  
$$
\mathscr{L}^{(3)} f(y) = f'(y) + L_{1}f(y) + V(y)f(y) = f'(y) + \int_{\RR} \big(f(z)-f(y)\big) \;R(y, dz) + V(y)f(y) \;,
$$
with $R$ satisfying (HJ) and the additional assumption

HJ4) : 
$$\exists\, \varepsilon, k_{0}>0, \  \forall x\in\RR, \ R(x,dz) \geq k_{0} 1_{(x-\varepsilon, x+\varepsilon)}(z)dz.$$

\begin{remark} Notice that our model can be generalized to a non constant non vanishing drift $a(x)$ satisfying (HA1) 
  by using a diffeomorphism conjugating to the case of constant drift. 
\end{remark}

\begin{theorem}
\label{the:C}
Assume  (HV), (HD) and (HJ), HJ4). 
Then hypotheses (HP)  and HQ) are satisfied.
\end{theorem}

The proof is split in  Proposition  \ref{HP3-diff}, the results of \cite{cloezgabriel} and Section \ref{HP3-sauts}.

\section{Main results}
\label{sec:results}

In the main theorems, we will always assume the assumptions (HV) and (HP). 

\subsection{Direct consequences of the Assumptions}
In what follows, we denote by $\,\nu\,$ the probability measure given by 
$$
\nu=\frac{ \mu_0}{\int \mu_0(dy)}.
$$
We will use equivalently the notations $\int f(y)\nu(dy)$ and $\nu(f)$. 

The next theorem  summarizes immediate consequences of Hypotheses HP1) and HP2), using the Feynman-Kac formula \eqref{FK} . 
\begin{theorem}
\label{the:corA}
\begin{enumerate}
\item $\eta=-\lambda_0$ is the exponential growth rate of $\,\Ex{N_t}$, that is
$$
\eta=\lim\limits_{t\to \infty} \frac{\log\left(\Ex{N_t}\right)}{t}=-\lambda_0.
$$
\item For all $\phi\in \CB$ and for any $x\in \RR$, we have
$$
\lim\limits_{t\to \infty} e^{\lambda_0 t}\, \Ex{\langle Z_t,\phi\rangle}=
\lim\limits_{t\to \infty} e^{\lambda_0 t}\, \EE_x\left(e^{\int_0^t V(X_s)} \phi(X_t)\right)=
\Theta_0(x) \int \phi(z)\,\mu_0(dz).
$$
In particular, we have the following ratio limit: for all $\phi$ bounded measurable,$$
\lim\limits_{t\to \infty} \frac{ \Ex{\langle Z_t,\phi\rangle}}{\Ex{N_t}}
=\nu(\phi).
$$
\end{enumerate}
\end{theorem}


\subsection{General properties of the survival probability}

For any $t\ge 0$ and $x\in \RR$, we denote
by $u_{0}(t,x) $ the survival probability of $Z_{t}$ issued from one individual with trait $x$, namely
$$ u_{0}(t,x) = \PP_{\delta_{x}}(N_{t}>0).$$

\begin{theorem} Under (HV) and (HP),

(i) For any $t>0$, $u_{0}(t,\vsbullet) \in \CBZ$.

(ii) The function $u_{0}(t,\vsbullet)$ uniformly converges (when $t$ tends to infinity),  to a function  $h\in \CBZ$ which satisfies 
$$
h(x)=\int_{0}^{\infty}\EE_{x}\left(e^{- \int_{0}^{s} (b+d)(X_{\tau})\,d\tau}\big(
2\, b(X_{s})\,h(X_{s})-b(X_{s})\,h^2(X_{s})\big)\;\right)\;ds\;.
$$
\end{theorem}
The proof of this theorem is given in Section \ref{sec:survival}.

\begin{remark} We notice that $h$ satisfies also the equation: for all $\e>0$
\begin{equation}
\label{eq:representation K^+ epsilon}
h(x)=\int_0^\infty e^{-\e s}\, \EE_x\left( e^{-\int_0^s (b+d)(X_u)\, du} \, 
\left((2b(X_s)+\e)\,h(X_s)-b(X_s)(h(X_s))^2\right)\right) ds.
\end{equation} 
This representation can be obtained by applying Proposition 2.9 in \cite{ZL} p.34.
\end{remark}

\medskip
The following results  distinguish the three regimes (critical, sub-critical, supercritical), but our 
main interest is the critical case whose long term behavior is more subtle. Therefore, 
we begin the list of our results by this case.

\subsection{Critical case $\eta=-\lambda_0=0$}

The proofs of the following results are given  in Section \ref{sec:proof_critical}.

We denote by $A$ and $B$ the  two main  constants:
\begin{equation}
\label{eq:AB}
A=   \int \mu_0(dy )
\ \ ;\ \ 
B=  \int \Theta_{0}^2(y) \, b(y)\,   \mu_{0}(dy).
\end{equation}

In the following,  $\,\xi\,$ will denote   an exponential random variable with mean 1, whose distribution is denoted by $\E(1)$.

The next proposition gives the crucial estimates on the moments of $\langle Z_t\,, f\rangle$ for a measurable bounded function $f$. Recall that $\nu=\mu_0/\mu_0(1)$

\begin{proposition}
\label{pro:critical estimation}
Let $f \in \CB$ and assume that $n\ge 1$. There exist constants $0<D<\infty$ independent of $n$
and $0<E_n^f<\infty$ so that, for all $t,x$,
\begin{align}
&|\Ex{\langle Z_t\, ,f\rangle^n}| \le \|f\|_\infty^n\, \Ex{\langle Z_t\, ,1\rangle^n} \le \|f\|_\infty^n\, n!\, D^{2n-1}\,(t+1)^{n-1};
\label{eq:CRITICAL bounds moments I} \\
\nonumber\\
&\left|\frac{\Ex{\langle Z_t\, ,f\rangle^n}}{(t+1)^{n-1}}-\Theta_0(x)\,n! \,\nu(f)^n\, A^n\, B^{n-1}\right|\le E_n^f \,\frac{1}{t+1}.
\label{eq:CRITICAL bounds moments II}
\end{align}

In particular for $n\ge 1$ we have 
\begin{equation}
\label{eq:convergence moments critical}
\lim\limits_{t\to \infty} \frac{B(t+1)}{\Theta_0(x)}\, 
\Ex{\left(\frac{\langle Z_t\, ,f\rangle}{(t+1)AB}\right)^n}
=\left(\nu(f)\right)^n \,n!=\left(\nu(f)\right)^n \EE(\xi^n).
\end{equation}
\end{proposition}

Proposition \ref{galere} proves the uniform convergence to $0$ of the survival probability.

\begin{proposition}
\label{galere}
The survival probability $\Px{N_t>0}$ decreases to $h\equiv 0$ uniformly in $x$ when $t$ tends to infinity.
\end{proposition}
Theorem \ref{the:critical} proves the  decrease of the survival probability at rate $1/t$ from which  we deduce the convergence of normalized conditional moments. 

\begin{theorem} 
\label{the:critical}
 Assume (HV) and (HP). 
The following asymptotics for the survival probability holds
\begin{equation}
\label{eq:asym survival critical}    
\lim\limits_{t\to \infty} (t+1) \Px{N_t>0}=\frac{1}{B} \Theta_0(x).
\end{equation}
Moreover, we have for all $t\ge 0$
\begin{equation}
\label{eq:Uniform survival prob}
\sup\limits_{x\in \RR}\left|(1+t)\,  \Px{N_t>0}-
\frac{1}{B}\Theta_{0}(x)\right|\le\frac{\Oun\log(2+t)}{1+t}\;. 
\end{equation}
This together with \eqref{eq:convergence moments critical} gives the following conditional asymptotics for the
moments of  $\frac{\langle Z_t\, ,f\rangle}{(t+1)AB}$, for $f\in \CB$,
\begin{equation}
\label{eq:asym normalized moments critical}    
\lim\limits_{t\to \infty} 
\Ex{\left(\frac{\langle Z_t\, ,f\rangle}{(t+1)AB}\right)^n \, \Big | \, N_t>0}
=\left(\nu(f)\right)^n \EE(\xi^n).
\end{equation}
\end{theorem}

Theorem \ref{the:critical convergence} provides the convergence in law for the joint normalized random variables
$\langle Z_t\,,f\rangle$ and $\langle Z_t\,,1\rangle$, under the conditioning on survival for large times, a law of large numbers and a Central like limit Theorem.

\begin{theorem}\label{the:critical convergence} 
Assume (HV) and (HP). Let $f\in \CB$.
\begin{enumerate}[(1)]
\item 

The joint asymptotic conditional distribution of $\frac{1}{(t+1)AB}(\langle Z_t\, ,f\rangle, \langle Z_t\, ,1\rangle)$ is given by
the two dimensional distribution of $\, (\xi \nu(f),\xi)$,  namely for all $F:\RR^2\to \RR$, 
bounded and continuous it holds
$$
\lim\limits_{t\to \infty}\Ex{\left.F\left(\frac{1}{(t+1) AB}
\Big(\langle Z_t\, ,f\rangle,\langle Z_t\, ,1\rangle\Big)\right) \,\, \right| \, N_t>0}
=\EE\Big(F( \xi\nu(f),\xi)\Big).
$$
In particular, the  random measures  $\frac{1}{(t+1)AB} Z_t$ 
(taking values on the space 
$M_{F}(\RR)$) conditioned to have no extinction before $t$,  converge
in law (as $t$ tends to infinity),  to the measure $\Lambda$ on $M_F(\RR)$ given by
$$
\Lambda(d\mu)= \int_{0}^\infty e^{-z}\,  \delta_{z\,\nu}(d\mu)\, dz.
$$
This means that $\Lambda$ is a Yaglom limit.

\item Law of Large Numbers holds. The distribution of the ratio 
$$
\frac{\langle Z_t\, ,f\rangle}{\langle Z_t\, ,1\rangle}
$$
conditioned on survival at time $t$, converges to   
the Dirac mass at $\nu(f)$.

\item Assume that $\nu(f)\neq 0$, then the following normalized ratio
$$
\Upsilon_t:=\frac{1}{\nu(f)(AB)^{1/2}} \frac{\frac{\langle Z_t\, ,f\rangle}{t+1}-\nu(f) A B/2}{\sqrt{\frac{N_t}{t+1}}}
$$
conditioned on survival at time $t$, converges conditionally in distribution to 
$$
\xi^{1/2}-\frac12\xi^{-1/2}.
$$
\end{enumerate}
\end{theorem}

\begin{remark} We point out that if we consider a finite family 
$\left(\langle Z_t\,, f_1\rangle,...,\langle Z_t\,, f_k\rangle\right)$, 
the joint (normalized) conditional distribution converges to $ \, \left(\nu(f_1),...,\nu(f_k)\right)\xi$.
\end{remark}

\begin{remark}
Part $(2)$ Theorem \ref{the:critical convergence} can also be interpreted as a weak law of large numbers.
Indeed, we have
$$
\Ex{\left.\frac{\langle Z_t\, ,f\rangle}{t+1}\,\right | N_t>0}\to B  \int  f(y) \mu_0(dy )
=AB  \int  f(y) \nu(dy ).
$$
\end{remark}

\begin{remark}Part $(3)$ Theorem \ref{the:critical convergence} can be seen as a non-centered 
non-Gaussian limit Theorem since the asymptotic limit of the expectation of 
$\frac{\langle Z_t\, ,f\rangle}{t+1}$ is $AB \nu(f)$. The extra $1/2$ on the mean, in the numerator, is used
to get a centered limit. In fact, an easy computation using integration by parts shows that the random
variable $\ \xi^{1/2}-\frac12\xi^{-1/2}\,$ is centered. This bias compensates for the fact that we
use $N_t$ as a normalization. Note that since $\langle Z_t\, ,f\rangle\le
\|f\|_\infty N_t$, $N_t$ can be large while $\langle Z_t\, ,f\rangle$ can be small. \end{remark}

\begin{remark}
The distribution of $\xi^{1/2}-\frac12\xi^{-1/2}$ can be calculated
as follows. Let $y\in \RR$, condition
$\xi^{1/2}-\frac12\xi^{-1/2}\le y$ is equivalent to 
$$
\xi-y \xi^{1/2} -1/2\le 0,
$$
which is equivalent (since $\xi\ge 0$) to
$$
\xi^{1/2}\le \frac{ y+\sqrt{y^2+2}}{2} \Leftrightarrow
\xi\le \left(\frac{y+\sqrt{y^2+2}}{2}\right)^2=h(y).
$$
Notice that for large positive $y$ 
$$
\PP(\xi^{1/2}-\xi^{-1/2}/2>y)=e^{-h(y)} = e^{-y^2+O(1)},
$$
while for large negative $y$
$$
\PP(\xi^{1/2}-\xi^{-1/2}/2<y)=1- e^{-h(y)}=\frac{1}{4\,y^2+O(1)}.
$$
The asymptotic distribution is quite non symmetric.
\end{remark}

\subsection{Subcritical case $\eta=-\lambda_0<0$}

The proofs of the results stated in this section are given in Section 
\ref{sec:proof_subcritical}.

The next proposition gives the basic estimates on the moments of $\langle Z_t\, , f\rangle$ 
needed for the main theorem of this section.
\begin{proposition}
\label{pro:subcritical estimation} 
There exists a constant $\mathcal{D}$, so that,
for all $n\ge 1$ and $f\in \CB$, 
$$
\begin{array}{l}
C^-_n=\sup\limits_{t,x} e^{\lambda_0 t}\,\Ex{\langle Z_t\, ,1\rangle^n}<+\infty, \quad
C^-_n(f)=\sup\limits_{t,x} e^{\lambda_0 t}\,\Ex{\langle Z_t\, ,f\rangle^n}\le C^-_n \|f\|^n_\infty\,,\\
C^-_n\le C^-_1+ 
C^-_1 b^*\lambda_0^{-1}\sum\limits_{k=1}^{n-1} \binom{n}{k}  C^-_{k}C^-_{n-k}\le (n!)^{1+\e}\,\mathcal{D}^{3n-2}.
\end{array}
$$
\end{proposition}
The proof of this Proposition can be found in Section \ref{sec:moments_subcritical}.

We introduce the normalized moments and survival probability by
$$
\begin{array}{l}
v_n^-(f)(t,x)=e^{\lambda_0 t}\,\Ex{\langle Z_t\, ,f\rangle^n}:\, n\ge 1,\hbox{ and }
v_0^-(t,x)=e^{\lambda_0 t}\,\Px{N_t>0},
\end{array}
$$
We also denote by $v_n^-(t,x)=e^{\lambda_0 t}\Ex{N_t^n}$.

\begin{theorem} 
\label{the:subcritical}
Assume (HV) and (HP) and that the process  $(Z_t)$ is subcritical ($\lambda_{0}>0$).
For any $f\in \CB$, 
\begin{enumerate}[(1)]
\item {\bf Asymptotic for the moments.} The following limits hold for the normalized moments: 
$$
\begin{array}{l}
V^-_1(f)=\frac{1}{\Theta_0(x)}\lim\limits_{t\to \infty} e^{\lambda_0 t}\,\Ex{\langle Z_t\, ,f\rangle}
=\int f(y) \mu_0(dy),\hbox{ and for $n\ge 2$}\,,\\
V^-_n(f)=\frac{1}{\Theta_0(x)}\lim\limits_{t\to \infty} e^{\lambda_0 t}\,\Ex{\langle Z_t\, ,f\rangle^n}
=V^-_1(f^n)+\sum\limits_{k=1}^{n-1} \binom{n}{k} 
\int_0^\infty e^{-\lambda_0 r} \int v_{n-k}^-(f)(r,y)\,v_{k}^-(f)(r,y)\, b(y)\, \mu_0(dy)\, dr.
\end{array}
$$
We also have the following bound on the rate of convergence
$$
F^-_n(f)=\sup\limits_{t,x} e^{\,\beta_n t} \, |e^{\lambda_0 t}\Ex{\langle Z_t\, ,f\rangle^n}-\Theta_0(x) V^-_n(f)|<+\infty,
$$
with $\beta_1=\lambda_1-\lambda_0>0$ and $\beta_n=\frac{\lambda_0}{\lambda_1}(\lambda_1-\lambda_0)>0$, 
for all $n\ge 2$.

\item {\bf The survival probability.} The following asymptotic for the survival probability holds
\begin{equation} 
\label{eq:Uniform survival prob subcritical}
\sup\limits_{t,x} \frac{e^{\beta t}}{1+t\, \ind_{\lambda_1=2\lambda_0}}
\left|e^{\lambda_0 t}\Px{N_t> 0}-\Theta_0(x) \mathfrak{K}^-\right|<+\infty,
\end{equation}
where $\beta=\lambda_0 \wedge (\lambda_1-\lambda_0)>0$ and 
\begin{equation}
\label{kappamoins}
\mathfrak{K}^-=\int \mu_0(dy)-\int_0^\infty e^{-\lambda_0 r} \int (v_0^-(r,y))^2 b(y) 
\mu_0(dy)\, dr>0.
\end{equation}

\item {\bf Convergence in law: Yaglom Limit.} The distribution of $\langle Z_t\, ,f\rangle$ conditioned on
survival at time $t$  
defined on the Borel sets (of $M_F(\RR)$) as:
$$
\mu_t(f,x)(A)=\Px{\langle Z_t\, ,f\rangle \in A \,\, | \, N_t>0}
$$
converges to a probability measure $\zeta^-_f(dy)$, which is independent of $x$, defined  as the unique solution 
of the Hamburger moment problem : for all $n\ge 1$
$$
\int_\RR y^n \, \zeta^-_f(dy)=\frac{V^-_n(f)}{\mathfrak{K}^-}.
$$

The distribution of the random measure $Z_t$, conditioned to non extinction,  
converges to a unique probability measure $\Lambda^-$ defined in $M_F(\RR)$. 
This unique probability measure does not depend on the initial condition $\delta_x$
and satisfies
$$
\int_{M_{F(\RR)}} e^{-\langle \mu\,,f\rangle} \Lambda^-(d\mu)= \int_\RR e^{-y} \zeta^-_f(dy). 
$$
Then, $\Lambda^-$ is a Yaglom limit for $(Z_t)_t$.
\end{enumerate}
\end{theorem}

\subsection{Supercritical case $\eta=-\lambda_0>0$}
\label{sec:Theorem supercritical}

The proofs of the results stated in this section are given in 
Section \ref{sec:proof_supercritical}.

The next proposition gives the basic estimates on the moments of $\langle Z_t\, , f\rangle$ 
needed for the main theorem of this section.
\begin{proposition}
\label{pro:supercritical estimation} 
 Assume (HV), (HP), then there exists a positive constant $\mathcal{D}$, so that,
for all $n\ge 1$ and $f\in \CB$, 
$$
\begin{array}{l}
C^+_n=\sup\limits_{t,x} e^{n\lambda_0 t}\,\Ex{\langle Z_t\, ,1\rangle^n}<+\infty,
\quad C^+_n(f)=\sup\limits_{t,x} e^{n\lambda_0 t}\,\Ex{\langle Z_t\, ,f\rangle^n}\le C^+_n \|f\|_\infty^n\,,\\
C^+_n\le C^+_1+ \frac{C^+_1 b^*|\lambda_0|^{-1}}{n-1}
\sum\limits_{k=1}^{n-1} \binom{n}{k}  C^+_{k}C^+_{n-k}\le n! \mathcal{D}^{2n-1}.
\end{array}
$$
\end{proposition}
The proof of this Proposition can be found in Section \ref{sec:moments_supercritical}.

We consider the normalized moments and survival probability:
$$
\begin{array}{l}
v_n^+(f)(t,x)=e^{n\lambda_0 t}\,\Ex{\langle Z_t\, ,f\rangle^n}:\, n\ge 1, \hbox{ and }
v_0^+(t,x)=\Px{N_t>0}.
\end{array}
$$
We also denote by $v_n^+(t,x)=e^{n\lambda_0 t}\, \Ex{N_t^n}$.

\begin{theorem} 
\label{the:supercritical}
Assume (HV), (HP), assume that  the process $(Z_t)$ is supercritical 
($\lambda_{0}<0$) and (HQ), then for $f\in \CB$, 
\vspace{-0.3cm}
\begin{enumerate}[(1)]
\item {\bf Asymptotic for the moments.} The following limits hold for the moments for all $x\in \RR$
\begin{equation}
\label{eq:moments supercritical}
\begin{array}{l}
V^+_1(f,x)=\lim\limits_{t\to \infty} e^{\lambda_0 t}\,\Ex{\langle Z_t\, ,f\rangle}
=\Theta_0(x) \int f(y) \mu_0(dy), \hbox{ and for $n\ge 2$}\,,\\
V^+_n(f,x)=\lim\limits_{t\to \infty} e^{n\lambda_0 t}\,\Ex{\langle Z_t\, ,f\rangle^n}
=\int_0^\infty e^{n\lambda_0 s} P_s\left(b(\sbullet)
\sum\limits_{k=1}^{n-1} \binom{n}{k} V^+_{n-k}(f,\sbullet) V^+_k(f,\sbullet)\right)(x)\, ds.
\end{array}
\end{equation}
Therefore, for every $n\ge 1$, we have
$V^+_n(f,x)=V^+_n(\Theta_0,x)\left(\int f\, \mu_0\right)^n$.

Moreover, we have the following bounds for the moments and the rate of convergence:
for some finite constants $\mathcal{D}$, $0<\beta_n$ and $\gamma_n$ (see \eqref{eq:bound moments I}
and \eqref{eq:iteration moments supercritical f})
\begin{align}
&\sup\limits_{x} |V^+_n(f,x)| \le n! \,\mathcal{D}^{2n-1}\|f\|_\infty^n\,,\label{eq:bound moments 0}\\
&F^+_n(f):=\sup\limits_{t,x} e^{\,\beta_n \, t}\, \left|e^{n\lambda_0 t}\,\Ex{\langle Z_t\, ,f\rangle^n}-V^+_n(f,x)\right|
\le \gamma_n \|f\|_\infty^n.
\label{eq:bound speed convergence moments supercritical}
\end{align}

\item {\bf The survival probability.} As $t$ tends to $\infty$, $\Px{N_t> 0}$ 
converges, uniformly on $x\in\mathbb{R}$,
to a function $h \in \CBZ$  satisfying  the extra properties: 
\begin{equation}
\label{eq:extra properties K*}    
\forall x\in \mathbb{R}, \  h(x)>0,\quad
\sup\limits_{x\in \RR} h(x)<1.
\end{equation}

\item {\bf Convergence in Law.} For all $x$, the law of $e^{\lambda_0 t} Z_t$, in $M_F$ 
computed under the conditional
probability $\Px{\hspace{0.2cm}\vsbullet\hspace{0.2cm} \Big | \, N_t>0}$, converges for the weak topology to a 
probability measure $\Lambda_x^+$, characterized by the moments:
for all $f\in \CB,\, n\ge 1$
$$
\int_{M_f} (\langle \mu, f\rangle)^n \, \Lambda_x^+(d\mu)=\frac{V_n^x(f,x)}{h(x)}=\frac{V_n^+(\Theta_0,x)}{h(x)}
\left(\hbox{ $\int f \,d\mu_0$}\right)^n
$$
The measure $\Lambda_x^+$ is a Yaglom limit for the process $(e^{\lambda_0 t} Z_t)_t$.

\item The process $\left(e^{\lambda_0 t} \langle Z_t, \Theta_0\rangle\right)_t$ is a 
non-negative martingale, which is uniformly integrable in all $L^n(\PP_{\delta_x})$ for all $ 1\le n<\infty, \, x\in \RR$. 
Therefore there exists a non-negative random 
variable $W_\infty(x)$,  such that 
\begin{equation}
\label{eq:a.s. convergence Theta0}
e^{\lambda_{0}t} \langle Z_{t},\Theta_{0}\rangle \stackarrow{\PP_{\delta_{x}}\hbox{-- a.s.}}{t\to \infty}{1.1}
W_{\infty}(x).
\end{equation}
The distribution of $~W_\infty(x)$ is characterized by its moments, which are given by
$\Ex{W_\infty(x)^n}=V_+^n(\Theta_0,x)$.
In particular, $\Ex{W_\infty(x)}=\Theta_0(x)>0$. In addition we have 
\begin{equation}
\label{eq:prob W>0}
\Px{W_\infty(x)>0}=h(x)=\lim\limits_{t\to\infty}\Px{N_t>0},
\end{equation}
and  $\Lambda_x^+$ is characterized by
$$
\Lambda_x^+(\vsbullet)=\Px{W_\infty(x)\mu_0 \in \vsbullet \Big |\, W_\infty(x)>0}.
$$

\item {\bf Convergence in probability.} For every $x\in \RR$, the measure 
valued process $(e^{\lambda_0 t} Z_t)_t$, converges
as $t\to\infty$, in probability (under $\PP_{\delta_x}$) for the weak topology on $M_F$, to the random measure
$W_\infty(x) \mu_0$.

\item For all $1\le n<\infty$, we have the strong $L^n$ convergence: 
\begin{equation}
\label{eq:strong convergence Ln}
\lim\limits_{t\to\infty}\,
\sup\limits_{\hbox{\scriptsize$\stackunder{0\le t,\,x\in \RR}{\|f\|_{\infty}\! \le 1}$}}
\,\Big\| \, e^{\lambda_{0}t} \langle Z_{t},f \rangle - 
W_{\infty}(x) \,\hbox{$\int f\,  d\mu_{0}$ }\Big \|_{L^n\left(\PP_{\delta_x}\right)} = 0.
\end{equation}
\end{enumerate}
\end{theorem}

The following result is an a.s. convergence of the normalized random measures. For that, 
we need an extra hypothesis.

{\bf \large Assumption} (HSM): 
\vspace{-0.5cm}
\begin{enumerate}
\item We assume that the Markov process $(X_t)$
has a generator $\mathscr{G}$ whose domain contains $\mathscr{H}=C_c^\infty(\RR)$ such that 
for any $f\in \mathscr{H}$, there exists a constant $C(f)$, so that
$$
|\mathscr{G} f|\le C(f) \ .
$$
\item We also assume that in the semimartingale decomposition \eqref{eq:Master General},
 $M^{f}$ is a square integrable martingale with predictable quadratic 
variation satisfying
$$
d\left\langle M^f\right\rangle_t \le C^2(f) \langle Z_t,1\rangle \, dt
$$
\end{enumerate}
We recall that $f\equiv 1$ is in the domain of the generator and that 
$\mathscr{G} 1 = 0$ as well as $M^1=0$, so formally we have $C(1)=0$.

Note that this assumption is quite general and it is obviously satisfied for our three examples. 

\begin{theorem} 
\label{the:convergence a.s.}
Assume (HV), (HP) and (HSM). 
For every $x\in \RR$, the measure valued process $(e^{\lambda_0 t} Z_t)$ converge
$\PP_{\delta_x}$-a.s., when $t\to\infty$ for the weak topology on $M_F$, to the random measure
$W_\infty(x) \,\mu_0$.
\end{theorem}

\subsection{The $Q$-process} 
\label{sec:Q-process}

Note that \eqref{FK} can be generalized to trajectories by introducing the historical process $(L_{t}, t\ge 0)$ (cf. 
\cite{dawsonperkinsAMS}, \cite{perkins} or \cite{meleardtran_suphist}), i.e. the point measure valued process on the 
trajectories, defined for any $t$  by
\begin{equation}
\label{historical}
    L_{t} = \sum_{i\in  {\cal V}_{t}} \delta_{X^i_{.\wedge t}}.
\end{equation}

Indeed, for any $T>0$ and $F: C([0,T],\RR) \to \RR$ a bounded and continuous function and $x\in \RR$, 
\begin{equation}
\label{FK2} 
\Ex{\langle L_{T},F \rangle} = \EE_{x}\Big(\exp\Big(\int_{0}^T V(X_{s})ds\Big)F\left(X^T\right)\Big),
\end{equation}
where $X^T=(X_{t\wedge T})$. These formulas/theory  come from
\cite{kurtz1997conceptual,lyons1995conceptual} further developed for
instance in \cite{bansaye_limit_2011,cloez,hardyharris3,Marguet2}. Different  and  simpler proofs are given in \cite{calvez22}.
\medskip

The proof of the following theorem is given in Section \ref{sec:Q-process proof}.

\begin{theorem} 
\label{the:Q-process} Assume (HV) and (HP).
Let us fix  $s>0$. For all $\Phi:C([0,s],\RR)\to \RR$ and $F:\RR\to \RR$ bounded continuous functions, we have
\begin{enumerate}[(1)]
\item In the critical and subcritical cases, 
the following limit exists (recall that $\lambda_0=0$ in the critical case)
\begin{equation}
\label{eq:lim Q-process critical & subcritical}    
\lim\limits_{T\to \infty} \Ex{ F(\langle L_s,\Phi\rangle) \,\, |\,\, N_T>0}=
\Ex{F(\langle L_s,\Phi\rangle)\, e^{\lambda_0 s}\frac{\langle Z_s, \Theta_0\rangle}{\Theta_0(x)}}.
\end{equation}
\item In the supercritical case, the limit is given by
\begin{equation}
\label{eq:lim Q-process supercritical}
\lim\limits_{T\to \infty} \Ex{ F(\langle L_s,\Phi\rangle) \,\, |\,\, N_T>0}=
\Ex{F(\langle L_s,\Phi\rangle)\, 
\frac{1-\exp\left(\langle Z_s, \log(1-h)\rangle\right)}
{h(x)}}.
\end{equation}
\end{enumerate}
\end{theorem}

\section{Recursion for the moments.}
\label{sec:moments}

\subsection{A priori bounds on the total mass}

From a coupling based on representation \eqref{eq:Master General}, we get that for all initial 
value $x$, the process  $\langle Z_t\, ,1\rangle$ is dominated   by the Yule (pure birth) 
process $N^*$ whose birth rate is $b^*$ (cf. Assumption HV2)) and $N^*_0=1$, given by
$$N^* =  1+\int_{[0,t]\times \RR_+\times \NN} \ind_{1\le i\le N^*_{s}} \,
\ind_{\theta \le b^*}
Q(ds,d\theta,di).$$

 This motivates the following estimates of Yule processes.

\begin{lemma}
\label{moment-geometric}
For a Yule
process $N^*$ whose birth rate is $b^*$ and $N^*_0=1$, it holds
$$
\EE((N_t^*)^n)\le n!\, e^{n b^* t}.
$$
\end{lemma}
\begin{proof} Denote by $p=e^{-b^*t}$. It is well known that for fixed
  $t$, the distribution of $N_t^*$ is a geometric law with parameter
  $p$ (see for example \cite{Ross} page 377). 
  Then by the Carlitz identity (see for example 
\cite{Petersen} formula (1.10) page 10) we have for all $n\ge1$
$$
\EE((N_t^*)^n)=\sum\limits_{k\ge 1} k^n p(1-p)^{k-1}=\frac{1}{p^{n}}
\sum\limits_{q=0}^{n-1}\left\langle\begin{array}{c}
n\\q\end{array}\right\rangle (1-p)^{q},
$$
with $\left\langle\begin{array}{c}
n\\q\end{array}\right\rangle$ the Eulerian numbers.
It follows immediately from the combinatorial definition of the Eulerian
numbers that
$$
\EE((N_t^*)^n)= \frac{1}{p^{n}}
\sum\limits_{q=0}^{n-1}\left\langle\begin{array}{c}
n\\q\end{array}\right\rangle\le \frac{n!}{p^{n}}.
$$
  \end{proof}

Using the previous coupling and  that for all  
$f\in \MB$, $|\langle Z_t\, ,f\rangle | \le \|f\|_{\infty} \langle Z_t\, ,1\rangle\le  \|f\|_{\infty}  N_t^*$ 
and  the fact that the Yule process is increasing, we  immediately deduce the  following a priori 
estimates on the total mass moment.

\begin{corollary}
\label{massmoment}
For all $n\ge 1$, $x\in \RR, T> 0$, $f$ a bounded measurable function,
\begin{equation}
\label{eq:bound_un}
\left|\,\Ex{\sup_{t\le T} \langle Z_t\, ,f\rangle^n} \right| \le \|f\|_\infty^n \, n! \, e^{nb^* T}.
\end{equation}
\end{corollary}

\subsection{Cutoff on the death rate}
\label{sec:cutoff}

We now deal with the unboundedness of the death rate $d$.

In this section we prove that $Z$ can be approximated by a sequence of branching diffusion 
processes $(Z^m)_{m\in \NN}$
with bounded death rates. For that purpose, consider the following truncation of $d$
$$
d^{(m)}(x)=\begin{cases} d(x) &\hbox{if } |x|\le m\\
					      d(m) 	&\hbox{if } x\ge m\\
					      d(-m) &\hbox{if } x\le -m\\
	\end{cases}
$$

Let $Z^m$ be the branching diffusion process with birth rate $b$ and death rate $d^{(m)}$, 
driven by the same Brownian motions $(B^i)$ and Poisson point measure $Q$ as
 for $Z$.
 
 \begin{proposition} Assume (HP5). For all $x,n,t$ and all $f$ bounded
and measurable, the following limits hold uniformly in $t,x$ in compact sets
$$
\lim\limits_{m\to \infty} \Ex{\langle Z^m_t\, ,f\rangle^n}=
\Ex{\langle Z_t\, ,f\rangle^n},
$$
and
$$
\lim\limits_{m\to \infty} \Px{\langle Z^m_t,\ind\rangle>0}=
\Px{\langle Z_t,\ind\rangle>0}.
$$

\end{proposition}

\begin{proof}
(i) First, prove that for all $x$ it holds $\PP_{\delta_x}$-a.s., for all $t\ge 0$,
$$
Z_{t\wedge T_m}=Z^m_{t\wedge T_m}.
$$
 It is enough to consider $|x|<m$, because when $|x|\ge m$, then $T_m=0$ and the result
is obvious. So, we assume that $T_m>0$.\\
Consider $(\mathscr{T}_k)_{k\ge 1}$ the increasing sequence of 
random times associated with the branching mechanism that defines $Z$. We set $\mathscr{T}_0=0$. 
Since $Z$ is well defined on $\mathbb{R}_{+}$, there exists in the event that $T_m<+\infty$, 
a unique finite $k$ so that $\mathscr{T}_k<T_m\le \mathscr{T}_{k+1}$.
If $k=0$ then under $\PP_{\delta_x}-$a.s. $Z_{t\wedge T_m}=\delta_{X_{t\wedge T_m}}=Z^m_{t\wedge T_m}$, 
where the last equality
follows from the representation \eqref{eq:MasterZ} for $Z^m$ because
$d^{(m)}(X_{\sbullet \wedge T_m})=d(X_{\sbullet \wedge T_m})$.
Now, assume that $k\ge 1$. This means that $T_m>\mathscr{T}_1$ and, therefore, both
$Z,Z^m$ are killed at $\mathscr{T}_1$ or
both are branched, where two independent particles are born at position $X_{\mathscr{T}_1}$. 
This shows that
$Z_{t}=Z^m_{t}$ for $t\in 
[\mathscr{T}_1,T_m\wedge \mathscr{T}_2]$, proving
the case $k=1$. When $k\ge 2$
a similar argument shows that $Z_{t\wedge \mathscr{T}_j}=Z^m_{t\wedge \mathscr{T}_j}$ for $ j=1,\cdots,k$. 
Again, the representation \eqref{eq:MasterZ} shows
that $Z_{t}=Z^m_{t}: t \in [\mathscr{T}_k,T_m]$. When $T_m=+\infty$ we replace $T_m$ by
$T_m\wedge n$ for each fixed $n\in \NN$, in the previous argument.  

(ii)
Consider $n\ge 1, t\ge 0,1\le m\le \infty, x\in \RR$, and $f$ a bounded measurable function, then
$$
\begin{array}{ll}
\left|\Ex{\langle Z^m_t\, ,f\rangle^n}-\Ex{\langle Z^m_{t\wedge T_m},f\rangle^n}\right|
&\hspace{-0.2cm}\le 2\|f\|_{\infty}^n \,\Ex{\sup\limits_{s\le t} \,\langle Z^m_s,\ind\rangle^{2n}}^{1/2}
\Px{T_m\le t}^{1/2}\\
&\hspace{-0.2cm}\le 2\|f\|_{\infty}^n \,C(2n,t)^{1/2}\,  \Px{T_m\le t}^{1/2},
\end{array}
$$
where $C(2n,t)$ is the $2n$-th moment of a pure birth process with birth rate $b^*$ starting from $1$ at time $0$.
From Lemma \ref{moment-geometric}, we have that
$$
C(2n,t) \le (2n)! e^{2nb^*t}.
$$
The same bound holds for $Z$.
$$
\def\arraystretch{1.4}
\begin{array}{ll}
\left|\Ex{\langle Z_t\, ,f\rangle^n}-\Ex{\langle Z_{t\wedge T_m},f\rangle^n}\right|
\hspace{-0.2cm}&\le 2\|f\|_{\infty}^n \, \Ex{\sup\limits_{s\le t} \,\langle Z_s,\ind\rangle^{2n}}^{1/2}
\Px{T_m\le t}^{1/2}\\
\hspace{-0.2cm}&\le 2\|f\|_{\infty}^n\, C(2n,t)^{1/2}\, \Px{T_m\le t}^{1/2},
\end{array}
$$
and then
\begin{equation}
\label{eq:Moments_Z_Z^m}
\left|\Ex{\langle Z_t\, ,f\rangle^n}-\Ex{\langle Z^m_{t},f\rangle^n}\right|\le 
4 \|f\|_{\infty}^n\, C(2n,t)^{1/2}\, \Px{T_m\le t}^{1/2}.
\end{equation}
Finally, we also have 
\begin{equation}
\label{eq:survival_Z_Z^m}
\left | \Px{\langle Z^m_t,\ind\rangle>0}-
\Px{\langle Z_t,\ind\rangle>0}\right| \le 2 \Px{T_m<t}.
\end{equation}
We conclude using Assumption (HP5). 
\end{proof}

\subsection{Moments recursion and survival probability equation}
In this section we provide formulas for the moments and survival probability.
In what follows, for $t\ge 0, x\in \RR, n\ge 1$, we denote
\begin{equation}
\label{def:moments}
\def\arraystretch{1.5}
\begin{array}{l}
u^f_n(t,x)=\Ex{\langle Z_t\, ,f\rangle^n},\, u_n(t,x)=\Ex{\langle Z_t\, ,1\rangle^n}\ ;\\
u_0(t,x)=\Px{\langle Z_t\, ,1\rangle>0}.\\
\end{array}
\end{equation}

The key to studying the asymptotic for the moments and the survival probability is the following 
result, giving  recursion equations for these quantities.

\begin{proposition}
\label{lem:for_moments} 
\begin{enumerate}[(i)]
\item For  $f\in \CB$ and all $T<\infty$, the function 
$H_{f}$ given by 
$$
(t,x,w)\in [0,T]\times \RR\times \left\{w\in\CC: \, |w|<\left(\|f\|_\infty e^{b^* T}\right)^{-1}\right\}
\longrightarrow H_{f}(t,x,w)=\Ex{1-e^{w \langle Z_t\,,\, f\rangle}}\in \CC,
$$
is well defined. 

For each $(t,x)\in [0,T]\times \RR$ the function $H_{f}(t,x,\sbullet)$ is analytic in 
$\left\{w\in\CC: \, |w|<\left(\|f\|_\infty e^{b^* T}\right)^{-1}\right\}$ 
and satisfies the equation
\begin{equation}
\label{eq:equation for Laplace}   
H_{f}(t,x,w)=P_t\left(1-e^{w \,f(X_t)}\right)(x)-
\int_0^t P_s\left( b(\sbullet) H_{f}^2(t-s,\sbullet,w) \right)(x)\, ds.
\end{equation}

\item For all $t,x$ and $f\in \CB$, the following set of formulas hold for the survival probability and moments:
\begin{align}
&u_{0}(t,x)=u_{1}(t,x)-\int_0^t P_s\left(b(\sbullet)\, u_{0}^2(t-s,\sbullet)\right)(x)\, ds\label{eq:u0}\;;\\
&u^f_1(t,x)=\Ex{\langle Z_t\, ,f\rangle}=\Px{f} \hbox{ and for } n\ge 2 \label{eq:u1}\;; \\
&u_n^f(t,x)=u_1^{f^n}(t,x)+\int_0^t P_s\!\left(b(\sbullet)\,\sum\limits_{k=1}^{n-1} 
\binom{n}{k}\, u_{n-k}^f(t-s,\sbullet)\, u_{k}^f(t-s,\sbullet)\right)\!(x)\, ds\label{eq:un}.
\end{align}
\end{enumerate}
\end{proposition}
\begin{proof}
$(i)$ Consider  $f\in \CB$.
Corollary \ref{massmoment}  implies   the convergence of the series (in $w$) 
$\sum_{n\ge 1}\Ex{\langle Z_{t},f\rangle^n} \frac{w^n }{n!}$  in the complex open 
disc of radius $\left(\|f\|_\infty e^{b^* T}\right)^{-1}$. 
Then  for all  $x\in \RR$  and all $t\le T<\infty$, the function 
$H_{f}(t,x,\vsbullet)$ is well defined and analytic in the complex open disc of radius 
$\left(\|f\|_\infty e^{b^* T}\right)^{-1}$.      

To derive \eqref{eq:equation for Laplace}, we need to truncate $Z$ as in 
Section \ref{sec:cutoff}. For any $m\ge 1$, we define the function $H_{{f,m}}$ given by 
$$
H_{f,m}(t,x,w)=\Ex{1-e^{w \langle Z_t^m\,,\, f\rangle}}.
$$
Since we have similar moment estimates for $Z^m$, we deduce that
for each $(t,x)\in [0,T]\times \RR$, the function $H_{f,m}(t,x,\sbullet)$ is analytic in 
$\left\{w\in\CC: \, |w|<\left(\|f\|_\infty e^{b^* T}\right)^{-1}\right\}$. Let us prove that $H_{f,m}$ satisfies the equation 
\begin{equation}
\label{eq:equation for Laplace m}   
H_{f,m}(t,x,w)=\Ex{e^{\int_0^t W^m(X_s) ds} \left(1-e^{w f(X_t)}\right)}-
 \int_0^t  \Ex{e^{\int_0^s W^m(X_r)dr} H_{f,m}^2(t-s,X_s,w)}\, ds,
\end{equation}
where $W^{(m)}=-\; (b+d^{(m)})$ is the rate function associated to the first branching event.

Conditioning on the  first branching time, we have
\begin{align*}
H_{f,m}(t,x,w)=&\EE_x\left(e^{\int_0^t W^{(m)}(X_s) \, ds} \, \left(1-e^{w f(X_t)}\right)\right)\\
&+
\int_0^t \EE_x\left(\left(b+d^{(m)}\right)(X_s) \,e^{\int_0^s W^{(m)}(X_u)\, du} \frac{b(X_s)}{(b+d^{(m)})(X_s)}   \EE_{\delta_{X_s}}
\left(1-\exp(w\langle Z^m_{t-s}+Z^{'m}_{t-s}\,,f\rangle)\right)\right)ds,
\end{align*}
where $\ Z^m$ and $Z^{\,'m}$ are two independent identically distributed branching processes. 
For every $z\in \RR$, we have
$$
\EE_{\delta_{z}}\left(1-\exp(w\langle Z^m_{t-s}+Z^{'m}_{t-s}\,,f\rangle)\right)=2H_{f,m}(t-s,z,w)-H_{f,m}^2(t-s,z,w).
$$
So, we have
\begin{equation}
\label{eq:formula con R}
\begin{array}{ll}
H_{f,m}(t,x,w)\hspace{-0.2cm}&=\EE_x\left(e^{\int_0^t W^{(m)}(X_s) \, ds} \, \left(1-e^{w f(X_t)}\right)\right)\\
&\hspace{0.2cm}+\int_0^t \EE_x{e^{\int_0^s W^{(m)}(X_u)\, du} 
\left( 2b(X_s) H_{f,m}(t-s,X_s,w)-b(X_s)H^2_{f,m}(t-s,X_s,w)\right)}ds.
\end{array}
\end{equation}
Now, we use Proposition 2.9 in \cite{ZL}  replacing $W^{(m)}$ by $V^{(m)}= b- d^{(m)}$, which gives the desired equation
$$
H_{f,m}(t,x,w)=\EE_x\left(e^{\int_0^t V^{(m)}(X_s) \, ds} \,  \left(1-e^{w f(X_t)}\right)\right)-
\int_0^t \EE_x\left( e^{\int_0^s V^{(m)}(X_u)\, du}\, b(X_s)H^2_{f,m}(t-s,X_s,w)\right)ds. 
$$
 The bounds for the moments of $\langle Z^m_t,f\rangle$, uniform on $t,x$, 
and the convergence of these moments to the moments of $\langle Z_t,f\rangle$, proves that for all $t\le T, x\in \RR,
0<a<\left(\|f\|_\infty e^{b^* T}\right)^{-1}$ and all $|w|\le a$, we have the pointwise convergence
$\ 
H_{f,m}\to H_f\ 
$
as $m\to\infty$, where the sequence $(H_{f,m})$ is uniformly bounded by a constant, depending on $T,a$. 
The Dominated Convergence Theorem allows us to conclude that $H_f$ satisfies the desired equation.

$(ii)$  The equations \eqref{eq:un}, satisfied by the moments, are obtained by successive derivations of $H$ w.r.t. 
$w$, at $w=0$ in equation \eqref{eq:equation for Laplace} using the Dominated Convergence Theorem.

Finally, we prove that $u_0$ satisfies equation \eqref{eq:u0}. This is obtained by taking the limit in 
$$
H(t,x,w)=\Ex{1-e^{w\langle Z_t,1\rangle}},
$$
as $w\to -\infty$. For $t,x$ fixed, this function is also analytic in the open left half plane $\Re w<0$. By analytic
continuation, $H$ satisfies the same equation on this domain and it is bounded near $-\infty$. The result, that is, equation 
\eqref{eq:u0} follows from the Dominated Convergence Theorem, since $H(t,x,w)\to \Px{\langle Z_t,1\rangle >0}$, as 
$w\to-\infty, w\in \RR$.
\end{proof}

Using \eqref{eq:u0} and the semigroup property of $(P_{t})$, one can also prove that for any $0\le t_{0} \le t$, 
\begin{equation}
\label{eq:eq:asym survival critical II}
u_0(t,x)=P_{t-t_{0}}(u_0(t_{0},\sbullet))(x)-\int_{0}^{t-t_{0}}P_{s}\left(b(\sbullet)\,u_0^2(t-s,\sbullet)\right)(x)\, ds.
\end{equation}
See \cite{ZL} Proposition 2.12, page 35.

\section{Some properties of the survival probability}
\label{sec:survival}

\begin{theorem}
\label{thm:smooth-survival}
Under assumptions (HV) and (HP), for any $t>0$, $u_{0}(t, \sbullet) $ belongs to $\CBZ$. 
\end{theorem}

\begin{proof}
In Proposition \ref{lem:for_moments}, see equation \eqref{eq:u0}, we have proved that
$$
u_{0}(t,x)=u_{1}(t,x)-\int_0^t P_s\left(b(\sbullet)\, u_{0}^2(t-s,\sbullet)\right)(x)\, ds.
$$
By assumption HP4) we know that for any $t>0$, $u_{1}(t,x)= P_{t}1(x)$ belongs to $\CBZ$. 
The proof  of the theorem is then  based on the following lemma whose proof is postponed at the end of the section.

\begin{lemma}\label{meilleur} We introduce the  number
$\ 
C=1+\sup_{x\in\RR}e^{V(x)}\;
$ and  
for $0<\tau<1$ fixed, let 
 $\MBD$ denote the Banach space of bounded measurable functions on
$[0,\tau]\times \RR$ equipped with the sup norm denoted also by $\|\vsbullet\|_{\infty}$.
Assume (HV) and (HP) and let
$$\tau= \frac{1}{4\,C^{2}\,b_{*}}\in (0,1).$$
  Then, for any 
$v_{0}\in\CB$ with $v_{0}\ge 0$ and $\|v_{0}\|_{\infty}\le 1$ and 
any function $w\in \MBD$ satisfying $0\le w\le  1$ and  for any $t\in[0,\tau]$,
$$
w(t,\vsbullet)=P_{t}v_{0}(\vsbullet)-
\int_{0}^{t}P_{s}\left(b\,w^2(t-s)\right)(\sbullet) \;ds,
$$
we have $w(t,\vsbullet)\in\CBZ$ for any $0<t\le \tau$. Moreover,
$w(0,\vsbullet)=v_{0}$ and for any $x\in\RR$
$$
\lim_{t\searrow0}w(t,x)=v_{0}(x)\;.
$$
\end{lemma}

The proof of Theorem \ref{thm:smooth-survival} consists then in applying  recursively Lemma \ref{meilleur}  
with $v_{0} = 1$  for $n = 0$ and on successive time intervals
$[n\tau,\,(n+1)\tau]$ for $n \ge 1$, with
$ v_{0} = u_{0}(n\tau, \sbullet)$, using the nonlinear semigroup property \eqref{eq:eq:asym survival critical II}.
\end{proof}

\begin{proof}of Lemma \ref{meilleur}

The proof is based on Picard's method of successive approximations. 
Let $w_{0}$ be defined by
$$
w_{0}(t,x)=P_{t}v_{0}(x)\;.
$$
Note that $w_{0}\in\MBD$, that $\|w_{0}\|_{\infty}\le C$, and for any
$0<t\le\tau$, $w_{0}(t,\sbullet)\in\CBZ$ by HP4).

For $n\ge1$, we define recursively a function $w_{n}$ on
$[0,\tau]\times\RR$ by
$$
w_{n}(t,\vsbullet)=P_{t}v_{0}(\vsbullet)-
\int_{0}^{t} P_{s}\left(b\,w_{n-1}^2(t-s,\vsbullet)\right)(\sbullet) \;ds\;.
$$
\
We first show that for any $n$, $w_{n}\in \MBD$. The proof is recursive.
Assume $w_{n-1}\in \MBD$, then for any $0\le t\le \tau$ and any $x\in\RR$
we have
$$
\int_{0}^{t}\big(P_{s}(b\,w_{n-1}(t-s)^{2})\big)(x) \;ds
\le \tau\,C\,b_{*}\,\|w_{n-1}\|_{\infty}^{2}\;.
$$
Choosing
$\tau=\frac{1}{4\,C^{2}\,b_{*}}$,
if $\|w_{n-1}\|_{\infty}\le 2\,C$, then   $\|w_{n}\|_{\infty}\le
2\,C$. This implies recursively that for every $n\ge0$, $w_{n}\in\MBD$ and
$$
\|w_{n}\|_{\infty}\le 2\,C\;.
$$
We now show that for any $n\ge0$ and for any $0<t\le\tau$,
$w_{n}(\tau,\vsbullet)\in\CBZ$. This is true for $w_{0}$. For
$n>0$ the proof is recursive. Assume that for any $0<t\le\tau$,
$w_{n-1}(t,\vsbullet)\in\CBZ$. Then for any $0<s<t\le\tau$ fixed,
by HP4) the function
$$
x\mapsto P_{s}\left(b\,w_{n-1}^{2}(t-s)\right)(x)
$$
belongs to $\CBZ$. It follows from the Dominated Convergence Theorem
that for any $0<t\le\tau$, the function
$$
x\mapsto \int_{0}^{t}\big(P_{s}(b\,w_{n-1}(t-s)^{2})\big)(x) \;ds
$$
belongs to $\CBZ$, and then  $w_n(t,\vsbullet)\in \CBZ$, for all $0<t\le \tau$.

We have for any $0<t\le\tau$ and any $n\ge0$
\begin{align*}
w(t,\vsbullet)-w_{n+1}(t,\vsbullet)=&
\int_{0}^{t} P_{s}\left(b\,\big[w_{n}^2(t-s)-w^2(t-s)\big]\right)(\sbullet) \;ds\\
=&\int_{0}^{t} P_{s}\Big(b\,[w_{n}(t-s)-w(t-s)]\;[w_{n}(t-s)+w(t-s)]\Big)(\sbullet) \;ds\;.
\end{align*}
Since  $C\ge 1$, our choice of $\tau$ implies that
$$
\sup_{0<t\le\tau}\sup_{x\in\RR}\big|w(t,x)-w_{n+1}(t,x)\big|
\le \tau\,b_{*}\,C\,(1+2C)\,\sup_{0<t\le\tau}\sup_{x\in\RR}\big|w(t,x)-w_{n}(t,x)\big|
\le \frac{3}{4}\;
\sup_{0<t\le\tau}\sup_{x\in\RR}\big|w(t,x)-w_{n}(t,x)\big|
\;.
$$
This implies for any $0<t\le\tau$
$$
\lim_{n\to\infty} \sup_{x\in\RR}\big|w(t,x)-w_{n}(t,x)\big|=0\;.
$$
Since $\CBZ$ is closed in $\MB$ in the sup norm,  we get
$w(t,\vsbullet)\in\CBZ$ for any $0<t\le\tau$.

The last assertion of the  lemma
follows from the  Dominated Convergence Theorem.
\end{proof}

It is clear that for each $x$ the function $t\to u_0(t,x)$ is bounded by 1 and decreasing on $t$.
We now need to prove that  $u_0$ converges to a continuous function $h$ when $t$ tends to infinity, uniformly on $x$.
For that, we need another
equation for $u_0$. Passing to the limit in $m\to \infty$ and then $w\to -\infty$ in \eqref{eq:formula con R}, with
$f\equiv1$, and defining   
$$
W= -(b+d),
$$
we obtain the equality
\begin{align}
u_0(t,x)&=\Px{N_t>0}=\EE_x\left(e^{\int_0^t W(X_r) \ dr}\right)+ \int_0^t \EE_x\left(e^{\int_0^s W(X_r) \ dr} 
(2b(X_s) u_0(t-s,X_s)-b(X_s)u_0^2(t-s,X_s))\right) \, ds\nonumber\\
&=
\label{eq:u0Q}
Q_t(1)(x)+ \int_0^t Q_{t-s}
\Big(2b(\sbullet) u_0(s,\sbullet)-b(\sbullet)u_0^2(s,\sbullet))\Big)(x)\,ds.
\end{align}
Recall that $(Q_{t})$ is the perturbation of $(P_{t})$ by the operator of pointwise
multiplication by the function$-2b$ defined in \eqref{def:Qt}, namely
$$Q_{t}f(x) = \EE_x\left(e^{\int_0^t W(X_r) \ dr}f(X_{t})\right).$$

\begin{lemma}\label{prop:Qt} Under HP3) and HP4)
the semigroup  $(Q_{t})$ is strongly continuous and irreducible in
$\CBZ$. It maps $\CB$ to $\CBZ$ for any $t>0$. For any $T>0$,
$\ 
\int_{0}^{T}Q_{t}\;dt\ 
$
maps $\MB$ to $\CB$ in the sense that for any $f\in\MB$ the function 
$$
x\mapsto \int_{0}^{T}\EE_x\left(e^{\int_0^t W(X_r) \ dr}f(X_{t})\right)\;dt
$$
belongs to $\CB$.
\end{lemma}
\begin{proof}
The operator of pointwise multiplication by the function $-2b$ is
bounded in $\CBZ$. It follows from Theorem 1.1 page 75 in
\cite{Pazy} that $(Q_{t})$ is strongly continuous in $\CBZ$.

Irreducibilty follows from Proposition 3.3 page 183 in
\cite{AGGGLMNFS}.

Let $f\in \CB$ and consider the function
$$
F(x,u)=\EE_x\left(e^{\int_0^t V(X_r) \ dr} e^{-2\int_0^u b(X_s)\,ds}f(X_{t})\right)\;.
$$
We have $F(x,0)=P_{t}f(x)$ and $F(x,t)=Q_{t}f(x)$. This function is
differentiable in $u$ and we have for all $x$
\begin{equation}\label{Duhamel}
Q_{t}f(x)=P_{t}f(x)+\int_{0}^{t}\partial_{u}F(x,u)\;du
=P_{t}f(x)-2\int_{0}^{t}\EE_x\left(e^{\int_0^t V(X_r) \ dr} 
e^{-2\int_0^u b(X_s)\,ds}b(X_{u})\;f(X_{t})\right)\;du\;.
\end{equation}
By the Markov property we have for all $0<u<t$
$$
\EE_x\left(e^{\int_0^t V(X_r) \ dr} 
e^{-2\int_0^u b(X_s)\,ds}b(X_{u})\;f(X_{t})\right)
=\EE_x\left(e^{\int_0^u W(X_r) \ dr} b(X_{u})\EE_{X_{u}}\left[
e^{\int_0^{t-u} V(X_r) \ dr}\;f(X_{t-u})\right]\right)
$$
$$
=Q_{u}\big(b\,P_{t-u}f)(x)\in\CBZ,
$$
since by HP4) $P_{t-u}f\in\CBZ$. It follows from the Dominated Convergence Theorem that
$$
-2\int_{0}^{t}\EE_x\left(e^{\int_0^t V(X_r) \ dr} 
e^{-2\int_0^u b(X_s)\,ds}b(X_{u})\;f(X_{t})\right)\;du\in \CBZ\;,
$$
hence $Q_{t}f\in\CBZ$ since $P_{t}f\in \CBZ$.

Using equation \eqref{Duhamel} we have for $T>0$ and $f\in\MB$,
$$
\int_{0}^{T}\EE_x\left(e^{\int_0^t W(X_r) \ dr}f(X_{t})\right)\;dt=
\int_{0}^{T}P_{t}f(x)\;dt-2\int_{0}^{T}\int_{0}^{u}
\EE_x\left(e^{\int_0^u W(X_r) \ dr} b(X_{u})\EE_{X_{u}}\left[
e^{\int_0^{t-u} V(X_r) \ dr}\;f(X_{t-u})\right]\right)\;du\;dt\;.
$$
The first term belongs to $\CB$ by HP4). For the second term we use
Fubini's theorem to get
$$
\int_{0}^{T}\int_{0}^{u}
\EE_x\left(e^{\int_0^u W(X_r) \ dr} b(X_{u})\EE_{X_{u}}\left[
e^{\int_0^{t-u} V(X_r) \ dr}\;f(X_{t-u})\right]\right)\;du\;dt
$$
$$
=\int_{0}^{T}du \;
\EE_x\left(e^{\int_0^u W(X_r) \ dr} b(X_{u})\int_{u}^{T}\EE_{X_{u}}\left[
e^{\int_0^{t-u} V(X_r) \ dr}\;f(X_{t-u})\right]\right)\;dt
$$
$$
=\int_{0}^{T}du \;
\EE_x\left(e^{\int_0^u W(X_r) \ dr} b(X_{u})\int_{0}^{T-u}\EE_{X_{u}}\left[
e^{\int_0^{s} V(X_r) \ dr}\;f(X_{s})\right]\;ds \right)\;.
$$
From HP4) it follows that for $0<u<T$ the function
$$
y\mapsto \int_{0}^{T-u}\EE_{y}\left[
e^{\int_0^{s} V(X_r) \ dr}\;f(X_{s})\right]\;ds=
\int_{0}^{T-u}P_{s}f(y)\;ds
$$
belongs to $\CB$. Therefore since we already proved that $Q_{t}$ maps
$\CB$ to $\CBZ$, $0<u<T$ the function
$$
x\mapsto \EE_x\left(e^{\int_0^u W(X_r) \ dr} b(X_{u})\int_{u}^{T}\EE_{X_{u}}\left[
e^{\int_0^{t-u} V(X_r) \ dr}\;f(X_{t-u})\right]\right)\;dt
=Q_{u}\left(b\,\int_{0}^{T-u}P_{s}f\;ds\right)(x)
$$
belongs to $\CBZ$. It follows from the Dominated Convergence Theorem
that the function 
$$
x\mapsto\int_{0}^{T}du \;
\EE_x\left(e^{\int_0^u W(X_r) \ dr} b(X_{u})\int_{0}^{T-u}\EE_{X_{u}}\left[
e^{\int_0^{s} V(X_r) \ dr}\;f(X_{s})\right]\;ds \right)\;.
$$
belongs to $\CBZ$. 
\end{proof}

\begin{lemma}
\label{lem:growth} We assume HP1), HP2) and HQ). 
The semigroup  $(Q_{t})$ has a strictly negative asymptotic growth rate in $\CBZ$ and in $\CB$.

\end{lemma}
\begin{proof}
We prove the result in $\CBZ$. The generalization to $\CB$ follows by approximating  bounded 
pointwisely $f \in \CB$ by a sequence of truncated functions belonging to $\CBZ$  and 
using the Dominated Convergence Theorem.

The subcritical case is obvious.

In the critical case, for $z\in \CC$, denote by $(Q_{t}^{z})$ the semigroup $(P_{t})$
perturbed by the multiplication operator $z \,b$. We are of course
interested in the case $z=-2$. As proved in Lemma \ref{prop:Qt}, $(Q_{t}^{z})$ is a strongly
continuous semigroup in $\CBZ$.
It follows from the convergence of the Dyson-Phillips series (see
Theorem 4.9 page 121 
in \cite{BA}) that $(Q_{t}^{z})$
is a holomorphic semigroup in $\CBZ$ (see \cite{Kato}). In particular
$Q_{1}^{z}$ is analytic in $z$ with $Q_{1}^0=P_{1}$. Therefore we can apply analytic
perturbation theory (see \cite{Kato}). 
It follows from HP2), Theorem 1.7
page 369 in \cite{Kato} and the results of section VII.1.5 in \cite{Kato}
 that there exists $r>0$ such that  for $|z|\le r$,
$Q_{1}^{z}$ has a simple eigenvalue 
$$
-\tilde\lambda_{0}(z)=z \,\int b\,\Theta_{0}\,d\mu_{0} +\mathcal{O}(|z|^{2})
$$
of multiplicity one, and the rest of the spectrum belongs to the
closed disk of radius $\exp(-\lambda_{1})+\mathcal{O}(|z|)$\;.
From the irreductibility and Proposition 3.5 page 185 in [1], we deduce that since $b \Theta_{0}$ is  positive and $\mu_{0}$ is positive (in the sense of \cite{AGGGLMNFS} p.118-119), then 
$$
\int b\,\Theta_{0}\,d\mu_{0}>0\;,
$$
and the semigroup $(Q_{t}^{-(r\wedge 2)})$ has strictly negative
asymptotic growth rate. The result follows observing that for any
$f\in\CBZ$ and any $x$
$$
|Q_{t}f(x)|\le Q_{t}|f|(x)\le Q_{t}^{-(r\wedge 2)}|f|(x)\;.
$$

Let us now focus on the supercritical case under HQ). 
If $b+d$ doesn't vanish, then its minimum is strictly positive and the results immediately follows. 
Now, assume that $(Q_{t})$ is quasi-compact. The growth
rate of $Q_{t}$ cannot be positive since $W\le0$. On the other hand, if the growth rate is zero, we shall get a 
contradiction, showing the result. Indeed, assume the growth rate is $0$, then
by Corollary 2.11(b) page 216 in \cite{AGGGLMNFS}, we conclude there is a strictly
positive  function $G_{0}\in\CBZ$ such that for any $t\ge0$ and any $x$
$$
G_{0}(x)=Q_{t}G_{0}(x)=
\EE_x\left(e^{\int_0^t W(X_r) \ dr}G_{0}(X_{t})\right)\le
\EE_x\big(G_{0}(X_{t})\big)
$$
since $W\le0$.  Let $K$  be the set of maximum for $G_{0}$. 
Since $G_{0}\in\CBZ$, $K$ is a nonempty compact set.  For any $x_{*}\in K$,
we have for all $t\ge0$
$$
G_{0}(x_{*})\le \EE_{x_{*}}\big(G_{0}(X_{t})\big)\le G_{0}(x_{*})\;.
$$
Therefore if $(X^{x_{*}}_{t})$ denotes the process starting from
$x_{*}$, we have for any $t$
$$
G_{0}(X_{t}^{x_{*}})=G_{0}(x_{*})\;,
$$
almost surely. Hence for all $t\ge 0$
\begin{equation}
\label{eq:support}
\mathbb{P}(X_{t}^{x_{*}}\in K) =1.
\end{equation}
Consider now a non zero $\CBZ$-function $f$ such $\support(f) \cap K =\emptyset$. 
Thanks to HP1 and HP2), and for $t$ large enough, $P_{t}f(x_{*})>0$. 
From the obvious inequality
$$P_{t}f(x_{*}) \le e^{b^* t}\,\mathbb{E}_{x_{*}}(f(X_{t})),$$
we then deduce that $\mathbb{E}_{x_{*}}(f(X_{t})) >0$. That is a contradiction with \eqref{eq:support}.
\end{proof}

\bigskip \noindent 
Let us now prove that $u_{0}(t,.)$ converges uniformly to $h$. For any fixed $x$, the function $u_{0}(t,x)$ is nonincreasing in
$t$. We denote by $h$ the measurable function
$$
h(x)=\lim_{t\to\infty}u_{0}(t,x)\;.
$$
We obviously have $0\le h(x)\le 1$ for all $x$.

\begin{corollary}\label{eqh}
Under Hypotheses HP) and HQ), the function $h$ satisfies for all $x$ the equation
$$
h(x)=\int_{0}^{\infty}\EE_{x}\left(e^{\int_{0}^{s} W(X_{\tau})\,d\tau}\big(
2\, b(X_{s})\,h(X_{s})-b(X_{s})\,h^2(X_{s})\big)\;\right)\;ds\;.
$$
\end{corollary}
\begin{proof}
The  function on $[0,\infty)\times \Omega$
$$
(s,X_{\sbullet})\mapsto e^{\int_{0}^{s} W(X_{\tau})\,d\tau}
$$
is integrable with respect to $ds\times d\PP_{x}$ since by Lemmas
\ref{prop:Qt} and \ref{lem:growth},  $Q_{t}1(x)$
decays exponentially fast in $t$.

The family of functions (with parameter $t$)
$$
(s,X_{\sbullet})\mapsto
\pcun_{[0,t]}(s)e^{\int_{0}^{s} W(X_{\tau})\,d\tau}
\big(
2\, b(X_{s})\,u_{0}(t-s,\,X_{s})-b(X_{s})\,u_{0}^{2}(t-s,\,X_{s})\big)
$$
converges pointwise in $[0,\infty)\times \Omega$, when $t$ tends to infinity,  to
$$
(s,X_{\sbullet})\mapsto e^{\int_{0}^{s} W(X_{\tau})\,d\tau}\big(
2\, b(X_{s})\,h(X_{s})-b(X_{s})\,h(X_{s})^{2}\big)\;.
$$
We have obviously
$$
\pcun_{[0,t]}(s)e^{\int_{0}^{s} W(X_{\tau})\,d\tau}
\big(
2\, b(X_{s})\,u_{0}(t-s,\,X_{s})-b(X_{s})\,u_{0}(t-s,\,X_{s})^{2}\big)\le 2\,\|b\|_{\infty} \;e^{\int_{0}^{s} W(X_{\tau})\,d\tau}\;,
$$
and the result follows from the  Dominated Convergence Theorem (for the measure $ds\times d\PP_{x}$).
\end{proof}

\begin{proposition}\label{prop:hcon}
The function  $h$ belongs to $\CBZ$, and $u_{0}(t,\vsbullet)$ converges to
$h$ uniformly.
\end{proposition}
\begin{proof}
For each $x$, the function $u_{0}(t,x)$ is decreasing in $t$. Therefore
$h$ is the limit of a decreasing sequence of continuous functions,
hence upper  semi continuous, see for example \cite{Royden}, page 51.

Similarly, let for $n\in\NN$
$$
h_{n}(x)=\int_{0}^{n}\EE_{x}\left(e^{\int_{0}^{s} W(X_{\tau})\,d\tau}\big(
2\, b(X_{s})\,h(X_{s})-b(X_{s})\,h(X_{s})^{2}\big)\;\right)\;ds
=\int_{0}^{n}Q_{s}\big(2\,b\,h-b\,h^{2})\;ds.
$$  
Since $0\le h\le1$, we have $2\,b\,h-b\,h^{2}\ge0$ and the sequence
$(h_{n})$ is non decreasing. By Lemma \ref{prop:Qt}, $h_{n}\in\CBZ$
for all $n$. Therefore
$h$ is the limit of an increasing sequence of continuous functions,
hence lower  semi continuous, see again \cite{Royden}, page 51.

Being upper and lower semi continuous, $h$ is continuous. Since for
all $x$ we have  $0\le h(x)\le u_{0}(1,\,x)$ and $u_{0}(1,\vsbullet)$
belongs to $\CBZ$ by Theorem \ref{thm:smooth-survival}, we conclude that
$h\in\CBZ$.

Now, the uniform convergence of $u_0(t,\sbullet)$ towards $h$ is direct from Dini's Theorem. 
Indeed, since $0\le  u_{0}(1,\vsbullet)\in\CBZ$ for any $\epsilon>0$, there exists a
compact set $K_{\epsilon}$ such that
$$
\sup_{y\in K_{\epsilon}^{c}}u_{0}(1,y)\le \epsilon\;.
$$
Therefore since $0\le h\le  u_{0}(1,\vsbullet)$,
$$
\sup_{y\in K_{\epsilon}^{c}}\big|h(y)-u_{0}(1,y)\big|\le 
\sup_{y\in K_{\epsilon}^{c}}\big(u_{0}(1,y)-h(y)\big)\le \epsilon.
$$
On $K_{\epsilon}$ the family of continuous functions
$\big(u_{0}(t,\vsbullet)\big)$ is decreasing and converges pointwise
to the  continuous function $h$. Uniform convergence on $K_{\epsilon}$
follows from Dini's Theorem. The result follows. 
\end{proof}

\section{Proofs in the critical case $\lambda_0=0$}
\label{sec:proof_critical}

\subsection{Convergence of the moments: proof of Proposition \ref{pro:critical estimation}}
\label{sec:critical_moments}

We start with the asymptotic for the moments of $\langle Z_t\,,1\rangle$. The general case
$\langle Z_t\, ,f\rangle$, for bounded measurable $f$ is treated similarly. We will use 
repeatedly that $\Theta_0$ is a right  eigenvector for each $P_t$, associated to
the eigenvalue $e^{\lambda_0 t}=1$ (Assumption HP1)). Recall that $u_{n}(t,x)$ is 
the $n$-th moment of the total mass,  defined in \eqref{def:moments}.
We consider  the normalized moments as 
$$
v_n(t,x)=\frac{u_n(t,x)}{(t+1)^{n-1}} , n\ge 1
$$
and we notice that $v_1=u_1$.
Therefore, a recursive  expression for $v_n$ (see \eqref{eq:un}) is given by
$$
v_n(t,x)=\frac{u_1(t,x)}{(t+1)^{n-1}}+
\int_0^t P_s\left(b(\sbullet)\sum\limits_{k=1}^{n-1} \binom{n}{k}\, v_k(t-s,\sbullet)\,v_{n-k}(t-s,\sbullet)\right)(x)\, 
\frac{(t-s+1)^{n-2}}{(t+1)^{n-1}} ds\,,
$$
and we will prove by induction that all $v_n$ are uniformly bounded on $(t,x)$. Indeed, from Assumption HP2), 
take $C_1$ a constant greater or equal to 1, that bounds $u_1$, that is 
\begin{equation}
\label{eq:0.0}
C_1:=\sup_{t\ge 0,x} u_1(t,x)\vee 1.
\end{equation}
By induction, we show that $v_n$ is bounded by 
\begin{equation}
\label{eq:0.00}
C_n:=\sup_{t,x} v_n(t,x)\le n! D^{2n-1},
\end{equation}
where $D=C_1 (b^*\vee 1)$. Indeed, we have
$$
\def\arraystretch{1.4}
\begin{array}{ll}
v_n(t,x)&\hspace{-0.2cm}\le \frac{C_1}{(t+1)^{n-1}}+b^* n! D^{2n-2}(n-1)\int_0^t u_1(s,x) 
\frac{(t-s+1)^{n-2}}{(t+1)^{n-1}} ds\\
&\hspace{-0.2cm}\le \frac{C_1}{(t+1)^{n-1}}+b^*C_1 n! D^{2n-2} \left(1-\frac{1}{(t+1)^{n-1}}\right)
\le n! D^{2n-1}.
\end{array}
$$
That concludes the proof of  \eqref{eq:CRITICAL bounds moments I}.

We now prove that $v_n(t,x)$, for $n$ fixed, converges as $t\to\infty$, uniformly on $x$, to
\begin{equation}
\label{eq:vn critical}
    V_n(x)=n!\Theta_0(x) A^n B^{n-1},
\end{equation}
at rate $1/(t+1)$. 
We recall that
\begin{equation}
\label{eq:A&B}
A=\int  \mu_0(dy),\quad B=\int b(y) \Theta_0^2(y) \, \mu_0(dy).
\end{equation}
The convergence result  is true for $n=1$ and we prove it by induction on $n$. 
 Recall that from Assumptions (HP), there exists a constant $H$ 
such that for all $g\in \CB$ and all $t\ge 0$, 
\begin{equation}
\label{eq:0}
\|P_t(g)- \Pi(g)\|_{\infty}\le H\;\|g\|_\infty \, e^{-\lambda_1 t} \
\hbox{ and then }\ \|P_t(g)- \Pi(g)\|_{\infty}\le \widetilde H\; \|g\|_\infty\, \frac{1}{t+1}\;,
\end{equation}
with $H<\widetilde H=\widetilde H(\lambda_1)<\infty$ and since $\lambda_{1}>0$ in this case (where $\lambda_{0}=0$).

Let us prove that for all $n\ge 1$ there exists a constant $E_n$ so that 
\begin{equation}
\label{eq:0.1}
\|v_n(t,x)-V_n(x)\|_\infty\le E_n \frac{1}{t+1}.
\end{equation}
This is true for $n=1$ with $E_1=\widetilde H$. Moreover, it also holds $|v_1(t,x)-V_1(x)|\le He^{-\lambda_1 t}$.
Now, we continue proving the bound by induction on $n$. We do the case $n=2$ and then for general $n\ge 3$.
Recall that $V_2= 2 A^2B\, \Theta_0$ is bounded by a constant $C_{2}$ and $P_s\Theta_0=\Theta_0$. Then $P_sV_2=V_2$ and we have
\begin{align*}
|v_2(t,x)-V_2(x)|&\leq\frac{v_1(t,x)+V_2(x)}{t+1}+\int_0^t 
\left|P_s(2b(\sbullet) v_1^2(t-s,\sbullet)-V_2(\sbullet))(x)\right| \frac{1}{t+1} ds\\
&\le\frac{v_1(t,x)+V_2(x)}{t+1}+\int_0^t \left|P_s\Big(2b(\sbullet) v_1^2(t-s,\sbullet)-
\Theta_0(\sbullet)\int 2 b(y) v_1^2(t-s,y) \mu_0(dy)\Big)(x)\right| \frac{1}{t+1} ds\\
&\hspace{+0.2cm}+2\int_0^t \left|P_s\Big(\Theta_0(\sbullet)\int b(y) v_1^2(t-s,y) \mu_0(dy)
-\Theta_0(\sbullet)A^2B\Big)(x)\right| \frac{1}{t+1} ds\\
&\le\frac{C_1+C_2}{t+1}+2H\int_0^t \|b v^2_1(t-s,\sbullet)\|_\infty e^{-\lambda_1 (t-s)}\frac{1}{t+1} ds\\
&\hspace{+0.2cm}+2\int_0^t \left|P_s\Big(\Theta_0(\sbullet)\int b(y) (v_1^2(t-s,y)-V_1^2(y)) \mu_0(dy)\Big)(x)\right| 
\frac{1}{t+1} ds\\
&\le\frac{C_1+C_2}{t+1}+2Hb^*C_1^2 \lambda_1^{-1}\frac{1}{t+1}+4\|\Theta_0\|_\infty b^* C_1H A \frac{1}{t+1}
\int_0^t e^{-\lambda_1(t-s)} ds\\
&\le \frac{C_1+C_2+2Hb^*C_1\lambda_1^{-1}(C_1+2A\|\Theta_0\|_\infty)}{t+1}=E_2\frac{1}{t+1}.
\end{align*}
Now, for $n\ge 3$ we use the same ideas with $V_n$ defined in \eqref{eq:vn critical}.
$$
\def\arraystretch{1.4}
\begin{array}{l}
|v_n(t,x)-V_n(x)|\le \frac{v_1(t,x)+V_n(x)}{(t+1)^{n-1}}+\int\limits_0^t \sum\limits_{k=1}^{n-1} \binom{n}{k} 
\Big|P_s(b(\sbullet)\, v_kv_{n-k}(t-s,\sbullet)-k!(n-k)!\Theta_0 A^n B^{n-1})(x)\Big| 
\frac{(t-s+1)^{n-2}}{(t+1)^{n-1}} ds\\
\le \frac{(C_1+C_n)}{(t+1)^{n-1}}+\sum\limits_{k=1}^{n-1} \binom{n}{k}
\int\limits_0^t \Big|P_s(b(\sbullet)\, v_kv_{n-k}(t-s,\sbullet)-b(\sbullet) 
V_kV_{n-k}+b(\sbullet) V_kV_{n-k}-k!(n-k)!\Theta_0 A^n B^{n-1})(x) \Big|
\frac{(t-s+1)^{n-2}}{(t+1)^{n-1}} ds \\
\le \frac{(C_1+C_n)}{(t+1)^{n-1}}+\sum\limits_{k=1}^{n-1} \binom{n}{k}
\int\limits_0^t \left(\Big|P_s(b(\sbullet)\,(v_k(t-s,\sbullet)-V_k)v_{n-k})(x)\Big|+
\Big|P_s(b(\sbullet)\,(v_{n-k}(t-s,\sbullet)-V_{n-k})V_{k})(x)\Big|\right) \frac{(t-s+1)^{n-2}}{(t+1)^{n-1}} ds\\
\hspace{0.4cm} +n!\sum\limits_{k=1}^{n-1}  A^n B^{n-2} \int\limits_0^t \Big|P_s(b\Theta_0^2-\Theta_0 B)(x)\Big|
\frac{(t-s+1)^{n-2}}{(t+1)^{n-1}} ds \\
\le \frac{(C_1+C_n)}{(t+1)^{n-1}}+\sum\limits_{k=1}^{n-1} \binom{n}{k}
(C_{n-k}E_k+C_kE_{n-k})\, b^*\int\limits_0^t P_s(\ind)(x) \frac{(t-s+1)^{n-3}}{(t+1)^{n-1}} ds
+ n!\sum\limits_{k=1}^{n-1}  A^n B^{n-2} \widetilde H b^* \|\Theta_0\|_\infty^2
 \int\limits_0^t \frac{(t-s+1)^{n-2}}{(t+1)^{n-1}} e^{-\lambda_1 s}\, ds\\ 
\le \frac{(C_1+C_n)}{(t+1)^{n-1}}+C_1b^*\sum\limits_{k=1}^{n-1} \binom{n}{k}
(C_{n-k}E_k+C_kE_{n-k})\int\limits_0^t \frac{(t-s+1)^{n-3}}{(t+1)^{n-1}} ds
+ n!\sum\limits_{k=1}^{n-1}  A^n B^{n-2} \widetilde H b^* \|\Theta_0\|_\infty^2
 \frac{(t+1)^{n-2}}{(t+1)^{n-1}} \int\limits_0^{t}  e^{-\lambda_1 s} ds\\
\le \left(C_1+C_n+\frac{C_1b^*}{n-2}\sum\limits_{k=1}^{n-1} {n \choose k}
(C_{n-k}E_k+C_kE_{n-k}) + n!(n-1) A^n B^{n-2} \widetilde H b^* \|\Theta_0\|_\infty^2 \lambda_1^{-1}\right)\frac1{t+1}\\
\le \Big[C_1+n!D^{2n-1} + n! b^*(C_1\vee (\widetilde H\lambda^{-1}_1))\Big(\frac{2}{n-2}\sum\limits_{k=1}^{n-1} \frac{D^{2k-1}}{(n-k)!}E_{n-k} 
+ (2A)^{n}B^{n-2} \|\Theta_0\|_\infty^2\Big)\Big] \frac{1}{t+1}=E_n\frac{1}{t+1}.\\
\end{array}
$$
Assume that $E_k\le k! G^{2k+1}$ for $k=1,...,n-1$ and consider a large enough $G$ so that 
\begin{equation}
\label{eq:bound_G}
E_1\vee D\le G, E_2\le 2 G^5, 2+5b^*(C_1\vee (\widetilde H\lambda_1^{-1}))\le G, 2AB\le G^2, A^2\|\Theta_0\|_\infty^2\le G^4,
\end{equation}
then
$$
E_n\le n! G^{2n+1}.
$$

So far, we have proved \eqref{eq:convergence moments critical} for $f=1$, that is, for all $x, n\ge 1$
$$
\lim\limits_{t\to\infty} \frac{B}{\Theta_0(x)}\, (t+1)\, 
\EE_x\left(\left(\frac{N_t}{(t+1)\, \int  \mu_0(dy)  \int   b(y) \Theta_0^2(y)\, \mu_0(dy)}\right)^n\right)
=n!=\EE(\xi^n),
$$
where $\xi$ is an exponential random variable with parameter 1.

\medskip
We now sketch the proof for
$f\in \CB$ generalizing the case $f\equiv 1$.
The moments $\unf(t,x)=\Ex{\langle Z_t\, ,f\rangle^n}$ satisfy a similar recursion as $\un$:
$$
v_1^{(f)}(t,x)=u_1^{(f)}(t,x)=\EE_x(\langle Z_t\, ,f\rangle)=\EE_x(\langle Z_t\, ,f\rangle, N_t>0),
$$ 
and for $n\ge 2$
$$
\begin{array}{l}
\unf(t,x)=u_1^{f^n}(t,x)
+\int_0^t \sum_{k=1}^{n-1} \binom{n}{k} P_{s}\Big(b(\sbullet) u_k^{(f)}(t-s,\sbullet)u_{n-k}^{(f)}(t-s,\sbullet)\Big)(x) ds\,;\\
\\
v_n^{(f)}(t,x)=
\frac{u_1^{f^n}(t,x)}{(t+1)^{n-1}}
+\int_0^t \sum_{k=1}^{n-1} \binom{n}{k} P_{s}\Big(b(\sbullet) v_k^{(f)}(t-s,\sbullet)v_{n-k}^{(f)}(t-s,\sbullet)\Big)(x) 
\frac{(t-s+1)^{n-2}}{(t+1)^{n-1}}ds .\\
\end{array}
$$
Notice that $v^{(f)}_1(t,x)=u^{(f)}_1(t,x)$ tends to $ \Theta_0(x) \int f(y) \mu_0(dy)$ when $t$ 
tends to infinity. By induction we deduce that
$$
V_n^{(f)}(x)=\lim\limits_{t\to \infty} v_n^{(f)}(t,x)=n!\Theta_0(x) \left(\int f(y) 
\mu_0(dy)\right)^n 
\left(\int  \Theta_0(y)^2  \mu_0(dy)\right)^{n-1}.
$$
Note that $\int f(y)  \mu_0( dy)=\nu(f)A$.
We obtain the uniform bound, for all $t\ge 0, x$ (see \eqref{eq:0.00}) 
\begin{equation}
\label{eq:0.00f}
|v_n^{(f)}(t,x)|\le \|f\|_\infty^n v_n(t,x)\le n! 
\|f\|_\infty^n D^{2n-1}.
\end{equation}
Notice that $V_n^{(f)}(x)\le n! \|f\|_\infty^n(C_1)^{2n-1}$.
Similarly as \eqref{eq:0.1}, we get the uniform convergence
$$
|v_n^{(f)}(t,x)-V_n^{(f)}(x)|\le E_n^{(f)}\frac{1}{t+1},
$$
where
\begin{equation}
\label{eq:2.f}
E_n^{(f)}\le n! (G^{(f)})^{2n+1},
\end{equation}
and $G^{(f)}$ is given by \eqref{eq:bound_G} replacing $D,C_1$ and $A$ by $\|f\|_\infty D,
\|f\|_\infty C_1, A\|f\|_\infty$,  respectively.

\subsection{The survival probability: proof of  Proposition \ref{galere}  }
\label{sec:surv_prob_critical0}

We have seen in  Proposition \ref{prop:hcon} that  $u_0$ converges to $h$ when $t$ tends to infinity, uniformly on $x$. 
Now, we want to prove that $h\equiv 0$. 

 First we prove the

\begin{lemma}\label{mu0bh}
Under Hypotheses (HV) and (HP),
$$
\int b\, h\,\,d\mu_{0}=0\;.
$$
\end{lemma}
\begin{proof}
Integrating  \eqref{eq:u0} with respect to $\mu_{0}$ we obtain
$$
\int u_{0}(t,x)\;d\mu_{0}(x)=\int d\mu_{0}(x)-
\int_{0}^{t}\int b(x) \; u_{0}(t,x)^{2}\;d\mu_{0}(x)\;ds\;.
$$
Since $u_{0}(t,x)$ is nonnegative, monotone decreasing in $t$ and
smaller than one  
this implies for any $t\ge 0$
$$
\int_{0}^{t}\int b(x) \; h(x)^{2}\;d\mu_{0}(x)\;ds\le \int d\mu_{0}(x)\;.
$$
Since 
$$
\int b\,h^{2}\,d\mu_{0}\ge 0
$$
this implies 
$$
\int b\,h^{2}\,d\mu_{0}=0\;.
$$
Since $b\ge0$ and $h\ge0$, the result follows. 
\end{proof}

\begin{lemma}\label{hnul} Under Hypotheses  (HV) and (HP),
for any $x\in\RR$, $h(x)=0$.

Hence,
$$\lim_{t\to \infty} \|u_{0}(t,.)\|_{\infty} = 0.$$
\end{lemma}
\begin{proof}
The proof is by contradiction.

Assume there exists $x_{0}\in\RR$ such that $h(x_{0})>0$, Corollary \ref{eqh}  implies
that for some $T>0$
$$
\int_{0}^{T}\EE_{x_{0}}\left(e^{\int_{0}^{s} W(X_{\tau})\,d\tau}\big(
2\, b(X_{s})\,h(X_{s})-b(X_{s})\,h(X_{s})^{2}\big)\right)\;ds>0\;.
$$
Hence since $V\ge W$ and $2\,b\,h-b\,h^{2}\le 2\,b\,h$ (since $h\ge0)$,
$$
\int_{0}^{T}\EE_{x_{0}}\left(e^{\int_{0}^{s}V(X_{\tau})\,d\tau}\big(
2\, b(X_{s})\,h(X_{s}\big)\right)\;ds =\int_{0}^{T} P_{s}(2bh)(x_{0})ds >0\;.
$$
It follows from HP3) that $\int_{0}^{T} P_{s}(2bh)(\vsbullet)\,ds \in \CB$. 
Hence there exists $\delta>0$ such that
$$
\inf_{x\in [x_{0}-\delta,\,x_{0}+\delta]}
\int_{0}^{T}\EE_{x}\left(e^{\int_{0}^{s}V(X_{\tau})\,d\tau}\,
2\, b(X_{s})\,h(X_{s}\big)\right)\;ds>0\;.
$$
Since 
$$
\int_{0}^{T}\EE_{y}\left(e^{\int_{0}^{s}V(X_{\tau})\,d\tau}\,
2\, b(X_{s})\,h(X_{s})\right)\;ds\ge0
$$
for all $y$, 
we conclude  by integrating with respect to
$d\mu_{0}$ and using Fubini's Theorem that
$$
 \int \int_{0}^{T} P_{s}(bh)(y)ds\,\mu_{0}(dy)=\int_{0}^{T}\int b\,h\;\mu_{0}(dy)\;ds >0,
$$
since $\mu_{0}$ charges positive mass to every open set, by HP2) in Section \ref{sec:hypotheses}.
This gives a contradiction with the result of Lemma \ref{mu0bh}. Hence there cannot
exist $x_{0}$ such that $h(x_{0})>0$. 
\end{proof}

\subsection{The survival probability: proof of  Theorem \ref{the:critical} }
\label{sec:surv_prob_critical}
Recall that the one-dimensional projection $\Pi$ has been defined in HP2). 
We consider $\Pio= Id -\Pi$  and define $\Psi$ as  $\Pio u_{0}$. We have
\begin{align}
u_0(t,x)
&= \Theta_0(x)\,r(t) + \Psi(t,x),
\end{align}
where $r(t)= \int u_0(t,y)\mu_0(dy)$. Using that $\mu_{0}$ is a left eigenmeasure of 
$P_{t}$ and \eqref{eq:u0}, we have\begin{align}  r(t)&:=\int u_0(t,y)\mu_0(dy)=\int u_1(t,y)\mu_0(dy)-
\int_0^t \int P_s\left(b(\sbullet)\,u_0^2(t-s,\sbullet)\right)(y)\, \mu_0(dy)\,ds\nonumber \\
\hspace{0.7cm}&=\int \mu_0(dy)-\int_0^t \int b(y) u_0^2(t-s,y) \, \mu_0(dy) \,ds
=A-\int_0^t \,\int  b(y) u_0^2(s,y)\,\mu_0(dy) \,ds.
\label{eq:def r}
\end{align}

Since $\mu_{0}$ is a positive measure, we deduce from Proposition \ref{galere} that 
the function $t\to r(t)$ is decreasing, bounded and converges to $0$
as $t\to \infty$, and  also  that 
$$
\lim\limits_{t\to \infty}\|\Psi(t,\sbullet)\|_\infty=0.
$$ 
 
 The crucial part of the proof is to show the upper bound
\begin{equation}
\label{eq:upper bound r(t)}
\sup\limits_{t\ge 0}\,(t+1)r(t)<+\infty. 
\end{equation}
For that purpose we introduce some constants that play an important role in the proof. 
We denote by $\Cpio=\|\,\Pio\,\|_\infty\vee 1$ and $\Cpi=\|\,\Pi\,\|_\infty\vee 1$, where as usual
$$
\|\,\Pio\,\|_\infty=\sup\limits_{f\in L^\infty:\, \|f\|_\infty=1} \|\Pio(f)\|_\infty\ ; \ \|\,\Pi\,\|_\infty=\sup\limits_{f\in L^\infty:\, \|f\|_\infty=1} \|\Pi(f)\|_\infty,
$$
are the norms of $\Pio$ and $\Pi$ on $\MB$. We note that $\Cpio\le \Cpi+1$.

In what follows, for every $\e>0$, we denote by $t(\e)>0$ a finite
number such that
\begin{equation}
\label{eq:t(e)}    
\forall t\ge t(\e):\quad  r(t)\le \e \ \hbox{ and }\   \|\Psi(t,\sbullet)\|_\infty\le \e.
\end{equation}
We introduce the semigroup
$$
\mathfrak{S}_t=P_{t}\,\Pi^{\perp},
$$
which from Assumption HP2) satisfies that for $f\in \CB$,
\begin{equation}
\label{eq:gap for S}
\|\mathfrak{S}_t(f)\|_\infty \le H e^{-\lambda_1 t} \|f\|_\infty.
\end{equation}
Using the definitions of $r$ and $\Psi$, Equations \eqref{eq:u0} and  \eqref{eq:eq:asym survival critical II} and the commutation between $P_{t}$ and $\Pio$, we can easily derive the following  equations\begin{align}
&\frac{dr}{dt}= -B\,r^{2}-r\,\mu_0\Big(2\,\Theta_0\, b\, \Psi(t,\sbullet)\Big)
-\mu_0\Big(  b \Psi^2(t,\sbullet)\Big), \label{eq:for r}\\
&\Psi(t,x)=\mathfrak{S}_{t-t_{0}}\left(\,\Psi(t_{0},\sbullet)\right)(x)
-\int_{0}^{t-t_{0}} \mathfrak{S}_{s}\Big(\Pio\Big(b(\sbullet) u_0^2(t-s,\sbullet)\Big)\Big)(x)\,ds,\label{eq:for Psi}
\end{align}
for any $t_{0}\in [0,t[$ and 
where as above $B=\int \Theta_0^2(y) b(y) \mu_0(dy)$. 

The first estimation we need is a lower bound for $r$.

\begin{lemma}
\label{lem:lower bound surv proba}
There exists a finite constant $0<C_-<+\infty$, so that for all $t\ge 0$, 
$$
r(t)\ge \frac{C_-}{t+1}.
$$
\end{lemma}
\begin{proof} From Cauchy-Schwartz' inequality, we have  
$$
\Px{N_t>0}\ge \frac{(\EE_x(N_t))^2}{\EE_x(N_t^2)}=\frac{(\EE_x(N_t))^2}{(t+1)\frac{\EE_x(N_t^2)}{t+1}}
$$
which implies that, see \eqref{eq:convergence moments critical}, 
$$
\liminf\limits_{t\to \infty} (t+1)u_0(t,x)\ge \frac{\Theta_0(x)}{2B},
$$
and from Fatou's lemma
$$
\liminf\limits_{t\to \infty}\,\, (t+1)r(t)\ge \frac{1}{2B}\int \Theta_0(y) \mu_{0}(dy) =\frac{1}{2B}.
$$
Hence, there exists a $\,t_0>0\,$ so that $r(t)\ge \frac{1}{2B(t+1)}$ for all $t\ge t_0$. On the other hand
for $0\le t\le t_0$, it holds 
$$
r(t)\ge r(t_0)\ge \frac{1}{2B(t_0+1)}\ge \frac{1}{2B(t+1)}\frac{t+1}{t_0+1}\ge \frac{C_-}{(t+1)},
$$
where $C_-=\frac{1}{2B(t_0+1)}$.
\end{proof}
\subsubsection{A trapping set.}

We shall prove that the functions $r(t),\|\Psi(t,\sbullet)\|_\infty$
satisfy for all large $t$ the inequality 
$$
\|\Psi(t,\sbullet)\|_\infty\le \Gamma r(t)^2,
$$
for some finite constant $\Gamma$.
For that purpose we define in advance  constants that play
 an important role in proving this inequality:
$$
\def\arraystretch{1.4}
\begin{array}{l}
0<\beta= \lambda_1\wedge 1,\quad C_L=H\vee 1, \quad K=1+C_L, \quad \alpha= 2\, (1\vee b^*) \Cpio \Big(\|\Theta_{0}\|_\infty^{2}+\Cpio\Big),\\
\e_{1}=\frac{\beta}{20\, C_L\,\alpha\,K}\wedge 1,\quad \delta=\frac{\beta}{20\, C_L\,\alpha\,K}, \quad \tau=\,\frac{\log(C_L\,4K)}{\beta}.
\end{array}
$$
In particular we have $\|\mathfrak{S}_t\|_\infty\le C_L e^{-\beta t}$ (see \eqref{eq:gap for S}).
\begin{lemma}
\label{expodec}
Let $0<\e<\e_{1}$ and $t(\e)$ defined in \eqref{eq:t(e)}. Assume that for some $t_{0}\ge t(\e)$,
$\|\Psi(t_{0},\sbullet)\|_{\infty}\ge \frac{1}{\delta}\,r(t_{0})^{2}$. Then 
$$
\|\Psi(t_{0}+\tau,\sbullet)\|_{\infty}\le \frac{1}{2\,K}\;\|\Psi(t_{0},\sbullet)\|_{\infty} \;.
$$
Moreover 
$$
\sup_{t_{0}\le t\le t_{0}+\tau}\|\Psi(t,\sbullet)\|_{\infty}\le 
K\; \|\Psi(t_{0},\sbullet)\|_{\infty}\;.
$$
\end{lemma}
\begin{proof}
For $t\ge t_{0}$ we have $r(t)\le r(t_{0})$ and
$$
\|\Psi(t,\sbullet)\|_{\infty}\le \Cpio\,\|u_0(t,\sbullet)\|_{\infty}
\le  \Cpio\,\|u_0(t_{0},\sbullet)\|_{\infty}
\le \Cpio\,\big(
\|\Psi(t_{0},\sbullet)\|_{\infty}+r(t_{0})\, \|\Theta_0\|_\infty\big)\;.
$$
Hence, from the definition of $t(\e)$, see \eqref{eq:t(e)}, we have
$\|\Psi(t_{0},\sbullet)\|_\infty\le \e\le \Cpio \e$ and then we deduce from the assumptions that
$$
\def\arraystretch{1.6}
\begin{array}{ll}
\big\|\Pio\big((b(\sbullet)\,u_0^{2}(t,\sbullet)\big)\big\|_{\infty}
&\hspace{-0.2cm}\le b^*\, \Cpio\, \big\|u_0^{2}(t,\sbullet)\big\|_{\infty}
\le b^*\, \Cpio\, \big\|u_0^{2}(t_{0},\sbullet)\big\|_{\infty}
\le b^*\, \Cpio\,\big\|\left(r(t_{0})\,\Theta_{0}+\Psi(t_{0},\sbullet)\right)^{2}\big\|_{\infty}\\
&\hspace{-0.2cm}\le 2\, b^*\,
\Cpio\,\left(r^{2}(t_0)\,\|\Theta_0\|_\infty^{2} +\big\|\Psi(t_{0},\sbullet)\big\|^{2}_{\infty}\right)\le 2\, b^*\,
\Cpio\,\left(\delta\,\|\Theta_0\|_\infty^{2} \,\big\|\Psi(t_{0},\sbullet)\big\|_\infty+
\big\|\Psi(t_{0},\sbullet)\big\|^{2}_{\infty}\right)\\
&\hspace{-0.2cm}\le 2\, b^*\,\Cpio\,\big(\delta\,\|\Theta_0\|_\infty^{2}
+\Cpio\,\e\big)\,\,\big\|\Psi(t_{0},\sbullet)\big\|_{\infty}
\le \kappa\,(\delta+\e)\;.
\end{array}
$$
This implies that for $s\ge0$ (see \eqref{eq:for Psi}),
$$
\big\|\Psi(t_{0}+s,\sbullet)\big\|_{\infty}
\le C_L\, e^{-\beta s}\, \|\Psi(t_{0},\sbullet)\big\|_{\infty}
+C_L\; \frac{\kappa}{\beta}\; \big(\delta
+\e\big)\;\|\Psi(t_{0},\sbullet)\|_{\infty}\;.
$$
The results follow from the definitions of $\tau$, $\delta$, $K$, $\kappa$ and
$\e\le\e_{1}$ since
$$
C_L\, e^{-\beta \tau}\le \frac{1}{4\,K}, \hbox{ and } C_L\;\frac{\kappa}{\beta}\; \big(\delta+\e\big)\le \frac{1}{4\,K}<1,
$$
and for $0\le s\le \tau$, $\, C_L\, e^{-\beta s}\le C_{L}$.
\end{proof}

\begin{corollary}
\label{cor:entre}
For any $0<\e \le \e_{1}$ and any $t_0\ge t(\e)$
such that $\|\Psi(t_0,\sbullet)\|_{\infty} > \frac{1}{\delta}\,r(t_0)^{2}$,
we have
$$
\inf\left\{t>t_{0}:\,  \|\Psi(t,\sbullet)\|_{\infty} \le \frac{1}{\delta}\,r(t)^{2}\right\}<\infty.
$$
\end{corollary}
\begin{proof} Suppose not, then for all $t>t_0$ it holds
$$
r(t)< \sqrt{\delta \,\|\Psi(t,\sbullet)\|_{\infty}}\,.
$$
So, we can apply the previous Lemma to $t_0$ to conclude that for some $\tau\ge 0$,
$$
r(t_0+\tau)< \sqrt{\delta \,\|\Psi(t_0+\tau,\sbullet)\|_{\infty}} \le \gamma \sqrt{\delta \,\|\Psi(t_0,\sbullet)\|_{\infty}}\,,
$$
with $\gamma=\sqrt{1/(2K)}<1$. Inductively, we can apply Lemma \ref{expodec} to $t_0+n\tau$, for $n\ge 1$, to get
$$
r(t_0+n\tau)\le \gamma^n \sqrt{\delta \,\|\Psi(t_0,\sbullet)\|_{\infty}}\,,
$$
which decays exponentially fast, contradicting the lower bound
from Lemma \ref{lem:lower bound surv proba}.
\end{proof}
Now, we prove the crucial step to show the upper bound.
\begin{proposition}
\label{pro:borneGamma}
There exist constants $\Gamma<+\infty, 0<\e_2\le \e_1$, such that if for some $t_{1}>t(\varepsilon_{2})$ we have $r(t_{1})<\e_{2}$ and
$$
\|\Psi(t_{1},\sbullet)\|_{\infty}\le \frac{r(t_{1})^{2}}{\delta}\;,
$$
then for any $t\ge t_{1}$ we have
$$
\|\Psi(t,\sbullet)\|_{\infty}\le \Gamma\;r(t)^{2}\;.
$$
\end{proposition}
\begin{proof} 
Define the constants
\begin{align*}
&\Gamma=\left(17\,\frac{C_L}{\delta}+\frac{257}
{\beta}\,C_L\,\Cpio\;b^*\,\big(\|\Theta_{0}\|_{\infty}+1\big)^{2}\right)\vee 2\,,\\
&\e_{2}=\e_{1}\wedge\frac{1}{2\Gamma+5(B\vee 1)/\beta+3\,\left(\Gamma\, \mu_0(2\Theta_O b)+
\Gamma^{2}\;\mu_0(b)\right)/B}\,.
\end{align*}

Let $t^*=\inf\{t\ge t_1:\, \|\Psi(t,\sbullet)\|_\infty> \Gamma\,
r(t)^2\}$. We proceed by contradiction assuming $t^{*}<\infty$.

Since $\Gamma > 1/\delta$, we
have $t^*>t_1$ and by continuity we obtain that
$$
\|\Psi(t,\sbullet)\|_\infty\le \Gamma\, r(t)^2: \quad t\in[t_1,t^*].
$$
For $t\in [t_1, t^*]$,  $r(t)\le r(t_1)\le \e_2$. Further, we get
$$
\begin{array}{ll}
\left|r(t) \mu_0\Big(2\Theta_0 b \Psi(t,\sbullet)\Big)+\mu_0\Big(b\Psi^2(t,\sbullet)\Big)\right|&\hspace{-0.2cm}\le
\Gamma\,r^{3}(t)\,\mu_0(2\Theta_0 b)+\Gamma^{2}\,r^{4}(t)\;\mu_0(b)\\
&\hspace{-0.2cm}\le r^{2}(t)\left(\e_{2}\Gamma\,\mu_0(2\Theta_0 b)
+\e_{2}^{2}\Gamma^{2} \mu_0(b)\right)\le \frac{B}{2}\; r^{2}(t)\;.
\end{array}
$$
Therefore, using the equation \eqref{eq:for r}, we obtain
$$
-2\,B\;r^{2}(t)\le \frac{dr}{dt}\le -\frac{B}{2}\;r^{2}(t)\;,
$$
and then for all $t\in [t_1,t^*]$, it holds
\begin{equation}
\label{eq:bound r II}    
\frac{r(t_{1})}{1+2 {B}\, (t-t_{1})\,r(t_{1})}\le r(t)\le 
\frac{r(t_{1})}{1+ {B}(t-t_{1})\,r(t_{1})/2}\;.
\end{equation}
On the other hand, since $\Gamma \e_2\le 1$, we get for all $t\in [t_1,t^*]$,
$$
\|u_0(t,\sbullet)\|_{\infty}\le r(t)\;\|\Theta_{0}\|_{\infty}+
\Gamma\,r^{2}(t)\le r(t)\;\big(\|\Theta_{0}\|_{\infty}+1\big)\;,
$$
which together with the integral equation for $\Psi$ (see \eqref{eq:for Psi}, with $t_{0}=t_{1}$) gives the bound
\begin{equation}
\label{eq:bound Psi II}
\|\Psi(t,\sbullet)\|_{\infty}\le \frac{C_{L}}{\delta}\,e^{-\beta(t-t_{1})}\,r(t_{1})^{2}+
C_{L}\,\Cpio\;b^* \left(\|\Theta_{0}\|_{\infty}+1\right)^{2}\int_{0}^{t-t_{1}}\,e^{-\beta \,s}
\;r(t-s)^{2}\,ds.
\end{equation}
The lower bound for $r$ in \eqref{eq:bound r II} and the facts 
$\,2{{B}}\e_2\le \beta\,$ and $\,(1+x)^2\le 2e^x\,$ for $x\ge 0$, 
gives
$$
\begin{array}{ll}
\frac{C_{L}}{\delta} \;e^{-\beta(t-t_{1})}\,r(t_{1})^{2}&\hspace{-0.2cm}\le \frac{C_{L}}{\delta} 
\;e^{-\beta(t-t_{1})}\,\big(1+2{B}\, (t-t_{1})\,r(t_{1})\big)^{2}\,r(t)^{2}\\
&\hspace{-0.2cm}\le \frac{C_{L}}{\delta}
\;e^{-\beta(t-t_{1})}\,\big(1+\beta\,(t-t_{1})\big)^{2}\,
r(t)^{2}\le \frac{2\,C_{L}}{\delta} r(t)^{2}\le \frac{\Gamma}{4}\;r(t)^{2}\;.
\end{array}
$$
Consider 
$G=C_{L}\,\Cpio\;b^* \left(\|\Theta_{0}\|_{\infty}+1\right)^{2}$. Now, we split the integral in \eqref{eq:bound Psi II} in two pieces.
Using that $2{{B}}\e_2\le \beta/2$ and 
${{\frac{8G}{\beta}\le \Gamma}}$,
one obtains for $ t_1\le t\le t^{*}$
$$
G\;\int_{(t-t_{1})/2}^{t-t_{1}}e^{-\beta\,s}
r(t-s)^{2}\,ds\le \frac{G}{\beta}\;
e^{-\beta\,(t-t_{1})/2}r(t_{1})^{2}
\le \frac{G}{\beta} e^{-\beta\,(t-t_{1})/2}
\big(1+2 {B}\, (t-t_{1})\,r(t_{1})\big)^{2}\;r(t)^{2}
\le \frac{2\, G}{\beta}\;r(t)^{2}\le \frac{\Gamma}{4}\;r(t)^{2}\;.
$$
Finally, using the upper and the lower bounds on $r(t)$ given in \eqref{eq:bound r II}, 
we have for the other part of the integral
$$
\begin{array}{l}
G\;\int_{0}^{(t-t_{1})/2}e^{-\beta\,s} r^{2}(t-s)\,ds\le \frac{G}{\beta}\;r^2(t_1+(t-t_{1})/2)^{2}
\le \frac{G}{\beta}\;\left(\frac{r(t_{1})}{1+{{B}}(t-t_{1})\,r(t_{1})/4}\right)^{2}
\le \frac{G}{\beta}\left(\frac{1+2{{B}}\,(t-t_{1})\,r(t_{1})}
{1+{{B}}(t-t_{1})\,r(t_{1})/4}\right)^{2}\;r(t)^{2}\\
\\
\le \frac{64\, G}{\beta}\;r(t)^{2}\le \frac{\Gamma}{4}\;r(t)^{2}\;,
\end{array}
$$
where we have used that ${{\frac{256\, G}{\beta}\le \Gamma}}$.
Then, taking $t=t^*$ gives that
$$
\Gamma \,r^{2}(t^*)=\|\psi(t^*,\bullet)\|_{\infty}\le 
\frac{3}{4}\;\Gamma \,r^{2}(t^*)\;,
$$
a contradiction since $r(t^*)>0$.
\end{proof}
Now, we prove the required upper bound.
\begin{proposition}
\label{pro:upper bound}
There exists $t_{2}>0$ such that for $t\ge t_{2}$,
$$
r(t)\le \frac{1}{B(t-t_{2})/2+1/r(t_{2})}\;,
$$
and
$$
\|\Psi(t,\bullet)\|_{\infty}\le \frac{\Gamma}{\big(B\,
(t-t_{2})/2+1/r(t_{2})\big)^{2}}\;.
$$
\end{proposition}
\begin{proof}
Since $r$ is decreasing and tends to zero, we can find $t_{2}'>0$ such
that $r(t_{2}')<\e_{2}$. If 
$$
\|\Psi(t_{2}',\bullet)\|_{\infty}\le \frac{r(t_{2}')^{2}}{\delta}\;,
$$
we take $t_{2}=t_{2}'$.
Otherwise, we  define 
$t_{2}=\inf\left\{t>t_{2}'\,,\, \|\Psi(t,\bullet)\|_{\infty}
\le \frac{1}{\delta}\,r(t)^{2}\right\}$,
which is finite by Corollary \ref{cor:entre}. 

We now apply Proposition \ref{pro:borneGamma} using $r(t_{2})\le r(t_{2}')<\e_{2}\,$ and  
$\,\,\|\Psi(t_{2},\bullet)\|_{\infty}\le \frac{r(t_{2})^{2}}{\delta}$. For $t\ge t_{2}$ we have
$$
\frac{dr}{dt}\le  -B\,r^{2}+r^{3}\,\Gamma\,\mu_0(2 \Theta_O b)
=-(B-r\,\Gamma \mu_0(2\Theta_O b))\,r^{2}
\le 
-(B-\e_{2}\,\Gamma \mu_0(2\Theta_O b))\,r^{2}\le -\frac{B}{2} r^{2}(t)\;, 
$$
because the choice of $\e_{2}$ gives $\,\e_{2}\,\Gamma \mu_0(2\Theta_O b)\le \frac{B}{2}$,
and the first statement follows. 
The second statement follows from the first one and Proposition \ref{pro:borneGamma}.
\end{proof}
Now, we prove that $(t+1)r(t)$ converges, and we give an upper bound for the speed of convergence.

\begin{proposition}
For any $t\ge0 $ we have 
$$
\left|(1+t)\,r(t)-\frac{1}{B}\right|\le 
\frac{\Oun\,\log (2+ t)}{1+t}\;.
$$
Therefore
$$
\lim_{t\to\infty} (1+t)\,r(t)=\frac{1}{B}\;.
$$
\end{proposition}
\begin{proof}
It is enough to prove the result for $t$ large enough, namely we
assume $t\ge t_{2}$ (see Proposition \ref{pro:upper bound} for the
definition of $t_{2}$).  We have, using
Lemma \ref{lem:lower bound surv proba} and Proposition \ref{pro:upper bound}
$$
\frac{d}{dt}\frac{1}{r}=-\frac{1}{r^{2}}\;\frac{dr}{dt}=
B+\frac{1}{r}\,\mu_0\Big(2\Theta_0 b \Psi\Big)
+\frac{1}{r^{2}}\mu_0\Big(b\,\Psi^{2}\Big)=B+\frac{\Oun}{1+t}\;.
$$
Therefore
$$
\frac{1}{r(t)}-\frac{1}{r(1)}=B\,(t-1)+\Oun \log(2+t),
$$
namely
$$
r(t)=\frac{1}{B\,(t-1)+\Oun \log(2+t)+1/r(1)},
$$
and the result follows.
\end{proof}
This allows us to prove the main result of this section, which is Theorem \ref{the:critical} formula \eqref{eq:Uniform survival prob}.
\begin{corollary}
For any $t\ge0 $
$$
\sup\limits_{x\in \RR}\left|(1+t)\,u_0(t,x)-
\frac{1}{B}\Theta_{0}(x)\right|\le\frac{\Oun\log(2+t)}{1+t}\;. 
$$
\end{corollary}
\begin{proof}
This result follows from the definition of $\Psi$ and the previous result.
\end{proof}

\subsection{Proof of Theorem \ref{the:critical convergence} }
\label{sec:Laplace_critical}

We  prove the asymptotic conditional distribution given in Theorem \ref{the:critical convergence}. For $r,u\in \RR$, we consider $g=rf+u $. 
We have shown the following result
$$
\lim_{t \to +\infty} \EE_{\delta_x}\Bigg(\Bigg(\frac{\langle Z_t\, ,g\rangle}
{(t+1)BA}\Bigg)^n \,\, \Bigg |\,   N_t >0 \Bigg)
=n! (\nu(g))^n= (\nu(g))^n \EE(\xi^n),
$$
where $\xi$ is a $\E(1)$ random variable.
By multilinearity,  we obtain for all 
$p$ and $q$ nonnegative integers that
$$
\lim_{t \to +\infty} \EE_{\delta_x}\Bigg(\Bigg(\frac{\langle Z_t\, ,f\rangle}{(t+1)BA}\Bigg)^p
\Bigg(\frac{\langle Z_t\, ,1\rangle}{(t+1)BA}\Bigg)^q
\Bigg) \,\, \Bigg |\,   N_t >0 \Bigg)
= (\nu(f))^p\EE(\xi^{p+q}).
$$
The distribution of the random vector $\xi(\nu(f),1)$ is determined by 
its moments (cf \cite[Problem 30.3]{Billingsley}). 
By \cite[Problem 30.5]{Billingsley}, the random vector 
$\ \Bigg(\frac{\langle Z_t\, ,f\rangle}{(t+1)BA}, \frac{\langle Z_t\, ,1\rangle}{(t+1)BA}\Bigg)$ conditioned 
to non extinction  converges in law (as $t$ tends to infinity) to the random vector $\ (\nu(f) \xi, \xi)$. 

This finishes the proof of Part $(1)$. 

It is immediate from Part $(1)$ that the random variables $\frac{\langle Z_t\, ,f\rangle}{\langle Z_t\, ,1\rangle}$ 
conditioned on non extinction converge in law to the random variable $\frac{\nu(f) \xi}{\xi} =\nu(f)$ and that proves Part $(2)$.
 
Recall that
$$
\begin{array}{ll}
\Upsilon_t&\hspace{-0.2cm}=\frac{1}{\nu(f)(AB)^{1/2}} \frac{\frac{\langle Z_t\, ,f\rangle}{t+1}-\nu(f)AB/2}{\sqrt{\frac{N_t}{t+1}}}
=\frac{1}{\nu(f)(AB)^{1/2}}\frac{\langle Z_t\, ,f\rangle-(t+1)\nu(f)AB/2}{\sqrt{(t+1) N_t}}\\
&\hspace{-0.2cm}=\frac{1}{\nu(f)(AB)^{1/2}}\sqrt{(t+1)N_t} \, \frac{\langle Z_t\, ,f\rangle-(t+1)\nu(f)AB/2}{(t+1) N_t}
=\frac{1}{\nu(f)(AB)^{1/2}}\left(\sqrt{\frac{N_t}{t+1}} 
\left(\frac{\langle Z_t\, ,f\rangle}{N_t}\right)-\sqrt{\frac{t+1}{N_t}}\nu(f)AB/2\right).
\end{array}
$$
By Theorem  \ref{the:critical convergence} part $(1)$, it   converges in distribution 
(for $\Px{\hspace{0.1cm}\vsbullet\hspace{0.1cm} | \, N_t>0}$) to 
$$
\Upsilon=\frac{1}{\nu(f)(AB)^{1/2}}\left((\xi AB)^{1/2}\nu(f) - (\xi AB)^{-1/2} \nu(f)AB/2\right)=\xi^{1/2}-\frac12 \xi^{-1/2},
$$
where $\xi$ is $\E(1)$. Part $(3)$ is proved.

\section{Subcritical case $\eta=-\lambda_0<0$: Proof of Proposition \ref{pro:subcritical estimation}  and Theorem
  \ref{the:subcritical}}
\label{sec:proof_subcritical}

\subsection{Proof of Proposition \ref{pro:subcritical estimation}}

Recall that $\un(t,x)=\Ex{\langle Z_t\, ,1\rangle^n}$ and $v_n^-(t,x)=e^{\lambda_0 t}\un(t,x)$.
The equation satisfied by $v_n^-$ is given by $v_1^-(t,x)=e^{\lambda_0 t}u_1(t,x)=e^{\lambda_0 t}P_t(\ind)(x)$
and for $n\ge 2$,
\begin{equation}
\label{eq:moment_subcritical}
\def\arraystretch{1.4}
\begin{array}{ll}
v_n^-(t,x)&\hspace{-0.2cm}=e^{\lambda_0 t}u_1(t,x)+
\int_0^t e^{-\lambda_0(t-s)} e^{\lambda_0 s}P_s\left(b(\sbullet)\sum\limits_{k=1}^{n-1} \binom{n}{k} 
v_{n-k}^-(t-s,\sbullet)v_{k}^-(t-s,\sbullet)\right)\!(x)\, ds\\
&\hspace{-0.2cm}=v_1^-(t,x)+I_t(x).
\end{array}
\end{equation}
Then using induction on $n$, we show that 
$$
C^-_n=\sup\limits_{t,x} v_n^-(t,x)<\infty.
$$
This follows from the recursive inequality 
\begin{equation}
\label{eq:recursion Cn-}
\begin{array}{ll}
C^-_n&\hspace{-0.2cm}\le C^-_1+\int_0^t e^{-\lambda_0(t-s)} e^{\lambda_0 s}
P_s\left(b(\sbullet)\sum\limits_{k=1}^{n-1} \binom{n}{k} v_{n-k}^-(t-s,\sbullet)v_{k}^-(t-s,\sbullet)\right)(x)\,ds\\
&\hspace{-0.2cm}\le C^-_1 +b^* \sum\limits_{k=1}^{n-1} \binom{n}{k} C^-_{k}C^-_{n-k}
\int_0^t e^{-\lambda_0(t-s)} e^{\lambda_0 s}P_s(\ind)(x) ds\le
C^-_1 +C^-_1 b^*\sum\limits_{k=1}^{n-1} \binom{n}{k} C^-_{k}C^-_{n-k}\int_0^t e^{-\lambda_0(t-s)} ds\\
&\hspace{-0.2cm}\le C^-_1+ C^-_1 b^*\lambda_0^{-1}\sum\limits_{k=1}^{n-1} \binom{n}{k}  C^-_{k}C^-_{n-k}\,,
\end{array}
\end{equation}
and $C^-_1\le H+\int \mu_0(dy) \, \|\Theta_0\|_\infty$\,, by Assumption HP2) which states that
$$
F_1=\sup\limits_{t,x} e^{(\lambda_1-\lambda_0)t}\left|v_1^-(t,x)-\Theta_0(x)V^-_1\right|\le H,
$$
where $V^-_1=\int \mu_0(dy)$.

\subsection{Proof of Theorem
  \ref{the:subcritical} (1) - Convergence of moments.}
\label{sec:moments_subcritical}

Assume inductively that $v_k^-(t,x)$, converges, when $t\to \infty$
uniformly on $x$ (exponentially fast) to $V_k^-(x)=\Theta_0(x) V^-_k$, for some constant
$V^-_k$, for $k=1,...,n-1$ with $n\ge 2$. Then, we shall prove this property
for $n$ and provide a recursion for $V^-_n$. For $n=1$, the result is true. Moreover, we shall prove that
$$
F_{n}=\sup\limits_{t,x} e^{\frac{\lambda_0}{\lambda_1}(\lambda_1-\lambda_0)t}\,|v_n^-(t,x)-\Theta_0(x) V^-_n|<\infty.
$$

For $0<M<t$, we decompose the integral
in \eqref{eq:moment_subcritical} as
$I_t(x)=J_{t,M}(x)+K_{t,M}(x)
$, where
$$
\begin{array}{ll}
J_{t,M}(x)&\hspace{-0.2cm}=\int_{t-M}^t e^{-\lambda_0(t-s)} e^{\lambda_0 s}P_s\left(b(\sbullet)\sum\limits_{k=1}^{n-1} \binom{n}{k} 
v_{n-k}^-(t-s,\sbullet)v_{k}^-(t-s,\sbullet)\right)\!(x)\, ds\\
&\hspace{-0.2cm}=\Theta_0(x)\sum\limits_{k=1}^{n-1} \binom{n}{k} 
\int_0^\infty e^{-\lambda_0 r} \int v_{n-k}^-(r,y)v_{k}^-(r,y) b(y) \mu_{0}(dy)\, dr+S_{t,M}(x)
\end{array}
$$
with
$$
\begin{array}{ll}
|S_{t,M}(x)|&\hspace{-0.2cm}\le \int_0^M e^{-\lambda_0 r} \sum\limits_{k=1}^{n-1} \binom{n}{k} 
\left|e^{\lambda_0 (t-r)} P_{t-r}\left(b(\sbullet) v_{n-k}^-(r,\sbullet)v_{k}^-(r,\sbullet)\right)(x)- \Theta_0(x) 
\int v_{n-k}^-(r,y)v_{k}^-(r,y) b(y)\mu_0(dy) \right |dr\\
&\hspace{0.2cm}+\Theta_0(x)\sum\limits_{k=1}^{n-1} \binom{n}{k}  
\int_M^\infty e^{-\lambda_0 r} \int v_{n-k}^-(r,y)v_{k}^-(r,y) b(y)\mu_0(dy) dr\\
&\hspace{-0.2cm}\le \Big[ H e^{-(\lambda_1-\lambda_0)(t-M)} 
+\|\Theta_0\|_\infty \int \mu_0(dy)\,  e^{-\lambda_0 M}\Big]
\,\left( b^*\lambda_0^{-1} \sum\limits_{k=1}^{n-1} \binom{n}{k} C^-_{n-k}C^-_{k} \right)\\
&\hspace{-0.2cm}=\zeta_n e^{-(\lambda_1-\lambda_0)(t-M)}+ \rho_n  e^{-\lambda_0 M}.
\end{array}
$$
On the other hand,
$$
|K_{t,M}(x)|\le \int_0^{t-M}\!\!\! e^{-\lambda_0(t-s)} e^{\lambda_0 s} P_s(\ind)(x) ds \,\, 
b^*\sum\limits_{k=1}^{n-1}  \binom{n}{k} C^-_{n-k}C^-_{k} 
\le C_1^- \lambda_0^{-1} b^*\sum\limits_{k=1}^{n-1}  \binom{n}{k} C^-_{n-k}C^-_{k} 
\, e^{-\lambda_0 M} =\gamma_n \, e^{-\lambda_0 M} .
$$
Define 
$$
V^-_n=V^-_1+\sum\limits_{k=1}^{n-1} \binom{n}{k} 
\int_0^\infty e^{-\lambda_0 r} \int v_{n-k}^-(r,y)\,v_{k}^-(r,y)\, b(y)\, \mu_0(dy)\, dr,
$$
to obtain that
\begin{align*}
&|v_n^-(t,x)-\Theta_0(x)V^-_n| \\
& \le    |v_1^-(t,x)-\Theta_0(x)V^-_1| + |J_{t,M}(x)- \Theta_0(x)\sum\limits_{k=1}^{n-1} \binom{n}{k} 
\int_0^\infty e^{-\lambda_0 r} \int v_{n-k}^-(r,y)\,v_{k}^-(r,y)\, b(y)\, \mu_0(dy)\, dr| +    |K_{t,M}(x)|\\
& \le  |v_1^-(t,x)-\Theta_0(x)V^-_1| + |S_{t,M}(x)| +    |K_{t,M}(x)|\\
& \le (H e^{-(\lambda_1-\lambda_0)t}+\zeta_n e^{-(\lambda_1-\lambda_0)(t-M)}+
(\rho_n+\gamma_n) e^{-\lambda_0 M})\\
& \le  \varsigma_n \left(e^{-(\lambda_1-\lambda_0)(t-M)}+e^{-\lambda_0 M}\right),
\end{align*}
with $\varsigma_n=4\Big(H\vee \zeta_n\vee \rho_n\vee \gamma_n\Big)$.
Optimizing on $M\in (0,t)$, namely  with $M=\frac{\lambda_1-\lambda_0}{\lambda_1}t$, we obtain that
$$
F_n=\sup\limits_{t,x} e^{\frac{\lambda_0}{\lambda_1}(\lambda_1-\lambda_0)t}|e^{\lambda_0 t}
\Ex{\langle Z_t\, ,1\rangle^n}-\Theta_0(x)V^-_n|\le 2\varsigma_n.
$$
The result follows.

\subsection{Proof of  Theorem
  \ref{the:subcritical} (2) -The survival probability.}
\label{sec:surv_prob__subcritical}

Consider $v_0^-(t,x)=e^{\lambda_0 t}u_0(t,x)$, the rescaled survival probability. This 
function satisfies the equation
$$
v_0^-(t,x)=e^{\lambda_0 t}u_1(t,x)-\int_0^\infty e^{-\lambda_0(t-s)} e^{\lambda_0 s} P_s\left(b(\sbullet) \, (v_0^-(t-s,\sbullet))^2\right)(x)\, ds.
$$
So, we have
$$
\sup_{t,x} v_0^-(t,x)\le \sup_{t,x} e^{\lambda_0 t}u_1(t,x)=C^-_1<+\infty.
$$
On the other hand, consider
$$
\mathfrak{K}^-=\int \mu_0(dy)-\int_0^\infty e^{-\lambda_0 r} \int  b(y)\, (v_0^-(r,y))^2\,
\mu_0(dy),
$$
and let us study
$$
\def\arraystretch{1.4}
\begin{array}{ll}
\left|v_0^-(t,x)-\Theta_0(x)\mathfrak{K}^-\right|&\hspace{-0.2cm}\le \left|e^{\lambda_0 t}u_1(t,x)-\Theta_0(x)A\right|+
\Theta_0(x)\int_t^\infty e^{-\lambda_0 r}\int b(y) (v_0^-(r,y))^2 \, \mu_0(dy) \,dr\\
&\hspace{0cm}+\int_0^t e^{-\lambda_0 r} \left| e^{\lambda_0 (t-r)} P_{t-r}\left( b(\sbullet) (v_0^-(r,\sbullet))^2\right)(x)-
\Theta_0(x) \int b(y) (v_0^-(r,y))^2 \, \mu_0(dy) \right|\, dr\\
&\hspace{-0.2cm}\le H\,C^-_1 e^{-(\lambda_1-\lambda_0)t} 
+\frac{\|\Theta_0\|_\infty b^*(C^-_1)^2 A}{\lambda_0} e^{-\lambda_0 t}
+ H\,b^*\,(C^-_1)^2\int_0^t e^{-\lambda_0 r} e^{-(\lambda_1-\lambda_0)(t-r)} dr\\
&\hspace{-0.2cm}\le C\left(1+t \ind_{\lambda_1=2\lambda_0}\right) e^{-\beta t},
\end{array}
$$
where $C$ is a suitable constant and $\beta=\lambda_0\wedge (\lambda_1-\lambda_0)$.

The fact that $\mathfrak{K}^->0$, follows from the Cauchy-Schwartz inequality and the convergence of the normalized
moments. Indeed,
$$
\liminf\limits_{t\to \infty} e^{\lambda_0 t}\, \Px{N_t>0} \ge  \lim\limits_{t\to \infty}  
e^{\lambda_0 t}\,  \frac{\left(\Ex{N_{t}}\right)^2}{\Ex{N_{t}^2}}= \lim\limits_{t\to \infty} \frac{\left(v_1^-(t,x)\right)^2}{v^-_2(t,x)}
=\Theta_0(x)\frac{(V^-_1)^2}{V^-_2},
$$
and then
$$
\mathfrak{K}^-\ge \frac{(V^-_1)^2}{V^-_2}>0.
$$

\subsection{Proof of Theorem
  \ref{the:subcritical} (3) - Convergence in Law.}
\label{sec:Convergence in law subcritical}

Let $f\in \CB$. In this section we show the convergence of the law for 
the random variable
$\langle Z_t\, ,f\rangle$,
under the conditional  probability $\Px{\hspace{0.3cm} \bullet\hspace{0.3cm} \Big | \, \langle Z_t\, ,1\rangle > 0}$. As  in the previous section,
$$
C^-_n(f)=\sup\limits_{t,x} |v_n^-(f)(t,x)|\le \|f\|_\infty^n C^-_n
$$
satisfies the recursive bound
$$
C^-_n(f)\le C^-_1(f^n)+ C^-_1b^*\lambda_0^{-1}\sum\limits_{k=1}^{n-1} \binom{n}{k} C^-_{n-k}(f)C^-_k(f).
$$
Also, we have
$$
F^-_n(f)=\sup\limits_{t,x} e^{\frac{\lambda_0}{\lambda_1}(\lambda_1-\lambda_0)t} |v_n^-(f)(t,x)-\Theta_0(x) V^-_n(f)|<\infty,
$$
and in particular, for all $x$, $|\Theta_0(x)V^-_n(f)|\le C^-_n(f) \le \|f\|_\infty^n C^-_n$. Since $V^-_n(f)$
does not depend on $x$, we obtain the bound
\begin{equation}
\label{eq:bound Vn-}
|V^-_n(f)|\le \frac{\|f\|_\infty^n}{\|\Theta_0\|_\infty} C^-_n.
\end{equation}
Then, we have the limits, for all $n\ge 1$
$$
\Ex{\langle Z_t\, ,f\rangle^n \,\, \Big | \,\,\langle Z_t\, ,1\rangle >0}=
\frac{\Ex{\langle Z_t\, ,f\rangle^n}}{\Px{\langle Z_t\, ,1\rangle >0}}=
\frac{e^{\lambda_0 t}\Ex{\langle Z_t\, ,f\rangle^n}}
{e^{\lambda_0 t}\Px{\langle Z_t\, ,1\rangle >0}}\longrightarrow \frac{V^-_n(f)}{\mathfrak{K}^-}.
$$
The uniform bounds for the first conditional moment shows the tightness of the measures (in $(\RR,\B)$): 
$$
\mu_t(A)=\Ex{\langle Z_t\, ,f\rangle \in A \,\, \Big | \,\,\langle Z_t\, ,1\rangle >0}.
$$
Then, for every sequence $t_n\uparrow \infty$, there exists a subsequence $(t_{n'})_{n'}$ and a probability measure $\zeta$,
that depends on $f$, on the sequence $(t_{n'})_{n'}$ and the initial condition $x$, defined 
on $(\RR_+,\B)$ so that
$$
\mu_{t_{n'}} \stackarrow{\mathcal{L}}{n'\to \infty}{1} \zeta(dy)
$$
Using the uniform convergence of moments we get, by induction, that for all $n\ge 1$
$$
\int_\RR y^n \, \zeta(dy)=\frac{V^-_n(f)}{\mathfrak{K}^-}
$$
So, if we can prove that $\zeta$ is uniquely determined by its moments, then we would get the convergence in law and 
the limiting distribution $\zeta$ depends only on $f$.

\begin{lemma} 
\label{lem:bounds moments subcritial}
Assume that $(a_n)_{n\ge 1}$ is a positive sequence that satisfies for
some $\eta>0$ and any
$n\ge 2$ the inequality
\begin{equation}
\label{eq.recursion (an)}
a_n\le a_1+\eta \sum\limits_{k=1}^{n-1} \binom{n}{k}\, a_{n-k}\,a_k.   
\end{equation}
Then, there exists $0<r_*=r_{*}(a_{1},\eta)$ such that for any $n\ge2$,
$$
a_n\le \frac{n!}{2\,\eta\,r_{*}^{n}}\;.
$$
\end{lemma}
\begin{proof} Let $(b_{n})$ be the sequence defined recursively by
$b_{0}=0$, $b_{1}=a_{1}$ and for $n\ge2$
$$
b_{n}=\frac{a_{1}}{n!}+\eta\sum\limits_{k=1}^{n-1} b_{n-k}\,b_k\;. 
$$
It is easy to verify recursively that this sequence is positive.

For any integer $N\ge 2$ we define the polynomial
$$
B^{N}(u)=\sum_{k=0}^{N}b_{k}\;u^{k}\;.
$$
It follows from the recursion relation satisfied by $(b_k)_k$ that, for $u>0, N\ge 1$, we have $B^N(u)>0$ and
\begin{equation}
\label{eq:cuadratica}
B^{N}(u)< a_{1}\left(e^{u}-1\right)+\eta\; B^{N}(u)^{2}\;.
\end{equation}
Let $r_*=r_{*}(a_{1},\eta)>0$ be defined by
$$
r_{*}=\log\big(1+1/(4\eta\,a_{1})\big)\;.
$$
For $0\le r< r_{*}$, the equation
$$
\eta \,z^{2}-z+a_{1}\left(e^{r}-1\right)=0
$$
has two nonnegative distinct roots $s_1(r)<s_2(r)$, given by
$$
s_1(r)= \frac{1-\sqrt{1-4\, \eta\,a_{1}\left(e^{r}-1\right)}}{2\eta}\le \frac{1}{2\eta}<s_2(r)= 
\frac{1+\sqrt{1-4\, \eta\,a_{1}\left(e^{r}-1\right)}}{2\eta}.
$$
Since $B_N(0)=0$ and $B_N(u)$ satisfies the strict inequality \eqref{eq:cuadratica}, for $u>0$, we should have
that $B_N(u)<s_1(u)$ or $B_N(u)>s_2(u)$ for $0<u<r_{*}$. The second case cannot happen for small $u$. Thus, by continuity,
we prove that, for all $0<u<r_*$,
$$
b_N u^N\le B^{N}(u)\le \frac{1-\sqrt{1-4\, \eta\,a_{1}\left(e^{u}-1\right)}}{2\eta}\le \frac{1}{2\eta},
$$
showing that 
$$
b_{N}\le \frac{1}{2\eta \,r_{*}^{N}}\;.
$$
Since $a_{1}=b_{1}$, it is easy to check recursively that for
any $k\ge 2$,
$$
a_{k}\le b_{k}\,k!,
$$
and the lemma is shown.
\end{proof}

Recall that $V^-_n(f)\le \frac{\|f\|_\infty^n}{\|\Theta_0\|_\infty} C^-_n$. So, using the previous lemma with
$(a_n)_n= (C^-_n)_n$, which satisfies the recursion \eqref {eq.recursion (an)} (see \eqref{eq:recursion Cn-}) 
with $a_1=C^-_1$ and $\eta=C^-_1b^*\lambda_0^{-1}>0$, 
shows that
$$
\big|V^-_n(f)\big|\le \frac{\|f\|_\infty^n}{\|\Theta_0\|_\infty}  
\;\frac{n!}{2\,\eta\,r_{*}^{n}}\,,
$$
for $r_*=r_*(C^-_1,\eta)>0$.
This ensures that the Hamburger moment problem for the sequence
$\big(V^-_n/\mathfrak{K}^-\big)_n$ has a unique solution (see, for example,
\cite{Reed-Simon} pages 145, 205).

This implies that the (tight) family of random variables $(\langle
Z_t\, ,f\rangle)_{t\ge0}$ converges in law,  when conditioned on survival and 
$t\to\infty$, to a non-trivial limit law denoted by $\zeta_f$.

According to \cite{dawsonMaisonneuveSpencer} Theorem 3.2.6 and Corollary 3.2.7, there exists a unique probability
measure $\Lambda^-$, over $M_F(\RR)$, so that conditioned to non extinction, the process
$Z_t$ converges in law to $\Lambda^-$, independent of
the initial condition $\delta_x$. That is, for all nonnegative $f\in \CB$, we have
$$
\lim\limits_{t\to \infty}\Ex{\exp{(-\langle Z_t\, ,f\rangle)} \,\, \Big | \,\,\langle Z_t\, ,1\rangle >0}
=\int_{M_F(\RR)} \exp{(-\langle \mu\,, f\rangle) }\, \Lambda^-(d\mu)=\int_{\RR} \exp{(-y)}\,  \zeta^-_f(dy),
$$
and $\Lambda_-$ is a Yaglom Limit for $(Z_t)_t$.

\section{Supercritical case $\eta=-\lambda_0>0$.}
\label{sec:proof_supercritical}

The proof of Theorem \ref{the:supercritical} is split into several subsections. 

\subsection{Convergence of the moments.}
\label{sec:moments_supercritical}

In this section, we study the asymptotic for the moments in the supercritical case. We prove 
Proposition \ref{pro:supercritical estimation} and the asymptotic given in Theorem \ref{the:supercritical} $(1)$. 
The correct scaling is $v_n^+(t,x)=e^{n\lambda_0 t}\un(t,x)$.
The equation satisfied by $v_n^+$ is given by: $v_1^+(t,x)=e^{\lambda_0 t}u_1(t,x)=e^{\lambda_0 t}P_t(1)(x)$
and for $n\ge 2$,
\begin{align}
\label{eq:moment_supercritical}
v_n^+(t,x)&=e^{n\lambda_0 t}u_1(t,x)+
\int_0^t e^{n\lambda_0 s}P_s\left(b(\vsbullet)\sum\limits_{k=1}^{n-1} \,\binom{n}{k} \,
v_{n-k}^+(t-s,\vsbullet)v_{k}^+(t-s,\vsbullet)\right)\!(x)\, ds\nonumber \\
&=v_1^+(t,x)+I_t(x).
\end{align}
Notice that $e^{n\lambda_0 t}u_1(t,x)=e^{(n-1)\lambda_0 t}C^+_1\le C^+_1$, where 
\begin{equation}
\label{eq:C1+}
C^+_1=\sup\limits_{t,x} e^{\lambda_0 t}u_1(t,x)\le H+\|\Theta_0\|_\infty\int \mu_0(dy),
\end{equation}
by Hypothesis HP1), HP2).
Then, using induction on $n$, we show that 
\begin{equation}
\label{eq:Cn+}
C^+_n=\sup\limits_{t,x} v_n^+(t,x)<+\infty,
\end{equation}
satisfies the recursion: $C^+_1\le H+\|\Theta_0\|_\infty\int \mu_0(dy)$ and for $n\ge 2$,
$$
C^+_n\le C^+_1+C^+_1b^*\sum\limits_{k=1}^{n-1} \binom{n}{k}\, C^+_{n-k}C^+_{k} \int_0^\infty e^{(n-1)\lambda_0 s} ds
\le C^+_1+C^+_1b^*\left|\lambda_0^{-1}\right| \, \frac1{n-1}\sum\limits_{k=1}^{n-1} \binom{n}{k}\, C^+_{n-k}C^+_{k},
$$
from which it follows (proof by a recursion) that 
\begin{equation}
\label{eq:bound moments I}
C^+_n\le n! \,\mathcal{D}^{2n-1},\hbox{ with }
\mathcal{D}=2\, (C^+_1b^*|\lambda_0^{-1}|\vee C^+_1 \vee 1).
\end{equation}
Recall also that, from HP2),
$$
F^+_1:=\sup\limits_{t,x} e^{(\lambda_1-\lambda_0)t}\left | e^{\lambda_0 t} u_1(t,x)-\Theta_0(x) \int \mu_0(dy)\right|\le H<+\infty.
$$
In particular $V^+_1=\Theta_0(x) \int \mu_0(dy)$.

On the other hand, one can prove by induction that $v_n^+(t,x)$ converges uniformly to $V^+_n(x)$ which
are given by the following recursion. For $n\ge 2$,
\begin{align}
\label{eq:iteration moments supercritical}
V^+_n(x) & =\int_0^\infty e^{n\lambda_0 s} P_s\left(b(\sbullet)
\sum\limits_{k=1}^{n-1} \binom{n}{k}\, V^+_{n-k}(\sbullet) V^+_k(\sbullet)\right)(x)\, ds\le C^+_n,\nonumber \\
F^+_n & =\sup\limits_{t,x} e^{\,\beta_n t} \left|e^{\lambda_0 n t}u_n(t,x)-V_n^+(x)\right|<\infty, \nonumber \\
\beta_1 & =(\lambda_1-\lambda_0),\hbox{ and for } n\ge 2: \beta_n=\frac{\beta_{n-1}|\lambda_0|(n-1)}
{\beta_{n-1}+|\lambda_0|(n-1)}<\beta_{n-1}.
\end{align}
Indeed, we proceed by induction. Assume $n\ge 2$ and consider the integral $I_t(x)$ in
\eqref{eq:moment_supercritical}, where as before we take $0<M<t$ to be fixed,
and $0<\beta=\beta_{n-1}\le \beta_k: k=1,...,n-1$,
so that $F^+_k<\infty$, for all $k=1,...,n-1$ (defined in \eqref{eq:iteration moments supercritical}). 
Then, we have (see \eqref{eq:moment_supercritical})
$$
\begin{array}{l}
\left|I_t(x)-\sum\limits_{k=1}^{n-1}\binom{n}{k}\int_0^\infty e^{n\lambda_0 s} 
P_s\left(b(\sbullet) V^+_{n-k}(\sbullet)V^+_{k}(\sbullet)\right)(x)\, ds\right|
\le 2b^*\sum\limits_{k=1}^{n-1}\binom{n}{k} C^+_{n-k}F^+_ke^{-\beta(t-M)}\int_0^M e^{(n-1)\lambda_0 s} e^{\lambda_0 s} P_s(\ind)(x) \, ds\\
+  2b^*\sum\limits_{k=1}^{n-1}\binom{n}{k} C^+_{n-k}C^+_{k} \int_M^\infty e^{(n-1)\lambda_0 s} e^{\lambda_0 s} P_s(\ind)(x) \, ds\\
\le \left[2b^* C^+_1\left|\lambda_0^{-1}\right|\,\frac1{n-1} 
\sum\limits_{k=1}^{n-1}\binom{n}{k} C^+_{n-k}F^+_k+ C^+_{n-k}C^+_{k}\right]
\left( e^{-\beta(t-M)}+e^{(n-1)\lambda_0 M}\right)=\alpha_n\left( e^{-\beta(t-M)}+e^{(n-1)\lambda_0 M}\right),
\end{array}
$$
with
\begin{equation}
\label{eq:alfa n}
\alpha_n=2b^* C^+_1\left|\lambda_0^{-1}\right|\,\frac1{n-1} 
\sum\limits_{k=1}^{n-1}\binom{n}{k} \left( C^+_{n-k}F^+_k+ C^+_{n-k}C^+_{k}\right).
\end{equation}
Again, we take $\beta(t-M)=(n-1)|\lambda_0| M$, that is 
$M=\frac{\beta}{\beta+(n-1)|\lambda_0|}t$. On the other hand, the first term in
\eqref{eq:moment_supercritical} is bounded by $C^+_1 e^{(n-1)\lambda_0 t}\le C^+_1 e^{(n-1)\lambda_0 M}$ and therefore
$$
\sup\limits_{x} |v_n^+(t,x)-V_n^+(x)|\le (2\alpha_n+C^+_1) 
e^{-\beta_n t},
$$
where $\beta_n=\frac{\beta_{n-1}(n-1)|\lambda_0|}{\beta_{n-1}+(n-1)|\lambda_0|}<
\beta_{n-1}\le  \beta_1=(\lambda_1-\lambda_0)$. We have proved 
\eqref{eq:iteration moments supercritical} with 
$$
F^+_n\le \gamma_n=(2\alpha_n+C^+_1).
$$
So, the limits in \eqref{eq:moments supercritical} are shown for the function 
$f\equiv 1$. For the general case $f\in \CB$ 
the same technique works. 
We denote by $v_n^+(f)(t,x)=e^{\lambda_0 n t}\,\Ex{\langle Z_t\, ,f\rangle^n}$ and 
\begin{equation}
\label{eq:bound C^+_n(f)}
C^+_n(f)=\sup\limits_{t,x} |v_n^+(f)(t,x)|\le C^+_n \|f\|_\infty^n\le  n!\, \|f\|_\infty^n \mathcal{D}^{2n-1}.
\end{equation}
We have that
$$
\lim\limits_{t\to \infty} e^{\lambda_0 n t}\,\Ex{\langle Z_t\, ,f\rangle^n}=V_n^+(f,x),
$$
where $(V_n^+(f,x))_{n\ge 1}$ are recursively given by
\begin{equation}
\label{eq:iteration moments supercritical f}
\def\arraystretch{1.7}
\begin{array}{l}
V^+_n(f,x)=\begin{cases}\Theta_0(x) \int f(y) \mu_0(dy), &\hbox{for } n=1\\
\vspace{-0.2cm}
\int_0^\infty e^{n\lambda_0 s} P_s\left(b(\sbullet)
\sum\limits_{k=1}^{n-1} \binom{n}{k} V^+_{n-k}(f,\sbullet) V^+_k(f,\sbullet)\right)(x)\, ds &\hbox{for } n\ge 2
\end{cases}\\
\sup\limits_x |V^+_n(f,x)|\le C^+_n\,\|f\|_\infty^n\,,\\
F^+_n(f):=\sup\limits_{t,x} e^{\,\beta_n \, t} \left|v_n^+(f)(t,x)-V^+_n(f,x)\right|
\le (2\alpha_n+C^+_1)\,\|f\|_\infty^n\,,
\end{array}
\end{equation}
with $\beta_n$ as in \eqref{eq:iteration moments supercritical}, $\alpha_1=0$ and  $\alpha_n, n\ge 2$ as in \eqref{eq:alfa n}.
The recursion \eqref{eq:moments supercritical} for the moments
$(V^+_n(f,x))_n$ is shown, and the speed of convergence given by 
\eqref{eq:bound speed convergence moments supercritical} is obtained with the bound provided in
\eqref{eq:iteration moments supercritical f} and $\gamma_n=2\alpha_n+C_1^+$,
\begin{equation}
\label{eq:bound moments f supercritical}
F^+_n(f)\le\gamma_n \|f\|_\infty^n.
\end{equation}

\subsection{The survival probability.}
\label{sec:surv_prob__supercritical}

We have already proved in Section \ref{sec:survival} the uniform convergence of $u_{0}(t,\vsbullet)$ to $h\in \CBZ$ 
and characterized $h$ as solution of the equation in Corollary \ref{eqh}. It remains to prove that $h(x)>0$ 
for all $x$ and that its supremum is strictly less than $1$.

 The positivity follows from the Cauchy-Schwartz inequality. In fact, we have the inequality
$$
\Px{N_t>0}\ge \frac{\left(\Ex{N_{t}}\right)^2}{\Ex{N_{t}^2}}=\frac{\left( v_1^+(t,x)\right)^2}{v_2^+(t,x)},
$$
which implies the following inequality
$$
\liminf\limits_{t\to \infty} \Px{N_t>0}\ge \frac{\left(V_1^+(x)\right)^2}{V_2^+(x)}>0.
$$
Since the function $t\to \Px{N_t>0}$ is clearly decreasing,  we obtain
$\ 
h(x)>0
$.
Moreover, it is clear that for every $x, t>0$ the probability $\Px{N_t>0}<1$. 
Indeed, there is a positive probability that the first death event occurs before $t$. 
 This implies that $h(x)<1$, for all $x$ and in particular for an $x$ were $h$ attains its maximum.
 
 \subsection{Convergence in distribution.}
 \label{sec:convergence in distribution supercritical}
 
 As it has been done in previous sections, using the convergence of the moments 
 and the bounds on them, we get that, for
 all $x$ and all $f\in \CB$ and for all complex numbers $z$ in the ball 
 $B\left(0,\big(\|f\|_\infty \mathcal{D}\big)^{-1}\right)$,
 \begin{equation}
 \label{eq:laplace transform}
 \lim\limits_{t\to\infty}\Ex{\exp{\left(z e^{\lambda_0 t}\langle Z_t, f\rangle\right)}\, \Big | \, N_t>0}=
 1+\frac{1}{h(x)}\sum_{k\ge 1} \frac{V^+_k(f,x)}{k!} z^k.
 \end{equation}
 Carleman's condition
 $$
\sum\limits_{n\ge 1} \left(V^+_{2n}(f,x)/h(x)\right)^{-\frac{1}{2n}}=+\infty
$$
 prove  that  there exists a unique probability measure $\zeta^+_{f,x}$ defined on $\RR$, 
so that for all continuous and bounded real functions $\Psi$
$$
\Ex{\Psi\left(\langle e^{\lambda_0 t} Z_t\, ,f\rangle\right) \, \Big |\, N_t>0 }=
\frac{\Ex{\Psi\left(\langle e^{\lambda_0 t} Z_t\, ,f\rangle,\right),\, N_t>0}}{\Px{N_t>0}} \hspace{0.2cm}
\stackarrow{$ $}{t\to \infty}{1} \hspace{0.2cm} \int_{\RR} \Psi(y) \, \zeta^+_{f,x}(dy), 
$$
and its Laplace transform, in the ball $B\left(0,\big(\|f\|_\infty \mathcal{D}\big)^{-1}\right)$ is given by 
\eqref{eq:laplace transform}.

Using the results in \cite{dawsonperkinsAMS} (see Theorem 3.2.9), allows us to conclude that 
there exists a unique probability measure $\Lambda^+_x$, on $M_+(\RR)$, so that
$$
\int_{\RR} \Psi(y) \, \zeta^+_{f,x}(dy)=\int_{M_+(\RR)} \Psi(\langle \mu, f\rangle) \, \Lambda^+_x(d\mu),
$$
that is, the law of $e^{\lambda_0 t} Z_t$ under 
$\Px{\hspace{0.2cm}\vsbullet\hspace{0.2cm} \Big | \, N_t>0}$ 
converges to $\Lambda^+_x$.

A similar argument shows that the law of $e^{\lambda_0 t} Z_t$ under $\PP_{\delta_x}$ 
converges to a probability measure
$\Gamma_{\hspace{-0.07cm}x}$, which is related to $\Lambda_x^+$, by
$$
\Gamma_{\hspace{-0.07cm}x}=h(x)\Lambda_x^++(1-h(x))\delta_0.
$$

\subsection{Almost sure convergence for $(e^{\lambda_0 t}\langle Z_t,\Theta_0\rangle)_t$.}
\label{sec:a.s. convergence Theta_0}
Due to HP2), we observe that the real-valued process $e^{\lambda_{0}t} \langle Z_{t}, \Theta_{0}\rangle $ 
is a nonnegative martingale (for the filtration of the process $Z$). Indeed, for any $s\le t$, we have
$$
\Ex{e^{\lambda_{0}t} \langle Z_{t}, \Theta_{0}\rangle \, \Big |\, {\cal F}_{s} } 
= e^{\lambda_{0}s}\langle Z_{s}, e^{\lambda_{0}(t-s)}P_{t-s}\Theta_{0}\rangle 
=e^{\lambda_{0}s}\langle Z_{s}, \Theta_{0}\rangle .
$$
Then $(e^{\lambda_{0}t} \langle Z_{t}, \Theta_{0}\rangle)_t $ converges $\PP_{{\delta_{x}}}$-a.s. 
to a nonnegative random variable $W_{\infty}(x)$. From the bound  
\eqref{eq:bound C^+_n(f)}, applied to the function $\Theta_0$, we obtain
(uniformly on $x$)
$$
\sup\limits_{t, x} v_{n}^+(\Theta_{0})(t,x) = 
\sup\limits_{t,x}\Ex{\left(e^{\lambda_{0}t} \langle Z_{t}, \Theta_{0}\rangle\right)^n }\le 
\|\Theta_0\|_\infty^n  n!\, \mathcal{D}^{2n-1}<+\infty.
$$
Then, $(e^{\lambda_{0}t} \langle Z_{t}, \Theta_{0}\rangle)_t$ is a uniformly integrable martingale in 
$L^n(\PP_{\delta_x})$ for all $\, n\ge1$ and all $x$. Therefore  for all $n,x$,
$$
e^{\lambda_{0}t} \langle Z_{t}, \Theta_{0}\rangle \stackarrow{L^n(\PP_{\delta_x})}{t\to \infty}{1}
W_\infty(x).
$$ 
 In particular, we get
 \begin{equation}
\label{eq:moments W infty}
 \Ex{W_\infty(x)}=\Theta_0(x)>0\,, \, \Ex{(W_\infty(x))^n}=V_+^n(\Theta_0,x) \hbox{ for all } n\ge 2.
 \end{equation}
 
 We finish this section with a result a.s. for every function $f\in \CB$ along a fixed subsequence that converges to 
 infinity sufficiently fast.
 \begin{lemma} 
 \label{lem:convergence over tk}
 There exists a constant $\beta>0$, so that, for all sequences $(t_k)_k\subset \RR_+$
 satisfying $t_k\uparrow \infty$ and $\sum_k e^{-\beta t_k}<\infty$, we have for all $x,\, f\in \CB$,
 $$
 e^{\lambda_0 t_k}\langle Z_{t_k}, f\rangle  \stackarrow{\PP_{\delta_x}\hbox{-a.s.}}{k\to\infty}{1} W_\infty(x) 
 \hbox{ $\int f \,d\mu_0$}.
 $$ 
 The set of measure $0$ where the convergence doesn't hold could depend on $x,f$ and $(t_k)_k$. 
 
 In particular, we conclude that for all $1\le n<\infty, \, x\in \RR,\, f  \in \CB$,
 $$
 e^{\lambda_0 t}\langle Z_{t}, f\rangle \stackarrow{L^n(\PP_{\delta_x})}{t\to \infty}{1} W_\infty(x) 
 \hbox{ $\int f \,d\mu_0$},
 $$
 showing also that, for all $n\ge 1$
 \begin{equation}
 \label{eq:moments W infty II}
 V_n^+(f,x)=\lim\limits_{t\to \infty} \Ex{\left(e^{\lambda_0 t}\langle Z_t,f\rangle\right)^n}=
 \Ex{W_\infty(x)^n }\hbox{ $\left(\int f \,d\mu_0\right)^n$}= V_n^+(\Theta_0,x)\hbox{ $\left(\int f \,d\mu_0\right)^n$}.
 \end{equation}
 \end{lemma}
 \begin{proof}
 Consider $0<\beta=\beta_2$ given in \eqref{eq:iteration moments supercritical}. 
 Fix a sequence $(t_k)_k$ that satisfies the requirements of the lemma.
 Now, fix $g$ so that $\int g\, d\mu_0=0$. Then, for all $x$ it holds $\Pi(g)(x)=\Theta_0(x) \int g \, d\mu_0=0$,
 which implies that 
 $$
 V_1^+(g,x)=\lim\limits_{t\to \infty} e^{\lambda_0 t}\Ex{\langle Z_t, g\rangle}=0.
 $$
 Using the recursion \eqref{eq:moments supercritical} proved in Subsection
 \ref{sec:moments_supercritical}, we observe that for all $n,x$ the moments $V^+_n(g,x)=0$ and then
from \eqref{eq:iteration moments supercritical f}, we get uniformly on $x,t$
$$ 
\left|\Ex{e^{2\lambda_{0}t} \langle Z_{t}, g \rangle^2 }\right| \le (2 \alpha_{2} + C^+_{1}) \|g\|_\infty^2 e^{-\beta_{2}t}
=\gamma_2\|g\|_\infty^2 e^{-\beta_{2}t}.
$$
From here, we obtain
$$
\Ex{\sum\limits_k \left(e^{\lambda_0 t_k} \langle Z_{t_k}, g\rangle\right)^2}=
\sum\limits_k \Ex{\left(e^{\lambda_0 t_k} \langle Z_{t_k}, g\rangle\right)^2}\le \gamma_2 \|g\|_\infty^2
\sum_k e^{-\beta_2 t_k}<\infty,
$$
showing that, for all $x$, the following limit holds $a.s.$ under $\PP_{\delta_x}$
$$
\lim\limits_{k\to\infty} e^{\lambda_0 t_k} \langle Z_{t_k}, g\rangle=0=\Pi(g).
$$ For any $f\in \CB$, we have the decomposition
$f=g+\Theta_0 \hbox{ $\int f \,d\mu_0$}$, 
with  $g=f-\Pi(f)\in \CB$ satisfying $\Pi(g)=0$. So, for any $x$, 
the following limit holds $\PP_{\delta_x}$-a.s.
$$
\lim\limits_{k\to\infty} e^{\lambda_0 t_k} \langle Z_{t_k}, f\rangle=W_\infty(x) \hbox{ $\int f \,d\mu_0$}.
$$
\end{proof}

Now, we prove that $\Px{W_\infty>0}=h(x)=\lim\limits_{t\to \infty} \Px{N_t>0}$. 
For this purpose, we use that for all $z\in \CC$
in a neighborhood of $0$ (see \eqref{eq:laplace transform}), we have
$$
\begin{array}{l}
\Ex{1-\exp{(z W_\infty(x)) }, W_\infty(x)>0}=\Ex{1-\exp{(z W_\infty(x)) }}=
\lim\limits_{t\to \infty} \Ex{1-\exp{\left(z e^{\lambda_0 t} \langle Z_{t},\Theta_0\rangle\right)}}\\
=\lim\limits_{t\to \infty} \Px{N_t>0} 
\Ex{1-\exp{\left(z e^{\lambda_0 t}\langle Z_{t},\Theta_0\rangle\right)}\,\Big | N_t>0}
=h(x) \int_\RR (1-e^{zy})\, \zeta^+_{\Theta_0,x}(dy).
\end{array}
$$
The probability measure $\zeta^+_{\Theta_0,x}$ was introduced in Section 
\ref{sec:convergence in distribution supercritical} and since $\Theta_0>0$, it is supported on
$\RR_+$. Now, both functions  
$$
z\to\Ex{1-\exp{(z W_\infty(x))}} \hbox{ and } z\to h(x) \int_{\RR_+} (1-e^{zy})\, \zeta^+_{\Theta_0,x}(dy)
$$
 are analytic and agree in a neighborhood of zero. Also, they have analytic continuations
to the open left half plane $\{z\in \CC: \, \Re(z)<0\}$, because $W(x)$ is a nonnegative random variable. 
Therefore both functions agree on that open half plane, which allow us
to take the limit as $z\to -\infty$ on the real axis and conclude:
$$
\Px{ W_\infty(x)>0}=h(x)\,\zeta^+_{\Theta_0,x}(\RR)=h(x).
$$ 
This implies also that, $\Lambda_x^+$, which is the limit law of $e^{\lambda_0 t} Z_{t}$ on $M_F$ under 
$\Px{\hspace{0.2cm}\vsbullet\hspace{0.2cm} \Big | \, N_t>0}$, is the law of
$W_\infty(x)\, \mu_0$, under the probability measure $\Px{\hspace{0.2cm}\vsbullet\hspace{0.2cm} \Big | \, W_\infty(x)>0}$.
Similarly, $\Gamma_x$ is the law of  $W_\infty(x) \,\mu_0$, under the probability measure $\PP_{\delta_x}$.

\subsection{Convergence in Probability and in $L^n$} 
In this section, we prove Theorem \ref{the:supercritical} (5): the measured valued process 
$(e^{\lambda_0 t} Z_{t})_t$ converges in $\PP_{\delta_x}$-probability 
to $W_\infty(x)\, \mu_0$, for every $x$. For that we  use the distance on $M_F$ which metrizes 
the weak convergence defined by
$$
d(\mu,\tau)=\sum\limits_{n\ge 0}  2^{-n}|\mu(f_n)-\tau(f_n)|,
$$
where $(f_{n})_{n}$ is sequence of functions convergence determining, satisfying $f_{0}\equiv 1$ and for 
$n\ge 1$, $f_{n}\in C^{\infty}_{c}(\RR)$. We also assume that $\|f_n\|_\infty \le 1$, for all $n$.

The convergence in probability will follow from the $L^1(\PP_{\delta_x})$-convergence of 
$\,d\left(e^{\lambda_0 t} Z_t,W_\infty(x)\, \mu_0\right)\,$ to $0$ (as $t$ tends to infinity).
 
For every sequence $(s_m)_m$ that converges to
infinity, we can find a subsequence $(t_k=s_{m_k})_k$ so that $\sum\limits_k e^{-\beta t_k}<\infty$, where
$\beta=\beta_2$ as in Lemma \ref{lem:convergence over tk}.
For that is enough to assume that $t_k\ge k$ for large $k$. Using Lemma \ref{lem:convergence over tk}, we have the following 
$\PP_{\delta_x}$-a.s. convergence. For all $n\ge 0$,
$$
\lim\limits_{k\to\infty} e^{\lambda_0 t_k} \langle Z_{t_k},f_n\rangle=
W_\infty(x) \hbox{ $\int f_n \,d\mu_0$}.
$$
Thus, using \eqref{eq:bound C^+_n(f)} and \eqref{eq:moments W infty II}, we have for all $n$
$$
\def\arraystretch{1.7}
\begin{array}{ll}
\Ex{\left|e^{\lambda_0 t_k} \langle Z_{t_k},f_n\rangle-W_\infty(x) 
\hbox{ $\int f_n \,d\mu_0$}\right|}&\hspace{-0.2cm}\le
\Ex{e^{\lambda_0 t_k} \langle Z_{t_k},|f_n|\rangle}+
\Ex{W_\infty(x) \hbox{ $\int |f_n| \,d\mu_0$}}\\
&\hspace{-0.2cm}\le v_1^+(|f_n|,t_k,x)+V_1^+(|f_n|,t_k,x)
\le 2 C_1^+.
\end{array}
$$
Therefore, we obtain
$$
\Ex{d\left(e^{\lambda_0 t_k} Z_{t_k},W_\infty(x) \mu_0\right)}
=\sum\limits_{n\ge _0} 2^{-n} \, \Ex{\left|e^{\lambda_0 t_k} \langle Z_{t_k},f_n\rangle-W_\infty(x) 
\hbox{ $\int f_n \,d\mu_0$}\right|},
$$
which converges to $0$ as $k\to \infty$, thanks to the Dominated Convergence Theorem for series.

So far, we have proved that for every sequence $(s_m)_m$ there exists a subsequence 
$(t_k=s_{m_k})_k$ for which
$$
d\left(e^{\lambda_0 t_k} Z_{t_k},W_\infty(x) \mu_0\right) \stackarrow{L^1(\PP_{\delta_x})}{k\to \infty}{1} 0.
$$
Therefore,  the $L^1(\PP_{\delta_x})$ convergence holds:
$$
\lim\limits_{t\to +\infty} d\left(e^{\lambda_0 t} Z_{t},W_\infty(x) \mu_0\right)=0,
$$
and a fortiori the convergence in Probability.

\begin{remark} Using the same ideas one can prove the $L^p$ convergence of 
$\, d\left(e^{\lambda_0 t} Z_{t},W_\infty(x)\, \mu_0\right)$
towards $0$, for all $1\le p<\infty$.
\end{remark} 

Finally, we prove  Theorem \ref{the:supercritical} (6), namely for $n\ge 1$,
$$
\lim\limits_{t\to +\infty}\,
\sup_{\hbox{\scriptsize$\stackunder{0\le t,\,x\in \RR}{\|f\|_{\infty}\! \le 1}$}}
\,\Big\| \, e^{\lambda_{0}t} \langle Z_{t},f \rangle - 
W_{\infty}(x) \,\hbox{$\int f\,  d\mu_{0}$ }\Big \|_{L^n\left(\PP_{\delta_x}\right)} = 0.
$$
For that, we consider as before the decomposition  
$f=g+\Theta_0 \,\mu_0(f)$, where $g=f-\Pi(f)$ and so $\Pi(g)=0$.  Recall that for all 
$n\ge 1$,  $V_n^+(g)=0=\int g\,d\mu_0$, and then 
from \eqref{eq:iteration moments supercritical f}
(see also \eqref{eq:iteration moments supercritical}).
$$
\Big\| \, e^{\lambda_{0}t} \langle Z_{t},g \rangle - 
W_{\infty}(x)\mu_{0}(g)\Big\|^n_{L^n\left(\PP_{\delta_x}\right)}= 
\Big\| \, e^{\lambda_{0}t} \langle Z_{t},g \rangle-V_n^+(g) \Big\|^n_{L^n\left(\PP_{\delta_x}\right)}
\le e^{-\beta_n t} F_n^+ \|g\|_\infty^n\le e^{-\beta_n t} F_n^+ \left(\|\Theta_0\|_\infty\mu_0(1) +1\right)^n 
\|f\|_\infty^n,
$$ 

For the other term, we use the $L^n\left(\PP_{\delta_x}\right)$ convergence of 
$e^{\lambda_{0}t} \langle Z_{t},\Theta_0 \rangle$ to $W_\infty(x)$, hence
$$
\begin{array}{ll}
\big\| \, e^{\lambda_{0}t} \langle Z_{t},\Theta_0 \mu_0(f) \rangle - 
W_{\infty}(x)\mu_{0}(f)\big\|^n_{L^n\left(\PP_{\delta_x}\right)} &\hspace{-0.2cm}\le |\mu_0(f)|^n 
\big\| \, e^{\lambda_{0}t} \langle Z_{t},\Theta_0 \rangle-W_\infty(x)\big\|^n_{L^n\left(\PP_{\delta_x}\right)}\\
 &\hspace{-0.2cm}\le \|f\|^n_\infty (\mu_0(1))^n \big\| \, e^{\lambda_{0}t} \langle Z_{t},\Theta_0 \rangle-
W_\infty(x)\big\|^n_{L^n\left(\PP_{\delta_x}\right)}.
\end{array}
$$
Therefore, we need to study the uniform convergence in $t,x$ of the last term.
For that, we use the inequality,
for all $n\ge 1, u,v \in \RR_+$
$$
|u-v|^n\le |u^n-v^n|,
$$
with $u=e^{\lambda_{0}t} \langle Z_{t}, \Theta_0\rangle, \, v=W_{\infty}(x)$.
Using Cauchy-Schwartz inequality we get, using \eqref{eq:iteration moments supercritical f}
$$
\begin{array}{l}
\Big\| \, e^{\lambda_{0}t} \langle Z_{t},\Theta_0\rangle - 
W_{\infty}(x)\Big \|_{L^n}^n\le \Ex{|u^n-v^n|}\le  \left(\Ex{(u^n-v^n)^2}\right)^{1/2}=
\left(\Ex{u^{2n}-2u^nv^n+v^{2n}}\right)^{1/2}\\
\le \left(\Ex{v^{2n}-u^{2n}}\right)^{1/2}=\left|V_{2n}^+(\Theta_0,x)-v_{2n}^+(\Theta_0)(t,x)\right|^{1/2}
\le e^{-\beta_{2n}\, t/2} \left(F_{2n}^+\right)^{1/2} \|\Theta_0\|_\infty^n.
\end{array}
$$
Here, we have used that $u,v$ are nonnegative and $\Ex{v\,\Big|\, \F_t}=u$, which gives
$$
\Ex{u^nv^n}=\Ex{u^n\Ex{v^n\,\Big|\, \F_t}}\ge \Ex{u^n\Ex{v\,\Big|\, \F_t}^n}=\Ex{u^{2n}}
$$
and also $\Ex{u^{2n}}\le \Ex{v^{2n}}$.
Putting all together, we obtain the bound (recall that $\beta_{2n}\le \beta_n$)
$$
\sup\limits_{\hbox{\scriptsize$\stackunder{0\le t,\,x\in \RR}{\|f\|_{\infty}\! \le 1}$}}
\Big\| \, e^{\lambda_{0}t} \langle Z_{t},\Theta_0 \mu_0(f) \rangle - 
W_{\infty}(x)\mu_{0}(f)\Big\|_{L^n\left(\PP_{\delta_x}\right)}\le 
e^{-\frac{\beta_{2n}}{2n} t}\left(\Big(F_n^+\Big)^{1/n}+\left(F_{2n}^+\right)^{1/2n}\right)
\Big(\|\Theta_0\|_\infty(\mu_0(1)+1)+1\Big),
$$
and the result is shown.

\subsection{Convergence a.s. for $(e^{\lambda_0 t} Z_t)_t$, proof of 
Theorem \ref{the:convergence a.s.}}
\label{sec:proof theorem convergence a.s.}
Our first step is to show the a.s. convergence for the function $f\equiv1$.
\begin{lemma} 
\label{lem:masse} For all $x$, the following limit holds $\PP_{\delta_x}$-a.s.
$$
\lim\limits_{t\to \infty} e^{\lambda_0 t} \langle Z_{t}, 1\rangle=W_\infty(x) \,  \mu_0(1).
$$
\end{lemma}
\begin{proof} Consider a sequence $(t_k)_k$ such that $t_k\uparrow \infty$. 
Below we will impose more restrictions on this
sequence. For $k$ fixed,  denote 
$t=t_k,v=t_{k+1}, I=I_{k}=[t,v], \Delta=v-t$. Consider $t\le s\le v$. From \eqref{eq:Master General}
we have, after multiplying by $e^{\lambda_0 s}$, and using that $\lambda_0<0$
$$
\def\arraystretch{1.5}
\begin{array}{ll}
e^{\lambda_0 s}\langle Z_s,1\rangle&\hspace{-0.2cm}=e^{\lambda_0 s}\langle Z_t,1\rangle +
e^{\lambda_0 s} \int_{[t,s]\times \RR_+\times \NN} \ind_{j\le N_{u-}}
\left(\ind_{\theta\le b(X^j_{u-})}-\ind_{ b(X^j_{u-})\le \theta\le  b(X^j_{u-})+d(X^j_{u-})}\right)\, Q(du,d\theta,dj)\\
&\hspace{-0.2cm}\le e^{\lambda_0 t}\langle Z_t,1\rangle +
\int_{[t,s]\times \RR_+\times \NN}  e^{\lambda_0 u} \ind_{j\le N_{u-}} \ind_{\theta\le b^*}\, Q(du,d\theta,dj)\\
&\hspace{-0.2cm}\le e^{\lambda_0 t}\langle Z_t,1\rangle +
\int_{[t,v]\times \RR_+\times \NN}  e^{\lambda_0 u} \ind_{j\le N_{u-}} \ind_{\theta\le b^*}\, Q(du,d\theta,dj).
\end{array}
$$  
Notice that the last term on the right hand side does not depend on $s$ and so we get
$$
0\le \sup\limits_{s\in I} e^{\lambda_0 s}\langle Z_s,1\rangle-e^{\lambda_0 t}\langle Z_t,1\rangle\le
\int_{[t,v]\times \RR_+\times \NN}  e^{\lambda_0 u} \ind_{j\le N_{u-}} \ind_{\theta\le b^*}\, Q(du,d\theta,dj),
$$
which leads to (using that the intensity measure for $Q$ is $du\times d\theta\times n(dj)$)
\begin{equation}
\label{eq:cota supremo}
\begin{array}{ll}
\Ex{\left(\sup\limits_{s\in I} e^{\lambda_0 s}\langle Z_s,1\rangle-e^{\lambda_0 t}\langle Z_t,1\rangle\right)^2}
&\hspace{-0.2cm}\le b^* \int\limits_t^v \Ex{e^{2\lambda_0 u} \langle Z_u,1\rangle}\, du\le b^*e^{\lambda_0 t} 
\int_t^v \Ex{e^{\lambda_0 u} \langle Z_u,1\rangle} \, du\\
&\hspace{-0.2cm}\le C_1^+b^*\Delta \, e^{\lambda_0 t}.
\end{array}
\end{equation}
Recalling that $\lambda_0<0$, so we now assume that $(t_k)_k$ satisfies
$\sum_k e^{\lambda_0 t_k}+e^{-\beta t_k}<+\infty$, where $0<\beta=\beta_2$ as before. 
Therefore, we have 
$$
\sum_k \Ex{\left(\sup\limits_{s\in I_{k}} e^{\lambda_0 s}\langle Z_s,1\rangle-e^{\lambda_0 t_{k}}\langle Z_{t_{k}},1\rangle\right)^2}<\infty,
$$
showing that
\begin{equation}
\label{eq:control of sup}
\sup\limits_{s\in I_k} e^{\lambda_0 s}\langle Z_s,1\rangle-e^{\lambda_0 {t_k}}\langle Z_{t_k},1\rangle
\stackarrow{\PP_{\delta_x}\hbox{- a.s.}}{k\to \infty}{1} 0.
\end{equation}
We now prove a similar result with the infimum.  Recalling that $t\le s\le v$, we consider
$$
\def\arraystretch{1.5}
\begin{array}{ll}
e^{\lambda_0 v}\langle Z_v,1\rangle&\hspace{-0.2cm}=e^{\lambda_0 v}\langle Z_s,1\rangle +
e^{\lambda_0 v} \int_{[s,v]\times \RR_+\times \NN} \ind_{j\le N_{u-}}
\left(\ind_{\theta\le b(X^j_{u-})}-\ind_{ b(X^j_{u-})\le \theta\le  b(X^j_{u-})+d(X^j_{u-})}\right)\, Q(du,d\theta,dj)\\
&\hspace{-0.2cm}\le e^{\lambda_0 s}\langle Z_s,1\rangle +\int_{[s,v]\times \RR_+\times \NN}  
e^{\lambda_0 u} \ind_{j\le N_{u-}} \ind_{\theta\le b^*}\, Q(du,d\theta,dj)\\
&\hspace{-0.2cm}\le e^{\lambda_0 s}\langle Z_s,1\rangle +\int_{[t,v]\times \RR_+\times \NN}  
e^{\lambda_0 u} \ind_{j\le N_{u-}} \ind_{\theta\le b^*}\, Q(du,d\theta,dj).
\end{array}
$$
As before we get
\begin{equation}
\label{eq:control of inf}
e^{\lambda_0 {t_{k+1}}}\langle Z_{t_{k+1}},1\rangle-\inf\limits_{s\in I_k} e^{\lambda_0 s}\langle Z_s,1\rangle
\stackarrow{\PP_{\delta_x}\hbox{- a.s.}}{k\to \infty}{1} 0.
\end{equation}
From Lemma \ref{lem:convergence over tk} and the assumption made on $(t_k)_k$,  
we have that $(e^{\lambda_0 {t_{k+1}}}\langle Z_{t_{k+1}},1\rangle)_k$ and 
$(e^{\lambda_0 {t_{k}}}\langle Z_{t_{k}},1\rangle)_k$
converge to $W_\infty(x)\, \mu_0(1)$ $\PP_{\delta_x}$-a.s. We conclude that on the set where both converge
and \eqref{eq:control of sup},  \eqref{eq:control of inf} hold, we have the a.s. limit
$$
\lim\limits_{t\to \infty} \sup\limits_{s\ge t} e^{\lambda_0 s}\langle Z_s,1\rangle-
\inf\limits_{s\ge t} e^{\lambda_0 s}\langle Z_s,1\rangle=0.
$$
Therefore, we have the $\PP_{\delta_x}$-a.s. convergence
$$
\lim\limits_{t\to \infty} e^{\lambda_0 t}\langle Z_t,1\rangle =W_\infty(x)\, \mu_0(1),
$$
proving the desired result for the function $1$.
\end{proof}

Now, we generalize this result to a non-negative function $f\in C^\infty_{c}(\RR)$. 

\begin{lemma}
\label{lem:Cinfty}
For all $x$ and all $f\in C^\infty_{c}(\RR)$, the following limit holds $\PP_{\delta_x}$-a.s.
$$
\lim\limits_{t \to \infty} e^{\lambda_0 t}\langle Z_{t},f\rangle=W_\infty(x) \hbox{$\int f\, d\mu_0$}.
$$
\end{lemma}

\begin{proof}
For that we use the decomposition \eqref{eq:Master General}: for $s\in I=[t,v]$
$$
\begin{array}{ll}
\langle Z_s,f\rangle&\hspace{-0.2cm}=\langle Z_t,f\rangle +
\int_t^s \langle Z_u, \mathscr{G} f\rangle du+ M^f_s-M^f_t\\
&\hspace{0.1cm}+\int_{[t,s]\times \RR_+\times \NN}  f(X^j_{u-}) \ind_{j\le N_{u-}}
\left(\ind_{\theta\le b(X^j_{u-})}-\ind_{ b(X^j_{u-})\le \theta\le  b(X^j_{u-})+d(X^j_{u-})}\right)\, Q(du,d\theta,dj).
\end{array}
$$  
Using Assumption (HSM) and  after multiplying
by $e^{\lambda_0 s}$, we obtain  that
$$
\begin{array}{ll}
e^{\lambda_0 s}\langle Z_s,f\rangle&\hspace{-0.2cm}\le e^{\lambda_0 t}\langle Z_t,f\rangle +
C(f) \int_t^v e^{\lambda_0 u} \langle Z_u, 1\rangle du+ e^{\lambda_0 t}\sup\limits_{ r\in I}|M^f_r-M^f_t|\\
&\hspace{0.1cm}+\|f\|_\infty \int_{[t,v]\times \RR_+\times \NN}  e^{\lambda_0 u} \ind_{j\le N_{u-}}
\ind_{\theta\le b^*} \, Q(du,d\theta,dj),
\end{array}
$$  
where $C(f) = C$ is the constant appearing in the hypothesis (HSM).

Taking the supremum over $s$ on the left, pass the first term on the right to the left, then squaring both sides and 
taking the expected value,  we get  using Doob's inequality for $p=2$
$$
\def\arraystretch{1.5}
\begin{array}{l}
\Ex{\left(\sup\limits_{s\in I} e^{\lambda_0 s}\langle Z_s,f\rangle-e^{\lambda_0 t}\langle Z_t,f\rangle\right)^2}\le\\
\le 3C^2 \Ex{\left(\int_t^v e^{\lambda_0 u} \langle Z_u, 1\rangle du\right)^2}+
12 e^{2\lambda_0 t} \Ex{\left(M^f_v-M^f_t\right)^2}
+3b^*\|f\|_\infty^2\int_t^v e^{2\lambda_0 u} \Ex{\langle Z_u,1\rangle}\, du\\
\le 3 C^2 \Ex{\sup\limits_{u\in I} e^{2\lambda_0 u} \langle Z_u, 1\rangle^2}\, \Delta^2 +
12 C^2 e^{2\lambda_0 t}\int_t^v \Ex{\langle Z_u,1\rangle} \, du
+3b^*\|f\|_\infty^2 e^{\lambda_0 t}\int_t^v \Ex{ e^{\lambda_0 u} \langle Z_u,1\rangle}\, du\\
\le 3 C^2 \Ex{\sup\limits_{u\in I} e^{2\lambda_0 u} \langle Z_u, 1\rangle^2} \, \Delta^2+
3\left(4 C^2 e^{|\lambda_0|\Delta}+b^*\|f\|_\infty^2\right)
\left(\int_t^v \Ex{ e^{\lambda_0 u} \langle Z_u,1\rangle}\, du\right)\, e^{\lambda_0 t},
\end{array}
$$
where we have used that $\lambda_0(t-u)\le |\lambda_0|\Delta$. Recall that  $\Ex{ e^{\lambda_0 u} \langle Z_u,1\rangle}\le C_1^+$
and $\Ex{ e^{2\lambda_0 u} \langle Z_u,1\rangle^2}\le C_2^+$, uniformly in $u\in \RR_+, x\in \RR$ (see \eqref{eq:C1+}
\eqref{eq:Cn+} and \eqref{eq:bound moments I}). On the other hand, using \eqref{eq:cota supremo}, we have
$$
\Ex{\sup\limits_{u\in I} e^{2\lambda_0 u} \langle Z_u, 1\rangle^2}\le
2  C_1^+b^* \Delta \, e^{\lambda_0 t}+2\Ex{e^{2\lambda_0 t} \langle Z_t, 1\rangle^2}
\le 2  C_1^+b^* \Delta \, e^{\lambda_0 t}+2 C_2^+.
$$
Finally, we obtain the following bound
$$
\Ex{\left(\sup\limits_{s\in I} e^{\lambda_0 s}\langle Z_s,f\rangle-e^{\lambda_0 t}\langle Z_t,f\rangle\right)^2}\le
K e^{|\lambda_0|\Delta}\left(\Delta^2+e^{\lambda_0 t}\right),
$$
for a constant $K$ depending on $f$.  
 
We take 
$$
t_k=F\sum\limits_{j=1}^k \frac{1}{k}\approx F \ln(k), \hbox{ and } \Delta_k=F\frac1{k+1},
$$
for a large positive $F$ so that $\sum\limits_{k}(e^{\lambda_0 t_k}+e^{-\beta_2 t_k}+\Delta_k^2)<+\infty$ and therefore, the following 
limits hold $\PP_{\delta_x}$-a.s.
$$
\begin{array}{l}
\lim\limits_{k \to \infty} \sup\limits_{s\in I_k} e^{\lambda_0 s}\langle Z_s,f\rangle-e^{\lambda_0 t_k}\langle Z_{t_k},f\rangle=0,\\
\lim\limits_{k \to \infty} e^{\lambda_0 t_k}\langle Z_{t_k},f\rangle=W_\infty(x) \hbox{$\int f\, d\mu_0$}.
\end{array}
$$
Using the same ideas, we also obtain the limit
$$
\begin{array}{l}
\lim\limits_{k \to \infty} e^{\lambda_0 t_{k+1}}\langle Z_{t_{k+1}},f\rangle-
\inf\limits_{s\in I_k} e^{\lambda_0 s}\langle Z_s,f\rangle=0,\\
\end{array}
$$
and the convergence $\PP_{\delta_x}$-almost sure is shown
$$
\lim\limits_{t \to \infty} e^{\lambda_0 t}\langle Z_{t},f\rangle=W_\infty(x) \hbox{$\int f\, d\mu_0$}.
$$
\end{proof}

 We consider as previously a convergence determining sequence $(f_{n})_{n}$ with $f_{0}\equiv 1
$ and $f_{n}\in C^\infty_{c}(\RR)$. We deduce that 
 for all $x$,  $\PP_{\delta_x}$-a.s.,  $\forall n\ge 0$, 
$$
\lim\limits_{t \to \infty} e^{\lambda_0 t}\langle Z_{t},f_{n}\rangle=W_\infty(x) \hbox{$\int f_{n}\, d\mu_0$}.
$$
and then it follows the convergence for all $f\in \CB$.

\subsection{Branching Q-process: proof of Theorem \ref{the:Q-process}}
\label{sec:Q-process proof}

The proof is based on the Markov and Branching properties. Consider $0\le s <t$, we 
have for any measurable bounded function $F$,
$$
\begin{array}{ll}
\Ex{F\left(\langle L_s, \Phi\rangle \right),\, N_t>0}&\hspace{-0.2cm}=
\Ex{\EE\left(F\left(\langle L_s, \Phi\rangle \right),\, N_t>0\,\, | \,\, \mathcal{F}_s \right)}
\\
&\hspace{-0.2cm}=\Ex{F\left(\langle L_s, \Phi\rangle \right) \, 
\Big(1-\exp(\langle Z_s\,, \log(1-u_0(t-s,\sbullet))\rangle)\Big),\, N_s>0},
\end{array}
$$
where we recall that $u_0(t-s,z)=\EE_{\delta_z}(N_{t-s}>0)$.

In the critical case, from \eqref{eq:asym survival critical} and  
\eqref{eq:Uniform survival prob} we have that $\,u_0(t-s,\sbullet)\,$ tends to 
$0$ uniformly when $t$ tends to infinity. Then, 
$\ 
1-\exp\Big(\langle Z_s\,, \log(1-u_0\big(t-s,\sbullet)\big)\rangle\Big)\,$ is of the same order as $\, \langle Z_s\,, u_0(t-s,\sbullet)\rangle$,
and we have the almost sure convergence, as $t\to\infty$
$$
\frac{1-\exp\Big(\langle Z_s\,, \log\big(1-u_0(t-s,\sbullet)\big)\rangle\Big)}{\Px{N_t>0}} \to \frac{\langle Z_s\,, \Theta_0\rangle}{\Theta_0(x)}.
$$
Equation \eqref{eq:Uniform survival prob} gives the domination we need to show
$$
\lim\limits_{t\to \infty} \Ex{F\left(\langle L_s, \Phi\rangle \right)\, \Big |\, N_t>0}=
\Ex{F\left(\langle L_s, \Phi\rangle \right) \, \frac{\langle Z_s\,, \Theta_0\rangle}{\Theta_0(x)}, \, N_s>0}.
$$
In the subcritical case, the argument is similar using \eqref{eq:Uniform survival prob subcritical}, and when  $t$ tends to infinity,
$$
\frac{1-\exp\Big(\langle Z_s\,, \log\big(1-u_0(t-s,\sbullet)\big)\rangle\Big)}{\Px{N_t>0}} \to e^{\lambda_0 s}\frac{\langle Z_s\,, \Theta_0\rangle}{\Theta_0(x)},
$$
and then
$$
\lim\limits_{t\to \infty} \Ex{F\left(\langle L_s, \Phi\rangle \right)\, \Big |\,\, N_t>0}=
\Ex{F\left(\langle L_s, \Phi\rangle \right) \, e^{\lambda_0 s}\frac{\langle Z_s\,, \Theta_0\rangle}{\Theta_0(x)}, \, N_s>0}.
$$
The supercritical case is even simpler, because $u_0(t-s,y)\downarrow h(y)>0$ and then
$1-\exp\Big(\langle Z_s\,, \log\big(1-u_0(t-s,\sbullet)\big)\rangle\Big)$ is nonnegative, decreasing and bounded by $1$. 
Then, the Monotone Convergence Theorem shows
$$
\begin{array}{ll}
\lim\limits_{t\to \infty} \Ex{F\left(\langle L_s, \Phi\rangle \right)\, |\,\,  N_t>0}\hspace{-0.2cm}&=
\Ex{F(\langle L_s,\Phi\rangle)\, \frac{1-\exp(\langle Z_s, \log(1-h)\rangle)}{h(x)}}\\
\hspace{-0.2cm}&=\Ex{F(\langle L_s,\Phi\rangle)\, \frac{1-\exp(\langle Z_s, \log(1-h)\rangle)}{1-\exp(\langle Z_0, \log(1-h)\rangle)}}.
\end{array}
$$


\section{Proofs for the different classes of examples.}
\label{sec:examples}

\subsection{Proof of HP5) for the different classes of examples}

\begin{proposition}\label{HP3-diff}
HP5) holds in the three classes of examples introduced in Section \ref{sec:example},
under the corresponding set of hypothesis (HV), (HA) and (HJ).
\end{proposition}

\begin{proof} 
We give a proof which applies in the three different classes. Let $e\in \{1,2,3\}$ the index of the example class.  
We denote by $\mathscr{G}^{(e)}$ the generator of the  underlying process. More precisely
$$\mathscr{G}^{(1)}f =\frac12 f''-af' \ ;\  \mathscr{G}^{(2)} f = \frac12 f''-af'+ L_{1}f\ ;\   \mathscr{G}^{(3)} f = -f'+ L_{1}f.$$
 The empirical measure-valued process $(Z_{t})$  satisfies
the semi-martingale decomposition \eqref{eq:Master General}. More precisely, for all functions  $f \in C^2_c(\RR)$, 
we have for $Z_0=\delta_x$, $x\in \RR$
 and
Brownian motions $(B^i)_{i\in \NN}$ and independent Poisson point measure $N^i$ with intensity measure 
$ds \otimes d\theta \otimes \eta(d\xi)$ the decomposition
\begin{equation}
\label{eq:MasterZ}
\begin{array}{ll}
\langle Z_t\, ,f\rangle&\hspace{-0.2cm}=f(x)+\int_0^t \langle Z_s,\mathscr{G}^{(e)} f\rangle ds+ M^{(e)} _{f}(t) \\
&\hspace{0.1cm}+\int_{[0,t]\times \RR_+\times \NN} \ind_{1\le j\le N_{s}} \,f(X^j_s)
\left(\ind_{\theta \le b(X^j_{s-})}-\ind_{b(X^j_{s-})<\theta\le b(X^j_{s-})+d(X^j_{s-})}\right)
Q(ds,d\theta,dj),
\end{array}  
\end{equation}
where $Q$ is a Poisson point measure with intensity measure $ds \otimes d\theta \otimes n(di)$
on on $\RR_+\times \RR_+\times \NN$ 
 and $M^{(e)} _{f}$ a square integrable martingale with predictable quadratic variation
\begin{equation}
\label{eq:bound corchete I}
\left\langle M^{(e)} _{f}\right\rangle_{t}=1_{{e=1,2}}\int_0^{t} \left\langle Z_r, (f')^2\right\rangle \, dr 
+1_{{e=2,3}} \int_0^{t} \left\langle Z_r, \int  (f( z)-f(\sbullet))^2\;R(\sbullet,dz)\right\rangle\, dr.
\end{equation}

Let us  introduce the function  $\ h(y)=y^2+1$, and for any integer $n$ let us introduce the stopping time
$$
\tau_{n}= \inf\{t \ge 0; \langle Z_{t}, h\rangle\ge n\}.
$$
We will prove below that there exists a constant $C_{2}$, independent of $n$, such that for any $T>0$,
\begin{equation}
\label{estimate}
\Ex{\sup\limits_{t\le T\wedge \tau_{n}}\, \langle Z_{t}, h\rangle} \le C_{2}\,h(x)\, e^{C_{2}T}.
\end{equation}
 It is then classical to prove that the sequence $(\tau_{n})_{n}$ tends almost surely to infinity as $n$ tends to infinity  
 (see for example M\'el\'eard-Fournier \cite{fournier}) and by Fatou's Lemma, that
 \begin{equation}
 \label{estimate1}
\Ex{\sup\limits_{t\le T} \, \langle Z_{t}, h\rangle} \le C_{2}h(x)e^{C_{2}T}.
\end{equation}
Now, we prove HP5).  Recall that for any integer $m$, we consider the stopping time
$$
T_{m}= \inf\Big\{t \ge 0: \langle Z_{t}, 1_{\{[-m,m]^c\}}\rangle >0\Big\}.
$$
Let us first note that 
$\langle Z_{t}, h\rangle \ge \left \langle Z_{t}, 1_{\{[-m,m]^c\}}h\right\rangle \ge (1+m^2) \left\langle Z_{t}, 1_{\{[-m,m]^c\}}\right\rangle$.
Now, if $T_m<T$ then, we have $1\le  \sup\limits_{t<T\wedge T_m}\langle Z_t, 1_{\{[-m,m]^c}\}\rangle$, which implies that
$$
(1+m^2)\Px{ T_m<T} \le 
(1+m^2)\Ex{ \sup\limits_{t<T\wedge T_m}\langle Z_t, 1_{\{[-m,m]^c\}}\rangle} \le C_{2}h(x)e^{C_{2}T},
$$
which implies that $(T_{m})$ tends a.s. to infinity.

Let us now come back to the proof of \eqref{estimate}, which is very standard.
We assume $n$ large enough so that   $1+x^2<n$. Using  \eqref{eq:MasterZ} and the positivity of $h$, we get for $0\le s\le t$
\begin{equation}
\label{eq:masterZ g}
 \langle Z_{s\wedge  \tau_{n}},h\rangle \le h(x)+
\int_0^{s\wedge \tau_{n}} \langle Z_r,\mathscr{G}^{(e)}h\rangle dr+M^{(e)} _{h}(s\wedge  \tau_{n})
\hspace{0.2cm}+\int_{[0,s\wedge  \tau_{n}]\times \RR_+\times \NN} \ind_{1\le i\le N_{r-}} \,h(X^i_{r-})\,
\ind_{\theta \le b(X^i_{r-})}\,
Q(dr,d\theta,di).
\end{equation}

The predictable 
quadratic variation for the local martingale is 
\begin{equation}
\label{eq:bound corchete}
\left\langle M^{(e)} _{h}\right\rangle_{s\wedge \tau_{n}}= \pcun_{e\neq 3}\int_0^{s\wedge  \tau_{n}} \left\langle Z_r, (h')^2\right\rangle \, dr 
+\pcun_{e\neq 1} \int_0^{s\wedge  \tau_{n}} \left\langle Z_r, \int  (h( z)-h(\sbullet))^2\;R(\sbullet,dz)\right\rangle\, dr.
\end{equation}
Under the assumptions HA1) and HJ2) and since $z^2-x^2 = (z-x)^2+2x(z-x)$, we have for any $e\in \{1,2,3\}$ that
$$
\mathscr{G}^{(e)}h (x) \leq C^{(e)}\, h(x),\quad (h'(x))^2\le 4 h(x), \quad |h'(x)|\le 2 h(x),
$$
where $C^{(e)}$ is a positive constant. 
For any $T>0$, using Doob's inequality with exponent $p=2$, we obtain for $t\le T$, 
\begin{align*}
\Ex{\sup\limits_{s\le t\wedge \tau_{n}}\langle Z_{s},h\rangle} &\le h(x)+
C^{(e)}\, \Ex{\int_0^{t\wedge \tau_{n}}\langle Z_{s},h\rangle\,ds}+b^*\, \Ex{\int_0^{t\wedge \tau_{n}}\langle Z_{s},h \rangle\,ds}
+C_1 \Ex{\int_0^{t\wedge \tau_{n}} \langle Z_s, h \rangle \,ds}^{1/2}\\
&\le h(x) + C_1 +C_2 \Ex{\int_0^{t\wedge \tau_{n}}\langle Z_{s},h\rangle\,ds}
\le C_2 h(x) + C_2\int_0^t  \Ex{\sup\limits_{r\le s\wedge \tau_{n}}\langle Z_{r},h\rangle} \, dr,
\end{align*}
where $C_1=2\Big(1+\sup\limits_y \int (1+|z-y|^4) R(y,dz)\Big)^{1/2}$, 
and $C_2=3\max\{C^{(1)},C^{(2)},C^{(3)},b^*, C_1,1\}$.
Gronwall's Lemma allows us to conclude the proof of \eqref{estimate}.

Finally, we have proved  that $T_m$ tends a.s. to infinity  as $m\to \infty$.
\end{proof}

\subsection{Proof of Theorem \ref{the:A}}\label{preuve:the:A}

Here  $P^{(1)}_{t}$ denotes  the Feynman-Kac semigroup \eqref{FK}  defined for $t>0$ by 
$$
P^{(1)}_{t}f(x) =  \Ex{\langle Z_{t},f \rangle} = \EE_{x}\left(e^{\int_{0}^t V(X_{s})ds}\, f(X_{t})\right),
$$
for $(X_{t})$ introduced in Subsection \ref{21}.

 As seen in Corollary 3.4 in \cite{CMSM}, this semigroup coincides, for $f\in C^2_{c}(\RR)$, with a
 semigroup $\overline P_{t}$ defined on $L^2(\rho(dy))$, with $\rho(dy)= e^{-2\ell(y)} \, dy$.  This semigroup  is given by 
 \begin{equation}
 \label{serie}
\overline P_t(f)(x)=
\sum_{k\ge 0} e^{-\lambda_k t} \Theta_k(x) \int \Theta_k(y) f(y)\,  e^{-2\ell(y)} \, dy\;.
\end{equation}
Here $(\Theta_k)_k$ is an orthonormal basis of $L^2(\rho(dy))$, made up of eigenfunctions of each $\overline P_t$ for $t\ge 0$, 
with the corresponding set of eigenvalues $(e^{-\lambda_k t})_k$. We point out that $(\lambda_k)_k$ is a strictly
increasing sequence of real numbers converging to $+\infty$.

Using the Dominated Convergence Theorem, we deduce that for 
 $t>0$, the two semigroups  $P^{(1)}_{t}$ and $\overline P_t$  extend  and coincide in  $\CB$. 

\begin{proposition}\label{HP1-HP2-diff}
Under hypotheses (HV), (HD) and (HA),  the properties HP1) and HP2) hold.
\end{proposition} 
\begin{proof}
The proof of HP1) and HP2) follows at once from the hypotheses and from  Theorem 2.7
of \cite{CMSM}. Note that in Theorem 2.7-$(2)$  of \cite{CMSM},  Inequality  \eqref{eq:bound_uniforme} is
proved  for $t\ge t_0>0$, for
some $t_0>0$, 
and the inequality for $0\le t\le t_0$ follows from the continuity of
the operators. 
\end{proof}
We notice that the measure $\mu_0$ given in HP2) is $\mu_0(dy)=\Theta_0(y) \rho(dy)=\Theta_0(y) e^{-2\ell(y)} \, dy$.


\medskip 
\begin{proposition}\label{HP3-HP4-diff}
Under hypothesis (HV), (HD) and (HA),  the properties  HP3), HP4) and HQ) hold.
\end{proposition} 
\begin{proof}
The strong continuity of $\left(P_{t}^{(1)}\right)$ in $\CBZ$ follows from \cite{BL} Proposition 2.2.7 p.18 since 
 the hypotheses  Hyp. 2.0.1 in page 3 in \cite{BL} are satisfied in our case. 
Theorem 2.1 (iv) in \cite{CMSM} implies that $\left(P_{t}^{(1)}\right)_t$ maps $\MB$ to $\CBZ$. Indeed we show in particular in 
\cite{CMSM} Section 4.2 formulas (4.13) and (4.14), that $P^{(1)}_{t}1 \in \CBZ$. 
This also implies that $\int_{0}^t P^{(1)}_{s}\, ds$ 
maps $\MB$ to $\CB$.

Let us show that $\left(P_{t}^{(1)}\right)$ is irreducible (see Definition 3.1 (3) page 182  in \cite{AGGGLMNFS}). For that we take $f\in \CBZ^+$.
Let us first recall (as seen in \cite{CMSM}) that the function $\Theta_{0}$ appearing in \eqref{serie} is (strictly) positive on 
$\RR$ and that for any $t>0$, 
$$
P_{t}^{(1)}f(x) = \int_{\RR} f(y) \sum_{k} e^{-\lambda_{k}t}\Theta_{k}(x) \Theta_{k}(y) e^{-2\ell(y)}dy.
$$
 Let us fix $x\in \RR$. Using again the bounds obtained from \cite{CMSM} Section 4.2 formulas (4.13) and (4.14), we have
 $$P_{t}^{(1)}f(x) = e^{-\lambda_{0}t}\Theta_{0}(x)\int_{\RR} f(y)\Theta_{0}(y)  e^{-2\ell(y)}dy + R_{t}(x,f),$$
 where 
 $$|R_{t}(x,f)| \le C(x) \|f\|_{\infty} O(e^{-\lambda_{1}t}),$$
 $C(x)$ being a positive constant, from where irreducibility follows.
 
 \noindent For any $t>0$, the operator $P_{t}^{(1)}$  is compact in $\CBZ$ (see, for example, \cite{BL} Theorem 5.1.11 p. 84). 
 The semigroup $\left(Q_{t}^{(1)}\right)_t$ that is associated with $\left(P_{t}^{(1)}\right)_t$ 
 by the bounded perturbation $-2b$ of the generator 
 $\mathscr{L}^{(1)}$  is also compact (see, for example, \cite{Pazy} Theorem 1.1 p.76 ). Then it is quasi-compact.

\end{proof}

\subsection{Proof of Theorem \ref{the:B}}\label{preuve:the:B}
Here  $P^{(2)}_{t}$ denotes  the Feynman-Kac semigroup \eqref{FK}  defined for $t>0$ by 
$$
P^{(2)}_{t}f(x) =  \Ex{\langle Z_{t},f \rangle} = \EE_{x}\left(e^{\int_{0}^t V(X_{s})ds}\, f(X_{t})\right),
$$
for $(X_{t})$ introduced in Subsection \ref{22}.
Let us denote by $\mathscr{L}^{(2)}$ the associated generator. Recall that
$$\mathscr{L}^{(2)} f = \mathscr{L}^{(1)} f + L_{1}f,$$
where  $\mathscr{L}^{(1)}$ is the  generator associated with $P^{(1)}_{t}$ and that
$$L_{1}f(x) =  \int (f(z)- f(x))\;R(x,dz).$$

\medskip 
\begin{lemma}
\label{borneC0}
Under the  assumptions (HV) and (HA)
and the additional assumptions HJ1)-HJ2), the operator $L_{2}$ defined by
$$
L_{2}f(x) =  \int f(z)\;R(x,dz)
$$
is  bounded in $\CBZ$. It immediately follows that $L_{1}$ is  bounded in $\CBZ$.
\end{lemma}
\begin{proof}
Assume that $f\in \CBZ$. The continuity on $x$ of $L_2 f(x)$ follows from Hypothesis HJ1). In order to prove 
that $L_2 f \in \CBZ$, 
it is enough to do it under the extra assumption that $f$ is compactly supported. Denote by $K$ 
the support of $f$, then
$$
\left| \int f(z)\;R(x,dz)\right| =\left|  \int_{z\in K} \frac{f(z)}{1+ (x-z)^2}\;(1+ (x-z)^2)\;R(x,dz)\right| 
\leq \sup_{z\in K}  \frac{1}{1+  (x-z)^4}\; \|f\|_{\infty} \int (1+  (x-z)^4)\; R(y, dz), 
$$
which tends to $0$, when $x$ tends to infinity, using HJ2). 

The boundedness of $L_{1}$ follows  using the continuity  in $x$ of the total mass of the jump measure $R$.
\end{proof}

\medskip
\begin{proposition} 
\label{B4}
Under the hypotheses (HV), (HA), (HD) and (HJ), the  semigroup $\left(P^{(2)}_{t}\right)_t$ is 
a strongly continuous, irreducible, and compact semigroup in $\CBZ$. 
Its spectrum in $\CBZ$ contains a simple isolated real eigenvalue $-\lambda_{0}$, with a corresponding positive eigenfunction
$\Theta_0\in \CBZ$, namely for all $t\ge 0, x$,
$$
P^{(2)}_t(\Theta_0)(x)=e^{-\lambda_0 t}\Theta_0(x).
$$
The rest of the
spectrum is contained in a closed half plane $\Re z\le -\lambda_{1}$
for some $\lambda_{1}>\lambda_{0}$.  
\end{proposition}

\begin{proof}
The operator  $\mathscr{L}^{(2)} $  is a bounded perturbation of $\mathscr{L}^{(1)}$, hence it generates a strongly 
continuous semigroup in $\CBZ$ (see, for example,  \cite{Pazy} Theorem 1.1 p.76 ). 
The semigroup  $\left(P^{(2)}_{t}\right)_t$, being associated
to  $\left(P^{(1)}_{t}\right)_t$ by a bounded perturbation  of the generator $\mathscr{L}^{(1)}$, is also compact  (see
for example \cite{Pazy} Proposition 1.4 p.79). Therefore the eigenvalues of the generator  $\mathscr{L}^{(2)} $ with
maximal real part are in finite number and have 
finite multiplicity.

We now split the generator $\mathscr{L}^{(2)} $ in a different way. We have
$$
\mathscr{L}^{(2)} =\tilde{\mathscr{G}} + L_{2}
$$
where
$
\tilde{\mathscr{G}}=\frac12 f''-af'+\tilde Vf \;
$ and 
$
\tilde V=V-\int R(x, dz).
$

We denote by $(S_{t})$ the
 strongly continuous semigroup in $\CBZ$  with generator $\tilde{\mathscr{G}}$. 
Using  Duhamel's formula (see for example \cite{BA}),
$$
P^{(2)}_{\!t}f = S_{\!t} f +\int_{0}^t S_{\!t-u} \,L_{2} \,P^{(2)}_{\!u}f \;du\;,
$$
 we
conclude that if $f\in \CBZ$ is nonnegative, then for any $t\ge0$
$$
P^{(2)}_{\!t}f\ge S_{\!t}f\;.
$$

It is easy to verify that $\tilde{\mathscr{G}}$ satisfies the
assumptions of Theorem 2.1 in \cite{CMSM}, and we deduce asymptotic properties for $(S_{\!t})$. For this
semigroup, there exist  a strictly positive function $\tilde \Theta_{0}$
and a strictly positive measure $\tilde \mu_{0}$ such that for any given $f$ nonnegative (nonzero) and any $m$  positive
(nonzero measure) we have
$$
\lim_{t\to\infty}e^{\tilde \lambda_{0}t}\int S_{\!t}f\; dm=
\int \tilde \Theta_{0} \; dm \int d\tilde\mu_{0}\,f>0\;,
$$
which implies that for some $t_{0}>0$
$$
\int S_{\!t_{0}}f\; dm>0\;.
$$
Hence $(S_{\!t})$ is irreducible 
(see definition 3.1 (ii) page 182 in \cite{AGGGLMNFS} ). This implies 
that $(P^{(2)}_t)$ is irreducible. 

\noindent The spectral results follow from Corollary 2.2 page 210 in \cite{AGGGLMNFS}.

\end{proof}

\begin{proposition}\label{HP1-HP2-diff-sauts-comp-res}
Under hypotheses (HV), (HD), (HA) and (HJ),  the properties HP1) and HP2) hold.
\end{proposition} 
\begin{proof}
HP1) follows from Proposition \ref{B4}. The proof HP2) is as follows. HP2) holds in $\CBZ$ thanks to
 \cite{AGGGLMNFS} Corollary 2.2 p. 210. Now, any function $f$ in $\CB$ is the pointwise limit of a sequence of
 functions in $\CBZ$, which are uniformly bounded by $\|f\|_{\infty}$. Thus, the inequality in HP2) can be extended to 
 $\CB$ by the Dominated Convergence Theorem.
\end{proof}

\begin{proposition}\label{HP3-diff_sauts}
Under hypotheses (HV), (HD), (HA) and (HJ), the properties  HP3), HP4) and HQ) hold.
\end{proposition} 
\begin{proof}
Thanks to Proposition \ref{B4}, we only have to prove that 
for $t>0$, $P_{t}^{(2)}$ is a bounded operator from  $\MB$ to $\CBZ$.
\
Let us denote by $J$, the first jump time of the underlying process $(X_{s})$.
We have for each $f\in\MB$
$$
P_{t}f(x)=\EE_{x}\left(e^{\int_{0}^{t}V(X_{s})\,ds}
\pcun_{J> t}f\big(X_{t}\big)\right)
+\EE_{x}\left(e^{\int_{0}^{t}V(X_{s})\,ds}
\pcun_{J\le t}f\big(X_{t}\big)\right)\;.
$$
For the first term, we use a coupling argument to obtain 
$$
\EE_{x}\left(e^{\int_{0}^{t}V(X_{s})\,ds}
\pcun_{J> t}f\big(X_{t}\big)\right)=   \EE_{x}\left(e^{\int_{0}^{t}\left(V(X^{(1)}_{s})-\tilde R(X^{(1)}_{s})\right)\,ds}
f\big(X^{(1)}_{t}\big) \right) = \tilde P_{t}^{(1)} f(x),
$$
where $X^{(1)}$  is a
diffusion process  with generator $\mathscr{L}^{(1)}$ issued from $x$ and
$$
\tilde R(y)=\int R(y,\,dz).
$$
The results of \cite{CMSM}  allow to conclude that $\tilde P_{t}^{(1)}$ is a bounded operator from $\MB$ to $\CBZ$, since the function $V-\tilde R$ has the same properties as $V$.

For the second term we have by the strong Markov property and a coupling argument that 
$$
\EE_{x}\left(e^{\int_{0}^{t}V(X_{s})\,ds}
\pcun_{J\le t}f\big(X_{t}\big)\right)=
\EE_{x}\left(e^{\int_{0}^{J}V(X_{s})\,ds}\pcun_{J\le t}
\EE_{X_{J}}\left(e^{\int_{0}^{t-J}V(X_{s})\,ds} f\big(X_{t-J}\big)\right)\right)
$$
$$
=\int_{0}^{t}du\;
\EE_{x}\left(e^{\int_{0}^{u}(V(X^{(1)}_{\tau})-\tilde R(X^{(1)}_{\tau}))\,d\tau}
\int R(X^{(1)}_{u},dz)\;
\EE_{z}\left(e^{\int_{0}^{t-u}V(X_{s})\,ds} f\big(X_{t-u}\big)\right)\right)\;.
$$

For each $t>0$ and $0<u<t$ fixed, we have
$$
\EE_{x}\left(e^{\int_{0}^{u}(V(X^{(1)}_{\tau})-\tilde R(X^{(1)}_{\tau}))\,d\tau}
\int R(X^{(1)}_{u},dz)\;
\EE_{z}\left(e^{\int_{0}^{t-u}V(X_{s})\,ds}
  f\big(X_{t-u}\big)\right)\right)
=\tilde P_{u}^{(1)}
\int R(\,\sbullet\,,dz)\;
\EE_{z}\left(e^{\int_{0}^{t-u}V(X_{s})\,ds} f\big(X_{t-u}\big)\right)(x).
$$
Since  $\tilde P_{u}^{(1)}$ is a bounded
operator from $\MB$ to $\CBZ$, we obtain that
 for each $t>0$ and $0<u<t$ fixed, 
$$
\tilde P_{u}^{(1)}
\int R(\,\sbullet\,,dz)\;
\EE_{z}\left(e^{\int_{0}^{t-u}V(X_{s})\,ds} f\big(X_{t-u}\big)\right)\in\CBZ\;.
$$

Therefore by Dominated Convergence Theorem, the function
$$
x\longrightarrow \int_{0}^{t}du\;\Ex{e^{\int_o^u (V(X_{\tau})-\tilde R(X_{\tau}))\,d\tau}
\int R(X_{u},dz)\;\EE_{\partial_z}\left(e^{\int_{0}^{t-u}V(X_{s})\,ds} f\big(X_{t-u}\big)\right)}\in \CBZ
$$
and its uniform norm is bounded by
$$
\left\|\int_{0}^{t}du\;
\EE_{\sbullet}\left(e^{\int_{0}^{u}(V(X_{\tau})-\tilde R(X_{\tau}))\,d\tau}
\int R(X_{u},dz)\;
\EE_{z}\left(e^{\int_{0}^{t-u}V(X_{s})\,ds}
  f\big(X_{t-u}\big)\right)\right)\right\|_{\infty}
$$
$$
\le 
t\, e^{t\,\sup_{y}V(y)}\;\|f\|_{\infty}\;\sup_{x}\int R(x,dz)\;.
$$

 \noindent For any $t>0$, the operator $P_{t}^{(2)}$  is compact in $\CBZ$ (as seen in proof of Proposition 
 \ref{B4}). The semigroup $Q_{t}^{(2)}$ being associated to $P_{t}^{(2)}$  by the bounded perturbation  $-2b$ 
 of the generator $\mathscr{L}^{(2)}$  is also compact (  \cite{Pazy} Theorem 1.1 p.76 ). It is then quasi-compact.

\end{proof}

\subsection{Proof of Theorem \ref{the:C}}\label{preuve:the:C}

We denote by $P_{t}^{(3)}$ the semigroup $(P_{t})$ in this case.

Under  the assumptions (HV), (HD), (HJ) and HJ4),  the hypotheses HP1) and HP2) follow from 
Cloez Gabriel  \cite{cloezgabriel}. 


In order to prove HP3) and HP4), we need to introduce and study the semi group $(G_{t})$  on $\MB$ given by 
\begin{equation}
\label{def:G}
G_{t}f(x)=e^{\int_{0}^{t} V(x+\tau)\,d\tau}f(x+t)\;.
\end{equation}

\begin{lemma}\label{prop:G}
The semigroup $(G_{t})$ defined in \eqref{def:G} satisfies 
\begin{enumerate}[i)]
\item 
$(G_{t})$ is a semigroup of bounded operators in $\MB$, $\MBZ$,
$\CB$ and $\CBZ$. It maps nonnegative functions to nonnegative
functions.
\item For any $t>0$,  $G_{t}$ maps continuously  $\MB$ to $\MBZ$ and $\CB$ to
  $\CBZ$. 
\item In all these spaces, its asymptotic growth rate is negative.
\item $(G_{t})$ is a strongly continuous semigroup in $\CBZ$. 

\medskip The same results hold for the semigroup
\begin{equation}\label{def:G2}
\widetilde G_{t}f(x)=e^{\int_{0}^{t}(V(x+\tau)-\tilde R(x+\tau))\,d\tau}f(x+t)\;.
\end{equation}

\end{enumerate}

\end{lemma}
\begin{proof}
i) follows immediately from the definition and the upper bound on $V$.

For ii), if $0<t\le 1$ we have
$$
\int_{0}^{t}V(x+\tau)\,d\tau\le t \;\sup_{|y-x|\le 1}V(y),
$$
which tends to $-\infty$ if $|x|$ tends to $+\infty$. This implies that
$G_{t}$ maps continuously $\MB$ to $\MBZ$ and hence $\CB$ to
  $\CBZ$ for any $0<t\le 1$. 

For $t>1$, we write 
$$
G_{t}=G_{1}\,G_{t-1}\;.
$$
We will now estimate the asymptotic growth rate of $(G_{t})$ in $\MB$
(or in $\CB$).

Since $V$ tends to $-\infty$ for $x$  tending to $\pm\infty$,
there exists a compact set $K$ and a constant $A>0$ such that
$$
V(y)\le -1+A\,\pcun_{K}(y)\;.
$$
Therefore
$$
\int_{0}^{t}V(x+\tau)\,d\tau\le A\,|K|-t\;.
$$
This implies that the asymptotic growth rate of $(G_{t})$ in $\MB$
(or $\CB$) is at most $-1$ and iii) follows.

iv) follows from the divergence of $V$ at infinity and continuity of
$V$.

\end{proof}

\begin{proposition}\label{HP3-sauts}
Under hypotheses (HV), (HD), (HJ) and HJ4), the properties HP3), HP4) and HQ) hold.
\end{proposition} 

\begin{proof}

In $\CBZ$, the semi group $P_{t}^{(3)}$ is a perturbation of the
strongly continuous semigroup $(G_{t})$ by the bounded operator
$L_{1}$. Hence by  Theorem 1.1 page 76 in \cite{Pazy} it is also a
strongly continuous semigroup in $\CBZ$.

To prove the irreducibility, we observe that if $f$ is a nonnengative
nonzero function in $\CBZ$, there exist an interval $[x_{1},\,x_{2}]$ 
and a constant  $\delta>0$ such that $f(x) \ge \delta$ for any $x\in [x_{1},\,x_{2}]$. 
This implies that for any $t>0$,
$$
P_{t}^{(3)}f\ge \delta\; P_{t}^{(3)}\pcun_{[x_{1},\,x_{2}]}\;.
$$
 From Lemma 2.2 in \cite{cloezgabriel}, we deduce that for any $y_{1}, y_{2}$, there exists $\eta>0$ such that
 $$P_{t}^{(3)}\pcun_{[x_{1},\,x_{2}]} \ge \eta\pcun_{[y_{1},\,y_{2}]}\;.$$
  Irreducibility in $\CBZ$  follows from \cite{AGGGLMNFS} Definition 3.1 (ii) p.182. 

We have now to prove that for any $T>0$, $\int_{0}^T P^{(3)}_{s}ds$ maps $\MB$ into $\CB$.

We use the jump decomposition for the process
$(X_{s})$. Denoting as before by $J$ the  the first jump (stopping) time of $(X_{s})$.  We have for $f\in \MB$
$$
\int_{0}^{T} P^{(3)}_{t}f(x)\;dt
=
\int_{0}^{T}\EE_{x}\left(e^{\int_{0}^{t}V(X_{s})\,ds}
\pcun_{J\ge t}f\big(X_{t}\big)\right)\;dt
+\int_{0}^{T}\EE_{x}\left(e^{\int_{0}^{J}V(X_{s})\,ds} \pcun_{J< t}\EE_{X_{J}}\left(e^{\int_{0}^{t-J}V(X_{s})\,ds}
 f\big(X_{t-J}\big)\right)\right)\;dt\;.
$$

For the first term we have the explicit expression
$$
\int_{0}^{T}\EE_{x}\left(e^{\int_{0}^{t}V(X_{s})\,ds}
\pcun_{J\ge t}f\big(X_{t}\big)\right)\;dt
=\int_{0}^{T}e^{\int_{0}^{t}(V(x+s)-\tilde R(x+s))\,ds}f(x+t)\;dt
=
\int_{x}^{x+T}e^{\int_{x}^{u}(V(\tau)-\tilde R(\tau))\,d\tau}f(u)\;du,
$$
which is obviously continuous (even uniformly Lipschitz continuous) in $x$.

We also have
$$
\left|\int_{0}^{T}e^{\int_{0}^{t}(V(x+s)-\tilde R(x+s))\,ds}f(x+t)\;dt \right| \le \|f\|_{\infty}\int_{0}^{T}e^{\int_{0}^{t}(V(x+s)-\tilde
  R(x+s))\,ds}\;dt
\le T\; \|f\|_{\infty}\;e^{T\;\sup_{|y-x|\le T}V(y)},
$$
which tends to zero when $|x|$ tends to infinity.

For the second term we have by the strong Markov property 
$$ 
\EE_{x}\left(e^{\int_{0}^{t}V(X_{s})\,ds}
\pcun_{J< t}f\big(X_{t}\big)\right)=
\EE_{x}\left(e^{\int_{0}^{J}V(X_{s})\,ds}\pcun_{J< t}
\EE_{X_{J}}\left(e^{\int_{0}^{t-J}V(X_{s})\,ds} f\big(X_{t-J}\big)\right)\right)
$$
$$
=\int_{0}^{t}du\;e^{\int_{0}^{u}(V((x+s)-\tilde R(x+s))\,ds}
\int R(x+u,dz)\;
\EE_{z}\left(e^{\int_{0}^{t-u}V(X_{s})\,ds} f\big(X_{t-u}\big)\right)\;.
$$
Then,
\begin{align*}&\int_{0}^T dt\int_{0}^{t}du\;e^{\int_{0}^{u}(V((x+s)-\tilde R(x+s))\,ds}
\int R(x+u,dz)\;
\EE_{z}\left(e^{\int_{0}^{t-u}V(X_{s})\,ds}
  f\big(X_{t-u}\big)\right)\\
= &\int_{{D(x,T)}} dy d\tau \, e^{\int_{0}^{y-x} (V(x+s)-\tilde R(x+s))ds} \int R(y,dz)\;
\EE_{z}\left(e^{\int_{0}^{\tau}V(X_{s})\,ds} f\big(X_{\tau}\big)\right)\;,
\end{align*}
by the change of variable $u=y-x$ and $t = \tau +u$. The set $D(x,T)$ is a triangle 
whose vertices are given by the coordinates  $(x,0)$, $(x+T,0)$, $(x+T, -T)$. This implies the continuity in $x$. 

To prove HP4) it remains to prove that the term above tends to $0$ as $x$ tends to $\pm \infty$. It is immediate (thanks to HJ2) 
to observe that $$e^{\int_{0}^{u}(V((x+s)-\tilde R(x+s))\,ds}
\int R(x+u,dz)\;
\EE_{z}\left(e^{\int_{0}^{t-u}V(X_{s})\,ds}
  f\big(X_{t-u}\big)\right)$$ tends to $0$ when $x$ tends to infinity, and we apply the Dominated Convergence Theorem.
  
  The proof of HQ) is immediately deduced from Theorem 2.1 in Cloez-Gabriel \cite{cloezgabriel}.

\end{proof}

{\bf Acknowledgments :}
We thank Jean Bertoin for suggesting us to prove the almost convergence in the supercritical case.
This work has been supported by the Chair Mod\'elisation Math\'ematique et Biodiversit\'e of Veolia
Environnement-Ecole Polytechnique-Museum National d'Histoire Naturelle-Fondation X and by the European
Union (ERC, SINGER, 101054787). Views and opinions expressed are however those of the authors only
and do not necessarily reflect those of the European Union or the European Research Council. Neither the
European Union nor the granting authority can be held responsible for them. 
J. San Martin is thankful for the hospitality of CMAP-Ecole polytechnique and S. M\'el\'eard acknowledges the hospitality of CMM-Universidad de Chile. 

\end{document}